\documentclass[11pt]{amsart}
\usepackage{amsmath,amscd,amssymb}
\usepackage[mathscr]{eucal}
\usepackage{amsmath}
\usepackage{enumerate}
\usepackage{color}
\usepackage{tikz}
\usepackage{xcolor}
\usepackage{enumitem}
\usepackage{xcolor}
\usepackage{comment}
\usepackage{bm}

\usepackage{xcolor}

\numberwithin{equation}{section}
\newtheorem{theorem}{Theorem}[section]
\newtheorem{lemma}[theorem]{Lemma}
\newtheorem{definition}[theorem]{Definition}
\newtheorem{proposition}[theorem]{Proposition}
\newtheorem{corollary}[theorem]{Corollary}
\newtheorem{remark}[theorem]{Remark}
\theoremstyle{definition}
\newtheorem{example}[theorem]{Example}
\theoremstyle{plain}

\def\C{\mathbb{C}}
\def\N{\mathbb{N}}

\def\T{\mathbb{T}}

\def\L{L}

\def\D{\mathbb{D}}

\def\dist{\mathrm{dist}}

\def\length{\operatorname{length}}

\newcommand{\HS}{\mathrm{HS}}

\renewcommand{\L}{\mathcal L}

\title[Beurling--Kato theory]{Beurling--Kato theory, Hardy--Sobolev calculus and Ritt operators}
\author{A.~Borichev}
\address{Aix-Marseille University, CNRS, I2M\\
Marseille, France}
\email{alexander.borichev@math.cnrs.fr}
\author{A.~Gomilko}
\address{Faculty of Mathematics and Computer Science\\
Nicolaus Copernicus University\\
Chopin Street 12/18\\
87-100 Toru\'n, Poland}
\email{alex@gomilko.com}
\author{Yu.~Tomilov}
\address{
Institute of Mathematics\\
Polish Academy of Sciences\\
\'Sniadeckich 8\\
00-656 Warsaw, Poland\\
and\\
Faculty of Mathematics and Computer Science\\
Nicolaus Copernicus University\\
Chopin Street 12/18\\
87-100 Toru\'n, Poland
}
\email{ytomilov@impan.pl}

\thanks{The first author was supported by the grant  ANR-24-CE40-5470. The second and the third authors  were partially supported by the NAWA/NSF grant BPN/NSF/2023/1/00001. They were also partially supported by the NCN grant Weave-Unisono, 2024/06/Y/ST1/00044 and by the NCN grant UMO-2023/49/B/ST1/01961.}

\begin{document}
\begin{abstract}
We develop a discrete Beurling--Kato theory for bo\-und\-ed operators and relate it to a Hardy--Sobolev functional calculus on Stolz domains. Our central class is that of Ritt operators. We prove that a bounded operator is Ritt if and only if it admits a bounded Hardy--Sobolev calculus on a Stolz domain. The construction is based on a logarithmic reproducing formula and uniform bounds for the associated logarithmic kernels.
We then derive Kato-type results and characterise the Ritt property by Beurling--Kato defects formulated in terms of powers of the operator. This yields a discrete theory parallel in spirit to the sectorially bounded holomorphic semigroup setting, but intrinsically global in nature.
We also discuss examples and sharpness phenomena, and prove that the
Ritt property is preserved under convex combinations of powers and
under positive domination on Banach lattices.
\end{abstract}

\maketitle

\setcounter{tocdepth}{1}
\tableofcontents

\section{Introduction}\label{sec:introduction}

The Beurling--Kato theory reveals a striking rigidity phenomenon: for 
a $C_0$-semigroup $(e^{-tA})_{t \ge 0}$ on a Banach space $X,$ its holomorphicity may be encoded in a norm gap. One does
not need to inspect the generator $-A$ directly. It is enough to control suitable
polynomial expressions in $e^{-tA}$ as $t\downarrow0$. The familiar threshold $2$
in the estimate for $I-e^{-tA}$ near zero is only the simplest manifestation of a broader
function-theoretical principle.

Recall that a classical Beurling--Kato theorem says that
\begin{equation}\label{beurling}
  \limsup_{t\downarrow0}\|I-e^{-tA}\|<2
\end{equation}
forces holomorphicity of $(e^{-tA})_{t \ge 0},$ see e.g. the original papers \cite{Beurling} and \cite{Kato}. Its polynomial form is more revealing and stems from \cite{Neuberger}. In the form
emphasized by Fackler, see e.g. \cite{Fackler}, the relevant comparison is with the disc norm of a test polynomial $w$:
if there exists a polynomial $w$ such that
\begin{equation}\label{fackler}
  \limsup_{t\downarrow0}\|w(e^{-tA})\|<\|w\|_\infty:=\sup\{|w(z)|:\ |z|=1\},
\end{equation}
then  $(e^{-tA})_{t \ge 0}$ is holomorphic.

See \cite{Beurling,Kato,Fackler,Fackler-Thesis} for the classical zero--two criterion,
Beurling's polynomial theorem, and a modern reformulation. Related function-theoretical questions
were developed by Neuberger, see e.g. \cite{Neuberger1}, who in fact was behind the very first insight on norm gaps \cite{Neuberger}. See a nice historical discussion in \cite{NeubergerIn}.
Related work in a similar spirit can be found in \cite{Esterle}, but we omit its discussion here.

 One way to refine the zero--two norm-gap condition \eqref{beurling} is to consider
its power-averaged version.
 If, for some integer
\(N\geq 1\),
\begin{equation}\label{direct}
        \limsup_{t\downarrow 0}
        \bigl\|(I-e^{-tA})^N\bigr\|^{1/N}<2,
\end{equation}
then \((e^{-tA})_{t\geq 0}\) is holomorphic.
Conversely, if \((e^{-tA})_{t\geq 0}\) is holomorphic, then there exists
\(N_0\in\mathbb N\) such that for every \(N\geq N_0\) there is
\(k_0>0\) with
\begin{equation}\label{converse}
        \limsup_{t\downarrow 0}
        \bigl\|(I-e^{-tA})^N e^{-ktA}\bigr\|^{1/N}<2,
        \qquad k\geq k_0.
\end{equation}
Here \(e^{-ktA}\) is the regularizing factor corresponding to
\(z^k\). Thus the polynomial occurring in the displayed formula is
\(z^k(1-z)^N\).  More generally, for a polynomial \(w\) satisfying
\(|w(1)|<\|w\|_\infty\), the converse to \eqref{fackler} involves test polynomials
of the form \(z^kw(z)^N\).
See e.g. \cite[Theorems 2.2 and 5.4]{Fackler} for the proofs.
Fackler attributes the converse statement to Beurling \cite{Beurling},
although the formulation and proof in the original paper are rather indirect.

The additional factor \(z^k\) in the converse was essential in the
general Banach-space formulation, and it was removed only very recently,
\cite{Borichev1}.
This variant has the advantage of being invariant under equivalent
renormings, and it is, in general, finer than the original zero--two
condition \eqref{beurling}.
As a manifestation of its utility, one may recall Pisier's breakthrough proof of equivalence of K- and B-convexities in \cite{Pisier},
where the asymptotic variant played a decisive role.

If the gap condition \eqref{direct} holds (in particular, if the
stronger condition \eqref{beurling} holds) 
with $\limsup_{t \downarrow 0}$ replaced by $\sup_{t>0},$
and the semigroup is bounded,
 then the semigroup is
moreover sectorially bounded holomorphic, and its negative generator is a
sectorial operator of angle smaller than $\pi/2$. This point is already implicit in the arguments of
Beurling and Kato as well as in  Fackler's work.
However, it has not been elaborated there.

Thus,
in continuous time, a global norm defect at the identity detects global holomorphic
structure. This phenomenon is most naturally stated not for the particular
function $1-z$, but for general polynomial or holomorphic test functions. That
is the form in which the principle admits extensions, and it is also the form
that suggests a discrete counterpart.

This paper asks what remains of this principle in discrete time. Replacing a
semigroup $(e^{-tA})_{t\ge0}$ by the powers $(T^n)_{n\ge0}$ of a single bounded
operator changes the problem in an essential way. There is no small-time
regime for a single operator, and the relevant estimates must be formulated in
terms of the whole family of powers. 
The appropriate regularity class is the class of Ritt operators.  These are
bounded operators $T$ on $X$ whose spectrum is contained in \(\overline{\mathbb D}\)
and which satisfy the resolvent bound
\[
        \sup_{|z|>1}|z-1|\|(z-T)^{-1}\|<\infty,
\]
which, in particular, forces \(\sigma(T)\cap\mathbb T\subset\{1\}\).
 This class is fundamental in operator theory, and it has numerous
connections to harmonic analysis, ergodic theory, numerical analysis, and discrete maximal regularity. 
See, for example, \cite{Badea-Seifert,Blunck1,Blunck2,Cohen-Cuny-Lin, Coulhon-SaloffCoste,Cuny,ElFallah-Ransford,Hults-ReinholdLarsson,Kalton-Portal,Nevanlinna}.
A comprehensive account of the theory of Ritt operators and its
applications can be found in \cite{LeMerdy}.

 
Ritt operators are a discrete-time counterpart of 
sectorially bounded holomorphic \(C_0\)-semigroups
and many statements in both areas have parallel versions.
In this paper the analogy is
made intrinsic by replacing sectors by Stolz domains and by developing a
Hardy--Sobolev calculus adapted to those domains.

The transition from small times to powers has substantive consequences. In continuous
time one can test the semigroup at arbitrarily small times. In discrete time the family
$(T^n)_{n\ge0}$ is the natural object, and the estimates are global in the
power parameter. This leads to phenomena with no direct small-time analogue. First, the factor
$1-z$, which corresponds to the distinguished boundary point $1$, remains part
of the test function and cannot in general be removed from a defect estimate as
a separate normalization. Second, for a test polynomial $w$ with
$|w(1)|<\|w\|_\infty$, the relevant comparison with boundary values of
$w$ takes place on arcs of $\mathbb T\setminus\{1\}$
 on which $|w|$ is separated from
its value at $1$. These arcs may be chosen after $w$ is fixed.

To capture finer properties of Ritt operators, we construct an associated bounded
Hardy--Sobolev functional calculus.
We introduce Hardy--Sobolev algebras on Stolz domains and prove that a Banach space operator 
has  a bounded Hardy--Sobolev calculus on a
Stolz domain if and only if it is Ritt, see
Theorem~\ref{thm:Ritt-iff-HS}. The proof is based on a logarithmic reproducing
formula whose kernels are chosen using cuts adapted to the Stolz geometry. This
calculus is therefore not merely a consequence of the Ritt condition. It is the
analytic framework in which both the necessity and sufficiency parts of the
discrete Beurling--Kato theory are formulated.

Given a polynomial \(w\), or more generally a function from a weighted
Wiener algebra, we use the exponential defect to measure, on the
exponential scale, how closely the family \(w(T^n)\) can approach the
scalar bound \(\|w\|_\infty\):
\[
  \gamma(w,T)
  :=
  \lim_{N\to\infty}
  \Bigl(\sup_{n\ge1}\|w(T^n)^N\|\Bigr)^{1/N}.
\]
Equivalently, \(\gamma(w,T)\) is the spectral radius-type quantity
associated with the submultiplicative sequence
\[
  a_N(w,T):=\sup_{n\ge1}\|w(T^n)^N\|,
  \qquad N\in\mathbb N.
\]


The guiding principle is the following: a strict exponential defect for a
suitable test function $w$ should force $T$ to be Ritt, and conversely Ritt
operators should satisfy such strict defects for large classes of $w$. The
defect is inherently global in the power parameter and involves the whole positive-power family
$(T^n)_{n\ge1}$ rather than a single-step quantity. In simplified form, our
main results may be summarized as follows.

\begin{itemize}[leftmargin=2.2em]
\item[(i)] A bounded operator $T$ is Ritt if and only if it admits a bounded
Hardy--Sobolev calculus on a Stolz domain.

\item[(ii)] If $T$ is power bounded and $\sigma(T)\cap\T\subset\{1\}$, then $T$ is
Ritt if and only if, for some $\zeta\in\T\setminus\{1\}$,
\[
   \sup_{n\ge1}\|(\zeta -T^n)^{-1}\|<\infty.
\]
Moreover, for Ritt operators this condition holds for every
$\zeta\in\T\setminus\{1\}$.

\item[(iii)] If $T$ is power bounded, $\sigma(T)\cap\T\subset\{1\}$, and a polynomial 
$w$ has a strict defect on an arc
$I\subset\T\setminus\{1\}$, then $T$ is Ritt. In particular, the conclusion
follows under hypotheses of the form
\[
   \sup_{n\ge1}\|w(T^n)\|<\|w\|_\infty,
\]
or under the corresponding exponential defect condition
\[
   \gamma(w,T)<\|w\|_\infty.
\]

\item[(iv)] Conversely, if $T$ is Ritt and a polynomial $w$ satisfies $|w(1)|<\|w\|_\infty$, then
\[
   \gamma(w,T)<\|w\|_\infty.
\]
Moreover, the estimates are quantitative and uniform over natural classes of
test functions.

\item[(v)] The theory yields several sharpness statements and
permanence results for convex combinations of powers and for positive
domination on Banach lattices, illustrating the utility of the
Hardy--Sobolev calculus beyond the basic characterization of Ritt
operators.
\end{itemize}
In particular, for a bounded linear operator $T$ on  $X$ with $\sigma(T)\subset \mathbb D \cup\{1\}$
the uniform bound
\[
\sup_{n \ge 1}\|I-T^n\|<2
\]
implies that $T$ is Ritt.

Thus the paper develops a discrete Beurling--Kato theory parallel  to the one
in the theory of holomorphic semigroups. However, it is necessarily global in the power parameter, and thus it  
corresponds naturally to sectorially bounded holomorphic semigroups. Most of the arguments employed in the paper 
were motivated by a number of new problems encountered on the way of constructing the discrete Beurling--Kato theory.

The paper is organized as follows. Section~\ref{sec:preliminaries} recalls the
operator-theoretic background on Ritt operators, Stolz domains, and their
relation to the continuous-time sectorially bounded holomorphic semigroup picture. This is
the point where we fix the standing conventions and collect the geometric facts
used later. Section~\ref{sec:HS-spaces} develops the theory of Hardy--Sobolev spaces on Stolz domains
and deduces the corresponding reproducing formula.  Section~\ref{hcalculus} turns the scalar theory into the bounded Hardy--Sobolev calculus for Ritt operators, and shows
 that in fact the existence of such a calculus is equivalent to the Ritt property. This
provides the intrinsic analytic framework for the rest of the paper. The
forward Beurling--Kato mechanism is developed in
Section~\ref{sec:BK-sufficiency}. We first obtain a discrete analogue of Kato's resolvent
criterion deducing the Ritt
property by a geometric hitting argument together with a propagation
theorem for resolvent bounds. A similar technique leads 
to conditions for the Ritt property in terms  of ``defects'' for norms of test functions
evaluated at the positive powers
$(T^n)_{n \ge 1}.$ These results are illustrated by examples in Section~\ref{examples-sec}, where
we show in particular the limitations of the obtained statements.
In Section~\ref{sec:HS-back-powers} we establish converse results by means of the Hardy--Sobolev calculus constructed in Section~\ref{hcalculus}.
The calculus is used to derive the necessity side of the discrete Kato criterion,
quantitative necessity statements for defects, and the resulting ``if and only
if'' formulations that connect the analytic and power-based sides of the
theory. The final sections record permanence results for convex combinations of powers of Ritt operators (Section \ref{S8x}) 
and for domination by Ritt operators on Banach lattices (Section \ref{sec:positive-domination-ritt}).

A companion continuous-time paper \cite{Borichev1} will develop the corresponding glo\-bal
functional-calculus version of the Beurling--Kato--Fackler theory for
$C_0$-semigroups. The present paper should be read as its discrete counterpart: local
small-time estimates are replaced by global power estimates, sectorial domains
by Stolz domains, and the usual calculus for sectorially bounded holomorphic semigroups by the
Hardy--Sobolev calculus constructed below.

\section{Preliminaries on Ritt operators and Stolz domains}\label{sec:preliminaries}

This section collects the operator-theoretic background used thro\-ughout the
paper. In particular, it explains why Stolz domains are the natural geometric
objects in the discrete theory, and why the discrete analogue of the continuous
Kato--Beurling picture is necessarily global in the family of powers $(T^n)_{n \ge 0}$.

\subsection{Notation and standing conventions}\label{subsec:notation}

Throughout the paper,  we denote by $X$ a complex Banach space and by $\mathcal L(X)$
the space of bounded linear operators on $X.$
We let
$\sigma(T)$ and $\rho(T)$ stand for the spectrum and the resolvent set of an operator $T\in\L(X)$, respectively. 
We write $(T^n)_{n\ge0}$ for the full family of powers of an operator
$T$, including $T^0=I$.  When only positive powers are involved, as in
the power-resolvent and defect quantities below, the index starts at
$n=1$.

We denote by $\D$ and $\T$ the open unit disc and the unit circle, and by $\C$  the complex plane. Similarly, $D(z, r)$ will stand for the disc in $\mathbb C$ with center at $z$ and radius $r.$
When a single Stolz domain is involved, its parameter is usually
denoted by \(\sigma\). For nested Stolz domains, we use
\(
        1<\sigma<\delta<\tau,
\)
with \(\sigma\) and \(\tau\) denoting the inner and outer parameters,
respectively.
When a contour for the Riesz--Dunford functional calculus is needed,
we denote it by \(\Gamma\), with indices or superscripts when necessary.
For a rectifiable curve \(\Gamma\), we denote its arc length by
\(\operatorname{length}(\Gamma)\). The distance between bounded subsets $A$ and $B$ of $\mathbb C$ is denoted by ${\rm dist}(A, B).$
Spectral inclusions are formulated with the closure of the Stolz domain \(\overline{S_\sigma}\)
whenever boundary points may occur.

We shall work with the spaces of holomorphic functions ${\rm Hol}(\Omega)$ defined on domains $\Omega \subset \mathbb C,$ and we denote by 
$H^\infty(\Omega)$ the space of functions from ${\rm Hol}(\Omega)$ bounded
on $\Omega$ with the sup-norm.

We also use the Wiener algebra and the weighted Wiener algebra on $\mathbb D:$
\[
  A^1(\D)=\Bigl\{f(z)=\sum_{n\ge0}c_nz^n:\ \sum_{n\ge0}|c_n|<\infty\Bigr\},
\]
and
\[
  A^{1,1}(\D)=\Bigl\{f(z)=\sum_{n\ge0}c_nz^n:\ \sum_{n\ge0}(n+1)|c_n|<\infty\Bigr\},
\]
with norms
\[
  \|f\|_{A^1}:=\sum_{n\ge0}|c_n|,
  \qquad \text{and} \qquad
  \|f\|_{A^{1,1}}:=\sum_{n\ge0}(n+1)|c_n|,
\]
respectively. For a continuous function $w$ on $\overline{\mathbb D}$
we define
\[
  \|w\|_\infty:= \max_{|\zeta|=1} |w(\zeta)|,
\]
and for every nonempty closed arc \(I\subset\mathbb T\), we also set
\[
  m_I(w):=\min_{\eta\in I}|w(\eta)|.
\]

When no ambiguity is possible, limits of sequences are understood along the
index displayed in the notation. For instance, $a_k\to a$ means $a_k\to a$ as
$k\to\infty$.
\subsection{Ritt operators and operator powers}\label{poweritt}

Let $X$ be a complex Banach space and $T\in\L(X)$. We recall the standard definition of the class of Ritt operators.
\begin{definition}\label{def:Ritt-prelim}
We say that $T$ is a \emph{Ritt operator} if
\[
  \sigma(T)\subset \overline\D
\]
and
\begin{equation}\label{eq:Ritt-resolvent-prelim}
  \sup_{|z|>1}|z-1|\,\|(z-T)^{-1}\|<\infty.
\end{equation}
\end{definition}

An equivalent formulation states that $T$ is Ritt if and only if $T$ is power
bounded and the next discrete holomorphicity estimate holds:
\begin{equation}\label{analdiscr}
  \sup_{n\ge 1} n\,\|T^n-T^{n+1}\|<\infty,
\end{equation}
see e.g. \cite[Theorem 2.9]{LeMerdy}.
It seems to have been first obtained by Nevanlinna \cite{Nevanlinna}, \cite{Nevanlinna1}, although 
conditions of this form and their necessity for the Ritt property go back to \cite{Komatsu}. 
The estimate \eqref{analdiscr} is quite strong. Indeed, if
\(T\in\mathcal L(X)\), \(T\neq I\), and \(\sigma(T)=\{1\}\), or if
\(1\) is a limit point of \(\sigma(T)\), then
\[
        \limsup_{n\to\infty}
        n\,\|T^n-T^{n+1}\|
        \ge \frac1e;
\]
see \cite{Kalton-MontgomerySmith-Oleszkiewicz-Tomilov} and
\cite{Nevanlinna}, respectively. Nevertheless, as shown in
\cite{Kalton-MontgomerySmith-Oleszkiewicz-Tomilov},
\eqref{analdiscr} alone does not imply the Ritt property.
In the present paper we work with the resolvent form \eqref{eq:Ritt-resolvent-prelim},
because it is related directly to the Hardy--Sobolev calculus and the Stolz
geometry developed later on.

There are many straightforward ways to construct Ritt operators.  For instance, if $T$ is a contraction on $X$ and $0<\alpha<1$, then $I-(I-T)^\alpha$ is Ritt by Dungey's discrete subordination theorem \cite[Theorem 1.1]{Dungey}.  This useful fact was crucial, in particular, in \cite{Arhancet1}, where dilations of Ritt operators were studied. More general ways to produce Ritt operators out of contractions are discussed in \cite{GomTom-Indiana}.  Since contractions are easy to construct, these results provide a large supply of Ritt operators. 

\subsection{Stolz domains and spectral localization}

For $\sigma>1$ we write
\begin{equation}\label{eq:Stolz-Ssigma}
  S_\sigma:=\Bigl\{z\in\D:\ \frac{|1-z|}{1-|z|}<\sigma\Bigr\},
  \qquad \Gamma_\sigma:=\partial S_\sigma,
\end{equation}
where $\Gamma_\sigma$ is oriented positively. The boundary $\Gamma_\sigma$ is a
rectifiable convex Jordan curve which is piecewise $C^1$, consisting of two smooth arcs
meeting at the vertex $1$.
Note that in \cite{GomTom-Indiana}, $S_\sigma$  denoted the set in \eqref{eq:Stolz-Ssigma}
with adjoined $\{1\}.$ However, in this paper, we found it convenient to remove $\{1\}$
from the definition of $S_\sigma.$

The Ritt resolvent condition forces the spectrum to be contained in the closure 
\(\overline{S_\sigma}\) of a Stolz domain $S_\sigma$
 for a suitable \(\sigma>1\) and,
conversely, can be read off from resolvent control outside 
slightly larger Stolz
domains.
\begin{proposition}\label{prop:Stolz-localization-prelim}
Let $T\in\L(X)$. The following assertions are equivalent.
\begin{enumerate}
\item [(i)] $T$ is a Ritt operator.
\item [(ii)] There exists $\sigma>1$ such that $\sigma(T)\subset \overline{S_\sigma}$ and,
for every $\delta>\sigma$,
\begin{equation}\label{eq:off-Stolz-resolvent-prelim}
\sup_{z\in\C\setminus \overline{S_\delta}}
 |1-z|\,\|(z-T)^{-1}\|<\infty.
\end{equation}
\end{enumerate}
\end{proposition}

This equivalence is standard background. For the Stolz-domain localization and the original discrete setting we refer to \cite{Lyubich}. For a detailed discussion of the ratio-domain formulation used here  and its relation to sectorial geometry  see \cite{GomTom-Indiana}. For later developments related to functional calculi for Ritt operators consult, for example, \cite{LeMerdy}. 

The Stolz parameter in Proposition~\ref{prop:Stolz-localization-prelim} and in the Hardy--Sobolev calculus below is not merely a spectral enclosure. It records quantitative resolvent geometry.  Lyubich's one-point-spectrum examples show this sharply.  They exhibit  operators $T \in \L(X)$ satisfying Ritt's resolvent condition and such that 
\(
   \sigma(T)=\{1\},
\)
for which the maximal sector in the corresponding extended resolvent condition, together with the order of resolvent growth in the complementary sector, can be prescribed in advance, see \cite{Lyubich-single-point}.  Thus even the spectral equality $\sigma(T)=\{1\}$ does not determine the Stolz or sectorial geometry available for the calculus. That geometry is encoded in the resolvent bounds.


For \(\sigma>1\), we say that \(T\in\L(X)\) is of
\emph{Stolz type \(\sigma\)} if
\[
   \sigma(T)\subset\overline{S_\sigma}
\]
and \eqref{eq:off-Stolz-resolvent-prelim} holds for every
\(\delta>\sigma\).

\subsection{Continuous-time prototype and the discrete global re\-gime}

The continuous-time prototype is Kato's characterization of holomorphic $C_0$-semi\-groups. See \cite{Kato} and also \cite{Beurling,Fackler}. 
In one of its classical forms, a \(C_0\)-semigroup
\((e^{-tA})_{t\ge0}\) is holomorphic if and only if there exist
\(\zeta\in\mathbb T\setminus\{1\}\) and \(\delta>0\) such that
\(\zeta\in\rho(e^{-tA})\) for every \(0<t<\delta\) and
\begin{equation}\label{resest}
  \sup_{0<t<\delta}\|(\zeta-e^{-tA})^{-1}\|<\infty.
\end{equation}

 In particular, taking $\zeta=-1$ shows that the uniform small-time estimate
\[
  \|I-e^{-tA}\|\le 2-\varepsilon,\qquad 0<t<\delta,
\]
forces holomorphicity of \((e^{-tA})_{t\ge0}\).

If $(e^{-tA})_{t\ge0}$ is bounded and the 
 resolvent estimates \eqref{resest} are imposed uniformly for all $t>0$, 
then \((e^{-tA})_{t\ge0}\) is moreover sectorially bounded holomorphic.
Hence there is \(\theta\in[0,\pi/2)\) such that for every sector
\[
  \Sigma_{\theta'}:=\{z\in\mathbb C:|\arg z|<\theta'\}
\]
with \(\theta'\in(\theta,\pi/2)\), the negative generator \(A\)
satisfies
\begin{equation}\label{generat}
  \|(z-A)^{-1}\|
  \le C_{\theta'}/|z|,
  \qquad z\in\mathbb C\setminus\overline{\Sigma_{\theta'}},
\end{equation}
for some \(C_{\theta'}>0\).
Conversely,  \eqref{generat} implies that $-A$ generates a sectorially bounded holomorphic $C_0$-semigroup on $X,$ of angle $\pi/2-\theta.$
While this consequence does not seem to have been stated explicitly in \cite{Beurling,Kato}, it follows directly from the arguments there. See also \cite{Fackler}.

In discrete time there is no meaningful analogue of ``$t\downarrow0$'' for a single operator. However, the half-axis extension of the Beurling--Kato condition suggests that the only natural objects are the powers $(T^n)_{n\ge0}$ and conditions imposed globally on that family. Accordingly, the forward and reverse statements in the present paper are formulated in terms of arc-resolvent bounds for the powers, exponential defects, and Hardy--Sobolev calculus on Stolz domains.

The corresponding one-point resolvent statement is a discrete Kato criterion: under the standing peripheral spectral condition, a power-bounded operator is Ritt precisely when one resolvent point on $\T\setminus\{1\}$ is uniformly controlled for all powers.  This criterion is proved in Corollary~\ref{DisKato}; the direct implication belongs to the power estimates of Section~\ref{sec:BK-sufficiency}, while the converse uses the Hardy--Sobolev calculus.


A bridge between the continuous and discrete pictures is provided by the shifted operator $I-T$. The precise connection is that $T$ is Ritt if and only if 
$\sigma(T)\subset\D\cup\{1\}$ and $I-T$ is sectorial of angle $\theta\in [0,\pi/2).$
 The latter condition is equivalent to the fact that the semigroup $(e^{-t(I-T)})_{t\ge0}$ is sectorially bounded holomorphic of angle $\pi/2-\theta.$
This explains why the theories of sectorially bounded holomorphic semigroups and Ritt operators run in parallel. See, for example, \cite{Blunck1,Blunck2, Haase,Dungey,GomTom-Indiana}, and in particular, \cite[Chapter 2.2]{LeMerdy}.

For the purposes of the present paper, however, we keep the discrete side in the foreground: the geometry is encoded by Stolz domains, the functional calculus is built directly on such domains, and the link back to the semigroup picture is used only as motivation and comparison.

\subsection{Remarks on existing calculi and the present point of view}
Several bounded functional calculi for Ritt operators are already available in the literature. See, for instance, \cite{Arhancet-LeMerdy,Lancien,Nikolski,Schwenninger,Vitse-JFA,Vitse-Arch}. In particular, in \cite{Vitse-JFA} and \cite{Vitse-Arch}, Vitse
constructed bounded algebra calculi
based, respectively, on a multiplier algebra associated with
Cauchy--Stieltjes integrals and on a Besov-type algebra on the unit disc.
The norm of the former calculus is defined rather indirectly and is
difficult to estimate for concrete functions. The Besov calculus provides
more tractable function-norm estimates, but it does not reflect finer
spectral information about the operator, such as the location of its
spectrum in a particular Stolz domain and the corresponding resolvent
bounds.

An earlier related Cauchy-type construction on a weighted Smirnov class
was given by Solomyak \cite[Propositions~1.1 and~1.3]{Solomyak}. Under the Ritt
resolvent estimate, the corresponding class is essentially
\((1-z)H^1(\mathbb D)\), where $H^1(\mathbb D)$ stands for the Hardy space on $\mathbb D$. This class is generally not an algebra and does
not contain the Wiener algebra \(A^1(\mathbb D)\), 
which plays a basic role here. The estimates obtained
in this way are also weaker for the present purposes. For example, for
\(f_n(z)=1-z^n\), the bound for $I-T^n$ furnished by \cite{Solomyak} is of order
\(\log(n+1)\), whereas the estimate obtained below via a Hardy-Sobolev calculus yields
\(\sup_{n \ge 0}\|I-T^n\|\le C,\) with a constant $C$ depending only on $T$.

The estimates needed here involve precisely this finer Stolz geometry and
have to be compatible with power-based Beurling--Kato defects. We
therefore construct a Hardy--Sobolev calculus directly on Stolz domains
and use it to connect the analytic theory with the power-based conditions.

A continuous analogue of this theory was developed in \cite{BGT-Jussieu} and \cite{BGT-LMS},
with the latter paper revealing substantial applications to rational approximation theory.
While the approach in \cite{BGT-Jussieu} was also based on the reproducing ``arccot'' formula,
the corresponding details are much less demanding because of the simpler geometry 
and the more accessible half-plane Hardy–Sobolev algebras.

\section{Hardy--Sobolev spaces on Stolz domains}\label{sec:HS-spaces}

\subsection{Basic definitions}

Throughout this subsection we fix a Stolz domain \(S_\alpha\),
\(\alpha>1\), using the notation from \eqref{eq:Stolz-Ssigma}.  We write
\(E^1(S_\alpha)\) for the Hardy--Smirnov class on the 
domain \(S_\alpha\).  Equivalently, if \(\phi:\D\to S_\alpha\)
 is a
conformal map, then
\begin{equation}\label{defe}
        g\in E^1(S_\alpha)
        \quad\Longleftrightarrow\quad
        (g\circ\phi)\phi'\in H^1(\D).
\end{equation}
It is well-known that every $g \in E^1(S_\alpha)$ admits
the non-tangential boundary value on $\Gamma_\alpha,$ defined a.e.\ with respect to arc-length
measure on $\Gamma_\alpha.$ It will be denoted by the same symbol $g$ in the sequel.
Moreover, \eqref{defe} implies
\begin{equation}\label{ezero}
\int_{\Gamma_\alpha} g(\zeta)\, d\zeta=0, 
\end{equation}
and $g$ can be recovered from the boundary value by the standard Cauchy formula
\begin{equation}\label{cauchy}
g(z)=\frac{1}{2\pi i} \int_{\Gamma_\alpha}\frac{g(\zeta)}{\zeta-z}\, d\zeta, \qquad z \in S_\alpha.
\end{equation}
From \eqref{defe} it also follows that $E^1(S_\alpha)$ is a Banach space with the
norm given by
\[
        \|g\|_{E^1(S_\alpha)}
        :=
        \int_{\Gamma_\alpha}|g(\zeta)|\,|d\zeta|.
\]

We shall also use the following standard characterization of \(E^1\).  A
holomorphic function \(g\) belongs to \(E^1(S_\alpha)\) if and only if there
exists a sequence of rectifiable Jordan curves \((C_n)_{n \ge 1}\subset S_\alpha\), tending
to the boundary in the sense that \((C_n)_{n \ge 1}\) eventually surrounds every compact
subdomain of \(S_\alpha\), such that
\[
        \sup_{n\ge 1}\int_{C_n}|g(\zeta)|\,|d\zeta|<\infty.
\]
See, for example, \cite[Chapter~10]{Duren}.
Moreover, in this characterization one has
\begin{equation}\label{bounde}
        \int_{\Gamma_\alpha}|g(\zeta)|\,|d\zeta|
        \le
        \sup_{n\ge1}\int_{C_n}|g(\zeta)|\,|d\zeta|.
\end{equation}
Arguing as in \cite[Chapter 10]{Duren}, after passing to conformal
level curves, this estimate follows from Fatou's lemma.

The next inequality between integrals over nested convex curves
will be of fundamental importance.
\begin{lemma}\label{lem:Granados-convex}
There is a universal constant \(C_{\rm B}>0\) with the following
property.  Let \(\Gamma\) be a rectifiable Jordan curve bounding a convex
Jordan domain \(G\).  If \(g\) is holomorphic in \(G\) and has integrable
non-tangential boundary values on \(\Gamma\), and if \(\Lambda\) is a rectifiable convex
curve compactly contained in \(G\), then
\begin{equation}\label{eq:Granados-convex}
        \int_\Lambda |g(\zeta)|\,|d\zeta|
        \le
        C_{\rm B}
        \int_\Gamma |g(\zeta)|\,|d\zeta| .
\end{equation}
One may take, for instance, \(C_{\rm B}=4\).
\end{lemma}

The estimate is due to Beurling; in fact Beurling proved the stronger bound
\(C_{\rm B}<3.7\), see \cite[p.~456]{Beurling-CW1}.  A comparatively simple derivation
of \eqref{eq:Granados-convex} with $C_{\rm B}=4$ can be found in
\cite[p. 468]{Granados}.
Since a line segment is convex, \eqref{eq:Granados-convex} applies to 
line segments.

The next lemma relates  the usual
Hardy--Smirnov definition of the spaces $E^1(S_\alpha)$ given above to the definition of $E^1(S_\alpha)$ using inner boundaries.

\begin{lemma}\label{lem:touching-E1}
Let \(\alpha>1\), and let \(g\in \operatorname{Hol}(S_\alpha)\). Then
\[
        g\in E^1(S_\alpha)
        \quad\Longleftrightarrow\quad
        \sup_{1<\beta<\alpha}
        \int_{\Gamma_\beta}|g(\zeta)|\,|d\zeta|<\infty .
\]
Moreover, whenever these conditions hold,
\begin{equation}\label{eq:touching-E1-two-sided}
        \sup_{1<\beta<\alpha}
        \int_{\Gamma_\beta}|g(\zeta)|\,|d\zeta|
        \le
        C_{\rm B}
        \int_{\Gamma_\alpha}|g(\zeta)|\,|d\zeta|,
\end{equation}
and
\begin{equation}\label{eq:touching-E1-reverse}
        \int_{\Gamma_\alpha}|g(\zeta)|\,|d\zeta|
        \le
        C_{\rm B}
        \sup_{1<\beta<\alpha}
        \int_{\Gamma_\beta}|g(\zeta)|\,|d\zeta|,
\end{equation}
where \(C_{\rm B}\) is the constant from Lemma~\ref{lem:Granados-convex}.
\end{lemma}

\begin{proof}
Assume first that \(g\in E^1(S_\alpha)\).  Fix \(1<\beta<\alpha\).  For
small enough \(\eta>0\), let \(a_\eta^\pm\) be the two points of \(\Gamma_\beta\)
satisfying \(|a_\eta^\pm-1|=\eta\), one on each side of the vertex.  Let
\(\Gamma_{\beta,\eta}\) be the curve obtained from \(\Gamma_\beta\) by
removing
\[
        \Gamma_\beta\cap\{|\zeta-1|<\eta\}
\]
and replacing the removed part by the chord \([a_\eta^-,a_\eta^+]\).  This
curve bounds a convex domain and is compactly contained in \(S_\alpha\).  By
Lem\-ma~\ref{lem:Granados-convex}, applied with outer boundary
\(\Gamma_\alpha\),
\[
        \int_{\Gamma_{\beta,\eta}}|g(\zeta)|\,|d\zeta|
        \le
        C_{\rm B}
        \int_{\Gamma_\alpha}|g(\zeta)|\,|d\zeta| .
\]
Since
\[
        \Gamma_{\beta,\eta}\setminus [a_\eta^-,a_\eta^+]
        =
        \Gamma_\beta\cap\{|\zeta-1|\ge\eta\},
\]
we obtain
\[
        \int_{\Gamma_\beta\cap\{|\zeta-1|\ge\eta\}}
        |g(\zeta)|\,|d\zeta|
        \le
        C_{\rm B}
        \int_{\Gamma_\alpha}|g(\zeta)|\,|d\zeta| .
\]
Letting \(\eta\downarrow0\) and using monotone convergence on the fixed curve
\(\Gamma_\beta\), we get
\[
        \int_{\Gamma_\beta}|g(\zeta)|\,|d\zeta|
        \le
        C_{\rm B}
        \int_{\Gamma_\alpha}|g(\zeta)|\,|d\zeta| .
\]
Taking the supremum over \(\beta\) proves 
\eqref{eq:touching-E1-two-sided}.

Conversely, assume that
\[
        M:=\sup_{1<\beta<\alpha}
        \int_{\Gamma_\beta}|g(\zeta)|\,|d\zeta|<\infty .
\]
  
Choose a sequence
\(
   (\beta_n)_{n\ge1}\subset(1,\alpha)
\)
such that \(\beta_n\uparrow\alpha\), and put
\(\delta_n:=(\beta_n+\alpha)/2, n \in \mathbb N\). Choose also a sequence
\(
   (\eta_n)_{n\ge1}\subset(0,\infty)
\)
such that \(\eta_n\downarrow0\), and define
\((\Gamma_{\beta_n,\eta_n})_{n \ge 1}\) as above.
We choose \(\eta_n\) so that the curves
\((\Gamma_{\beta_n,\eta_n})_{n \ge 1}\) tend to \(\Gamma_\alpha\) in the sense that they
eventually surround every compact subdomain of \(S_\alpha\).  This is possible
because increasing \(\beta_n\) enlarges the Stolz domain, while decreasing
\(\eta_n\) removes a smaller neighbourhood of the common vertex.

Applying Lemma~\ref{lem:Granados-convex} to the convex curve $\Gamma_{\beta_n,\eta_n}$ with outer boundary
\(\Gamma_{\delta_n}\), gives
\[
        \sup_{n\ge 1} \int_{\Gamma_{\beta_n,\eta_n}}
        |g(\zeta)|\,|d\zeta|
        \le C_{\rm B}M<\infty .
\]
The exhaustion characterization of \(E^1(S_\alpha)\) gives
\(g\in E^1(S_\alpha)\). 

In view of \eqref{bounde},
\[
        \int_{\Gamma_\alpha}|g(\zeta)|\,|d\zeta|
        \le
        \sup_{n\ge 1} \int_{\Gamma_{\beta_n,\eta_n}}
        |g(\zeta)|\,|d\zeta|,
\]
which proves  \eqref{eq:touching-E1-reverse}.
The proof is complete.
\end{proof}



\begin{definition}\label{def:HS0}
Define
\[
\HS(S_\alpha):= \{f\in \operatorname{Hol}(S_\alpha): f'\in E^1(S_\alpha)\}.
\]
For \(f\in {\rm Hol} (S_\alpha)\)
 put
\[
        [f]_{\HS_0(S_\alpha)}
        :=
        \sup_{1<\beta<\alpha}
        \int_{\Gamma_\beta}|f'(\zeta)|\,|d\zeta|,
\]
with the convention that the value may be infinite, and define
\[
        \HS_0(S_\alpha)
        :=
 \{f\in \operatorname{Hol}(S_\alpha):[f]_{\HS_0(S_\alpha)}<\infty\}.
 \]
\end{definition}

Only the spaces \(\HS(S_\alpha)\) will be used in the sequel.  However,
it is useful to keep \(\HS_0(S_\alpha)\) in mind as a more explicit,
and in a sense natural,
alternative to \(\HS(S_\alpha)\), though less convenient in applications.

Applying Lemma~\ref{lem:touching-E1} to \(g=f'\), the next statement is immediate.
\begin{proposition}\label{prop:touching-E1-main}
Let \(\alpha>1\).  Then
\[
        \HS_0(S_\alpha)=\HS(S_\alpha)
\]
as sets.  Moreover, for every \(f\in\HS_0(S_\alpha)\),
\begin{equation}\label{eq:HS-E1-equivalence}
        \int_{\Gamma_\alpha}|f'(\zeta)|\,|d\zeta|
        \le
        C_{\rm B} [f]_{\HS_0(S_\alpha)},
        \qquad
        [f]_{\HS_0(S_\alpha)}
        \le
        C_{\rm B}
        \int_{\Gamma_\alpha}|f'(\zeta)|\,|d\zeta|.
\end{equation}

\end{proposition}


\begin{corollary}\label{cor:HS-absolute-continuity}
Let $f\in\HS(S_\alpha)$.  Then $f$ extends continuously to
$\overline{S_\alpha}$, and its restriction to every compact subarc of
$\Gamma_\alpha\setminus\{1\}$ is absolutely continuous.  Moreover,
\[
    \operatorname{Var}_{\Gamma_\alpha}(f)
    \le
    \int_{\Gamma_\alpha}|f'(\zeta)|\,|d\zeta|.
\]
\end{corollary}

\begin{proof}
By Proposition~\ref{prop:touching-E1-main}, we have
$f'\in E^1(S_\alpha)$.  The standard theorem on primitives of
$E^1$-functions on rectifiable Jordan domains therefore gives a
continuous extension of $f$ to $\overline{S_\alpha}$.
Setting
\(
    d:=\operatorname{length}(\Gamma_\alpha),
\)
let
\[
    \zeta:(0,d)\longrightarrow\Gamma_\alpha\setminus\{1\}
\]
be the positively oriented arclength parametrisation.  For every compact
interval $[a,b]\subset(0,d)$ and every $a\le u<v\le b$, one has
\[
    f(\zeta(v))-f(\zeta(u))
    =
    \int_u^v f'(\zeta(s))\zeta'(s)\,ds.
\]
Since
\[
    s\longmapsto f'(\zeta(s))\zeta'(s)
\]
belongs to $L^1(0,d)$, it follows that $f\circ\zeta$ is absolutely
continuous on every compact subinterval of $(0,d)$.  Hence $f$ is
absolutely continuous on every compact subarc of
$\Gamma_\alpha\setminus\{1\}$.

Moreover, for $0<a<b<d$,
\[
    \operatorname{Var}_{\zeta([a,b])}(f)
    \le
    \int_a^b |f'(\zeta(s))|\,ds.
\]
Letting $a\downarrow0$ and $b\uparrow d$, and using the continuity of
$f$ at the common endpoint $1$, we obtain
\[
    \operatorname{Var}_{\Gamma_\alpha}(f)
    \le
    \int_0^d |f'(\zeta(s))|\,ds
    =
    \int_{\Gamma_\alpha}|f'(\zeta)|\,|d\zeta|.
\]
\end{proof}

Next we equip \(\HS(S_\alpha)\)
with the norm
\begin{equation}\label{eq:HS-norm}
        \|f\|_{\HS(S_\alpha)}
        := \|f\|_{H^\infty(S_\alpha)}+
        \|f'\|_{E^1(S_\alpha)}=
        \sup_{z\in S_\alpha}|f(z)|
        +
        \int_{\Gamma_\alpha}|f'(\zeta)|\,|d\zeta|.
\end{equation}
We also equip $\HS_0(S_\alpha)$ with the norm
\[
        \|f\|_{\HS_0(S_\alpha)}
        :=
        \sup_{z\in S_\alpha}|f(z)|+[f]_{\HS_0(S_\alpha)}.
        \]

\begin{proposition}\label{prop:HS-Banach-algebra}
For every \(\alpha>1\), \(\HS(S_\alpha)\), equipped with the norm
\eqref{eq:HS-norm}, is a unital Banach algebra.
The norm on \(\HS_0(S_\alpha)\) is an equivalent Banach algebra norm,  with equivalence constants uniform in 
\(\alpha\).
\end{proposition}

\begin{proof}
Let \(f,g\in\HS(S_\alpha)\).
Then \(f',g'\in E^1(S_\alpha)\).  By
Corollary~\ref{cor:HS-absolute-continuity}, \(f\) and \(g\) extend
continuously to \(\overline{S_\alpha}\), so that they are bounded
on \(S_\alpha\).
 Therefore
\[
        (fg)'=f'g+fg'
\]
has summable boundary values on \(\Gamma_\alpha\), and
\[
\begin{aligned}
        \|fg\|_{\HS(S_\alpha)}
        &\le
        \sup_{S_\alpha}|f|\,\sup_{S_\alpha}|g|  \\
        &\quad+
        \sup_{S_\alpha}|f|
        \int_{\Gamma_\alpha}|g'(\zeta)|\,|d\zeta|
        +
        \sup_{S_\alpha}|g|
        \int_{\Gamma_\alpha}|f'(\zeta)|\,|d\zeta|       \\
        &\le
        \|f\|_{\HS(S_\alpha)}\,
        \|g\|_{\HS(S_\alpha)} .
\end{aligned}
\]
The constant function \(1\) is the unit.

If \((f_n)_{n \ge 1}\) is Cauchy in \(\HS(S_\alpha)\), then \((f_n)_{n \ge 1}\) converges
uniformly on \(S_\alpha\) to
 \(f \in H^\infty(S_\alpha)\), and
\((f_n')_{n \ge 1}\) is Cauchy in \(E^1(S_\alpha)\).  Since \(E^1(S_\alpha)\) is
complete, \(f_n'\to g\) in \(E^1(S_\alpha)\) for some \(g\in E^1(S_\alpha)\) as $n \to \infty$.


 Fix
\(z_0,z\in S_\alpha\) and join them by a rectifiable path
\(\gamma\subset S_\alpha\), so that 
\[
        f_n(z)-f_n(z_0)
        =
        \int_\gamma f_n'(\zeta)\,d\zeta .
\]
The convergence in \(E^1\) implies locally uniform convergence in
\(S_\alpha\). Thus
letting \(n\to\infty\), 
we obtain
\[
        f(z)-f(z_0)
        =
        \int_\gamma g(\zeta)\,d\zeta,
\]
so that \(f\) is holomorphic in \(S_\alpha\) and \(f'=g\).

Hence,
\(f\in\HS(S_\alpha)\), and \(f_n\to f\) in
\(\HS(S_\alpha)\) as $n \to \infty.$

The equivalence of the two norms, and hence the completeness of
\(\HS_0(S_\alpha)\), follows from
Proposition~\ref{prop:touching-E1-main}. The submultiplicativity of the norm
on \(\HS_0(S_\alpha)\) follows directly from the same argument as above,
with the boundary integrals replaced by the corresponding suprema over
\(1<\beta<\alpha\).
\end{proof}

The following boundary integral estimate for Stolz domains
will be crucial here and in Section~\ref{sec:HS-back-powers}
for obtaining $\HS$-norm bounds.

\begin{lemma}\label{lem:kernel-bound-Stolz}
Let \(\alpha>1\). Then there is a constant \(C_\alpha>0\) such that
\[
        \int_{\Gamma_\sigma} n|\zeta|^{n-1}\,|d\zeta|
        \le C_\alpha,
        \qquad 1<\sigma\le \alpha,\quad n\ge1 .
\]
\end{lemma}

\begin{proof}
By symmetry it is enough to estimate the integral along the upper half of \(\Gamma_\sigma\).
We write this arc in polar form as
\[
        \zeta=r_\sigma(\varphi)e^{i\varphi},
        \qquad 0\le \varphi\le \pi ,
\]
where \(r_\sigma\) is determined by the boundary identity
\[
        |1-r_\sigma(\varphi)e^{i\varphi}|
        =
        \sigma(1-r_\sigma(\varphi)).
\]
Equivalently, with \(b=\sigma^2-1\),
\begin{equation}\label{eq:Stolz-r-identity}
        2r_\sigma(\varphi)(1-\cos\varphi)
        =
        b(1-r_\sigma(\varphi))^2 .
\end{equation}
Thus \(r_\sigma\) is continuous on \([0,\pi]\), \(C^1\) on \((0,\pi]\), and
\[
        r_\sigma(0)=1,
        \qquad
        r_\sigma(\pi)=\frac{\sigma-1}{\sigma+1}.
\]
Differentiating \eqref{eq:Stolz-r-identity} on \((0,\pi]\) gives
\[
        r_\sigma'(\varphi)
        \bigl(1-\cos\varphi+b(1-r_\sigma(\varphi))\bigr)
        +r_\sigma(\varphi)\sin\varphi=0.
\]
Hence \(r_\sigma'(\varphi)<0\) for \(0<\varphi<\pi\), so \(r_\sigma\)
decreases on the upper half of \(\Gamma_\sigma\), 
and
\[
        r_\sigma(\varphi)\ge \frac{\sigma-1}{\sigma+1},
        \qquad 0\le\varphi\le\pi .
\]
Using that 
\[
        1-\cos\varphi\ge \frac{2}{\pi^2}\varphi^2,
        \qquad 0\le\varphi\le\pi ,
\]
in \eqref{eq:Stolz-r-identity}, we get
\[
        1-r_\sigma(\varphi)
        \ge
        \frac{2}{\pi(\alpha+1)}\,\varphi .
\]
Hence, with
\[
        a_\alpha:=\frac{2}{\pi(\alpha+1)},
\]
we have
\begin{equation}\label{eq:Stolz-r-exp-uniform}
        r_\sigma(\varphi)
        \le
        1-a_\alpha\varphi
        \le
        e^{-a_\alpha\varphi},
        \qquad 0\le\varphi\le\pi,\quad 1<\sigma\le\alpha .
\end{equation}

On $\Gamma_\sigma^+$, the upper half of \(\Gamma_\sigma\), we have 
\[
        |d\zeta|
        \le
        |dr_\sigma|+r_\sigma(\varphi)\,d\varphi .
\]
Therefore, for \(n\ge1\),
\[
\begin{aligned}
 \int_{\Gamma_\sigma^+} n|\zeta|^{n-1}\,|d\zeta|
 &\le
 \int_{\Gamma_\sigma^+} n r_\sigma^{\,n-1}\,|dr_\sigma|
 +
 \int_0^\pi n r_\sigma(\varphi)^n\,d\varphi .
\end{aligned}
\]
Since \(r_\sigma\) decreases from \(1\) to \((\sigma-1)/(\sigma+1)\) on $\Gamma_\sigma^+$,
\[
        \int_{\Gamma_\sigma^+} n r_\sigma^{\,n-1}\,|dr_\sigma|
        =
        \int_{(\sigma-1)/(\sigma+1)}^1 n r^{n-1}\,dr
        \le 1 .
\]
Moreover, by \eqref{eq:Stolz-r-exp-uniform},
\[
        \int_0^\pi n r_\sigma(\varphi)^n\,d\varphi
        \le
        \int_0^\infty n e^{-a_\alpha n\varphi}\,d\varphi
        =
        \frac1{a_\alpha}.
\]
Thus
\[
        \int_{\Gamma_\sigma^+} n|\zeta|^{n-1}\,|d\zeta|
        \le
        1+\frac1{a_\alpha}.
\]
The lower half is treated in the same way, and hence
\[
        \int_{\Gamma_\sigma} n|\zeta|^{n-1}\,|d\zeta|
        \le
        2\left(1+\frac1{a_\alpha}\right),
        \qquad 1<\sigma\le\alpha,\quad n\ge1 .
\]
This proves the assertion.
\end{proof}

Using Lemma \ref{lem:kernel-bound-Stolz}, we show next 
that the Wiener algebra $A^1(\D)$ is contained in any Hardy--Sobolev algebra
$\HS(S_\alpha), \alpha>1.$ This will allow us to measure the size of functions
from $A^1(\D)$ by $\HS(S_\alpha)$-norm.

\begin{proposition}\label{prop:A1-HS}
For every \(\alpha>1\),
if
\[
        f(z)=\sum_{n=0}^\infty c_nz^n\in A^1(\D),
\]
then \(f\in\HS(S_\alpha)\). More precisely,
\[
        \|f\|_{\HS(S_\alpha)}
        \le (1+C_\alpha)\|f\|_{A^1},
\]
where \(C_\alpha\) is the constant from Lemma~\ref{lem:kernel-bound-Stolz}.
\end{proposition}

\begin{proof}
Clearly,
\[
        \sup_{z\in S_\alpha}|f(z)|\le \|f\|_{A^1} .
\]
Then  Fubini's theorem and
Lemma~\ref{lem:kernel-bound-Stolz} give
\[
\begin{aligned}
        \int_{\Gamma_\alpha}|f'(\zeta)|\,|d\zeta|
        &\le
        \sum_{n=1}^\infty |c_n|
        \int_{\Gamma_\alpha} n|\zeta|^{n-1}\,|d\zeta|       \\
        &\le
        C_\alpha\sum_{n=1}^\infty |c_n|
        \le C_\alpha\|f\|_{A^1},
\end{aligned}
\]
and the statement follows.
\end{proof}

\subsection{Density of polynomials}

One of the basic features of \(\HS(S_\alpha)\) is that
the polynomials are dense in this space. 
This will considerably simplify several arguments
and is proved below.

\begin{theorem}\label{thm:density}
For every \(\alpha>1\), the polynomials are dense in \(\HS(S_\alpha)\). 
\end{theorem}

\begin{proof}
Let \(f\in\HS(S_\alpha)\), so that  \(f'\in E^1(S_\alpha)\).  Since
\(S_\alpha\) is a convex rectifiable Jordan domain, it is close-to-convex and thus a Smirnov domain,
see e.g.  \cite[p. 175]{Duren}.
The polynomial density theorem for \(E^1\) on such domains gives polynomials
\((q_n)_{n \ge 1}\) such that
\begin{equation}\label{eq:E1-density-derivatives}
        \int_{\Gamma_\alpha}|f'(\lambda)-q_n(\lambda)|\,|d\lambda|
        \longrightarrow0, \qquad n \to \infty.
\end{equation}
For every $n \in \mathbb N,$ define
\[
        p_n(z):=f(0)+\int_0^z q_n(\zeta)\,d\zeta,
        \qquad z\in S_\alpha,
\]
so that  \(p_n'=q_n\).  The
derivative part of the \(\HS\)-norm of $f-p_n$ converges to zero by
\eqref{eq:E1-density-derivatives}.  To estimate the supremum part, fix
\(z\in S_\alpha\).  Then
\[
        f(z)-p_n(z)=\int_{[0,z]}(f'(\zeta)-q_n(\zeta))\,d\zeta.
\]
The segment \([0,z]\) is compactly contained in \(S_\alpha\).  By
Lemma~\ref{lem:Granados-convex}, applied with outer boundary \(\Gamma_\alpha\)
and inner curve \([0,z]\),
\[
\begin{aligned}
        |f(z)-p_n(z)|
        &\le
        \int_{[0,z]}|f'(\zeta)-q_n(\zeta)|\,|d\zeta|       \\
        &\le
        C_{\rm B}
        \int_{\Gamma_\alpha}|f'(\lambda)-q_n(\lambda)|\,|d\lambda| .
\end{aligned}
\]
The right-hand side is independent of \(z\).  
Therefore
\[
      \lim_{n \to \infty}\sup_{z\in S_\alpha}|f(z)-p_n(z)|=0.
\]
Together with \eqref{eq:E1-density-derivatives}, this gives
\(\lim_{n \to \infty}p_n=f\) in \(\HS(S_\alpha)\).
\end{proof}

\subsection{The logarithmic kernel}\label{sec:HS-Stolz}

We next prepare the logarithmic kernels used in the Hardy--Sobolev reproducing
formula.  The construction involves two Stolz parameters, and we keep their
roles separate from the beginning.  We also collect here the elementary Stolz
geometry needed to build the cuts and to estimate the resulting kernels.

Fix
\[
        1<\delta<\tau .
\]
The parameter \(\tau\) is the Hardy--Sobolev parameter: \(f\) will belong to
\(\HS(S_\tau)\), and \(f'\) will be integrated over the boundary
\(\Gamma_\tau\).  The parameter \(\delta\) is the region on
which the resulting kernels are evaluated.  
In the operator part, if \(T\) is of Stolz type \(\sigma\), we shall choose
\[
        \sigma<\delta<\tau .
\]
Thus the scalar part itself uses only the pair \(\delta<\tau\), while the
operator part 
involves an additional spectral parameter \(\sigma\). 

For
\[
        \mu\in\Gamma_\tau\setminus\{1\}
\]
put
\[
        \varphi_\mu(z):=\frac{1-\mu z}{\mu-z},
        \qquad z\ne\mu .
\]
Then
\[
        \varphi_\mu(1)=-1,
\]
and \(\varphi_\mu\) has a zero at \(\mu^{-1}\) and a pole at \(\mu\).  
On \(S_\delta\) the scalar logarithmic kernel will be the branch of
\(\log\varphi_\mu\) whose continuous extension to the vertex \(1\)
is normalized by
\[
        \log\varphi_\mu(1)=i\pi.
\]
Since \(S_\delta\) is simply connected and \(\varphi_\mu\) has neither
zero nor pole there, this boundary normalization fixes the scalar
branch uniquely.

The auxiliary cut below does not choose a different scalar branch on
\(S_\delta\).  Rather, it gives a concrete extension of the already normalized
branch to a larger cut plane and, more importantly, an integral representation
controlled uniformly as \(\mu\to1\).  This point is already used
in the scalar estimates below: the function being estimated is defined on
\(S_\delta\), but the proof of the estimate uses the distinguished path.  The
same representation will later be transferred to the operator \(T\).

We therefore choose, for each \(\mu\), a cut \(\gamma_\mu\) joining
\(\mu^{-1}\) to \(\mu\), lying outside \(S_\tau\), and satisfying a uniform
logarithmic-length bound.  The scalar construction now has three tasks:
construct such cuts, verify that the corresponding integral represents the
normalized branch, and prove the reproducing formula.

\noindent\emph{Stolz geometry near the vertex.}
We first record the elementary geometry of Stolz domains near the vertex.
For a Stolz parameter \(\alpha>1\) write
\[
        \Omega_\alpha:=1-S_\alpha=\{u:\ 1-u\in S_\alpha\}.
\]
Also set
\[
        \psi_\alpha:=\arccos(1/\alpha)\in(0,\pi/2),
        \qquad
        R_\alpha:=\frac{2\alpha}{\alpha+1}\in(1,2).
\]

\begin{lemma}\label{lem:geometry}
Let \(\alpha>1\).
\begin{enumerate}[label=\textnormal{(\roman*)}]
\item If \(1-\mu=r e^{i\theta}\in\partial\Omega_\alpha\), with
\(r>0\) and \(\theta\in(-\pi/2,\pi/2)\), then
\begin{equation}\label{eq:costheta}
        \cos\theta
        =
        \frac1\alpha+\frac r2\Bigl(1-\frac1{\alpha^2}\Bigr).
\end{equation}
In particular \(\theta\to\pm\psi_\alpha\) as \(r\downarrow0\).
\item
\[
        \Omega_\alpha\subset B(0,R_\alpha)\cap \Sigma_{\psi_\alpha}.
\]
For \(0<r<R_\alpha\), put
\[
        \Theta_\alpha(r)
        :=\{\vartheta\in(-\pi/2,\pi/2):\ r e^{i\vartheta}\in\Omega_\alpha\}.
\]
Then \(\Theta_\alpha(r)=(-\theta_\alpha(r),\theta_\alpha(r))\), where
\(\theta_\alpha(r)\in[0,\psi_\alpha)\) is determined by
\[
        \cos\theta_\alpha(r)
        =
        \frac1\alpha+\frac r2\Bigl(1-\frac1{\alpha^2}\Bigr)
\]
whenever the right-hand side is at most \(1\), and
\(\Theta_\alpha(r)=\varnothing\) otherwise.
\item 
One has
\begin{equation}\label{eq:Stolz-local-length-main}
        \operatorname{length}\bigl(\Gamma_\alpha\cap\{|z-1|<t\}\bigr)
        \le 4 \pi C_{\rm B}t,\qquad t>0.
\end{equation}
Furthermore, for \(0<s<t\le 1\),
\[
        \int_{\Gamma_\alpha\cap\{s\le |z-1|<t\}}
             \frac{|dz|}{|z-1|}
        \le 2\sqrt2 \log \frac{t}{s}.
\]
\item Let \(1<\tau<\alpha\). For \(\eta>0\), set
\[
        A_\eta:=\partial D(1,\eta)\setminus S_\alpha.
\]
Then, for all sufficiently small \(\eta>0\), \(A_\eta\) is a closed arc
of length at most \(2\pi\eta\), and
\begin{equation}\label{curve}
        \Gamma_{\alpha,\eta}
        :=
        \bigl(\Gamma_\alpha\cap\{|z-1|\ge\eta\}\bigr)\cup A_\eta
\end{equation}
is a rectifiable Jordan curve
whose bounded complementary component
contains \(\overline{S_\tau}\).

\end{enumerate}
\end{lemma}

\begin{proof}
(i) On \(\Gamma_\alpha\) one has
\[
        |1-\mu|=\alpha(1-|\mu|).
\]
Writing \(1-\mu=r e^{i\theta}\), this gives
\[
        |\mu|=1-\frac r\alpha .
\]
On the other hand, \(|\mu|=|1-r e^{i\theta}|\).  Squaring both identities gives
\[
        1+r^2-2r\cos\theta
        =
        1-\frac{2r}{\alpha}+\frac{r^2}{\alpha^2},
\]
which is equivalent to \eqref{eq:costheta}.  Letting \(r\downarrow0\) gives
\(\cos\theta\to1/\alpha\), hence \(\theta\to\pm\psi_\alpha\).

(ii) The boundary identity in (i) gives the two boundary angles at radius
\(r\).  The defining inequality of \(S_\alpha\), written in the coordinates
\(u=1-z=re^{i\vartheta}\), is equivalent to
\[
        \cos\vartheta
        >
        \frac1\alpha+\frac r2\Bigl(1-\frac1{\alpha^2}\Bigr)
\]
inside \(\Omega_\alpha\).  Hence the admissible angles form exactly the
interval stated above.  The inclusion in the sector follows by letting
\(r\downarrow0\), and the maximal radius is obtained at \(\vartheta=0\),
which gives \(R_\alpha=2\alpha/(\alpha+1)\).

(iii)
Fix \(t>0\). Applying
Lemma~\ref{lem:Granados-convex} with \(g\equiv 1\), with outer curve
\(\partial D(1,2t)\), and 
with
\(\Gamma_\alpha\cap\overline{D(1,t)}\) as the inner convex curve
 gives
\[
\begin{aligned}
    \operatorname{length}\bigl(
        \Gamma_\alpha\cap\{|z-1|<t\}
    \bigr)
    &\le
    C_{\rm B}
    \operatorname{length}\bigl(\partial D(1,2t)\bigr)
    \\
    &=4\pi C_{\rm B}t.
\end{aligned}
\]


For the second estimate, write the upper and the lower half-plane parts of
\(\Gamma_\alpha\cap\{0<|z-1|\le1\}\) as
\[
        1-z=re^{\pm i\theta_\alpha(r)},
        \qquad 0<r\le1.
\]
By \eqref{eq:costheta},
\[
        \cos\theta_\alpha(r)
        =
        \frac1\alpha
        +
        \frac r2\left(1-\frac1{\alpha^2}\right).
\]
A direct differentiation shows that
\[
        r|\theta_\alpha'(r)|\le1,
        \qquad 0<r\le1,
\]
uniformly in \(\alpha>1\). 
Since \(0<s<t\le1\), the estimate
above applies for \(s\le r\le t\).
Hence, on either of the parts,
\[
        |dz|
        =
        \sqrt{1+r^2|\theta_\alpha'(r)|^2}\,dr
        \le \sqrt2\,dr.
\]
Adding the two corresponding integrals, we obtain, for \(0<s<t\le1\),
\[
\begin{aligned}
        \int_{\Gamma_\alpha\cap\{s\le |z-1|<t\}}
             \frac{|dz|}{|z-1|}
        \le
        2\sqrt2\int_s^{t}\frac{dr}{r} 
        =
        2\sqrt2\log\frac{t}{s}.
\end{aligned}
\]



(iv) By (ii), for all sufficiently small \(\eta>0\), the circle
\(\partial D(1,\eta)\) meets \(\Gamma_\alpha\) at exactly two points,
so that \(A_\eta\) is a well-defined rectifiable Jordan arc and
\(\operatorname{length}(A_\eta)\le 2\pi\eta\).
Since
\(\overline{S_\tau}\setminus\{1\}\subset S_\alpha\), 
the bounded complementary component of
$\Gamma_{\alpha, \eta}$
contains \(\overline{S_\tau}\) for all
sufficiently small \(\eta>0\).

\end{proof}

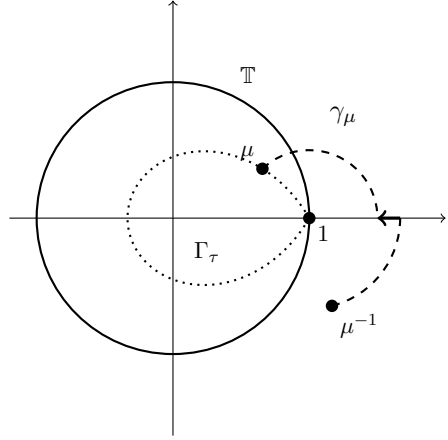
\begin{figure}[ht]
\centering
\begin{tikzpicture}[scale=1.8]

\draw[->] (-1.2,0) -- (2.0,0);
\draw[->] (0,-1.6) -- (0,1.6);

\draw[thick] (0,0) circle (1);
\node[scale=0.8] at (0.55,1.05) {$\mathbb T$};

\draw[thick,dotted,domain=0:360,samples=900,smooth,variable=\theta]
plot (
{
    (
        (8 - 2*cos(\theta))
        - sqrt((8 - 2*cos(\theta))*(8 - 2*cos(\theta)) - 36)
    )/6 * cos(\theta)
},
{
    (
        (8 - 2*cos(\theta))
        - sqrt((8 - 2*cos(\theta))*(8 - 2*cos(\theta)) - 36)
    )/6 * sin(\theta)
}
);

\node[scale=0.8] at (0.25,-0.25) {$\Gamma_\tau$};

\pgfmathsetmacro{\mux}{21/32}
\pgfmathsetmacro{\muy}{3*sqrt(15)/32}

\fill (\mux,\muy) circle (1.3pt);
\node[above left,scale=0.8] at (\mux,\muy) {$\mu$};

\pgfmathsetmacro{\ix}{7/6}
\pgfmathsetmacro{\iy}{-sqrt(15)/6}

\fill (\ix,\iy) circle (1.3pt);
\node[below right,scale=0.8] at (\ix,\iy) {$\mu^{-1}$};

\pgfmathsetmacro{\dxA}{\mux - 1}
\pgfmathsetmacro{\dyA}{\muy}
\pgfmathsetmacro{\rA}{sqrt(\dxA*\dxA + \dyA*\dyA)}
\pgfmathsetmacro{\angA}{atan2(\dyA,\dxA)}

\pgfmathsetmacro{\dxB}{\ix - 1}
\pgfmathsetmacro{\dyB}{\iy}
\pgfmathsetmacro{\rB}{sqrt(\dxB*\dxB + \dyB*\dyB)}
\pgfmathsetmacro{\angB}{atan2(\dyB,\dxB)}

\pgfmathsetmacro{\xA}{1 + \rA}
\pgfmathsetmacro{\xB}{1 + \rB}


\draw[thick,dashed]
    (1,0) ++(\angA:\rA)
    arc (\angA:0:\rA);

\draw[very thick,<-]
    (\xA,0) -- (\xB,0);

\draw[thick,dashed]
    (1,0) ++(0:\rB)
    arc (0:\angB:\rB);

\fill (1,0) circle (1.3pt);
\node[below right,scale=0.8] at (1,0) {$1$};

\node[scale=0.9] at (1.25,0.75) {$\gamma_\mu$};

\end{tikzpicture}
\caption{The path $\gamma_\mu$ joining $\mu^{-1}$ to $\mu$}
\end{figure}

\begin{lemma}\label{lem:cutexists}
Fix \(\tau>1\).  For every
\(\mu\in\Gamma_\tau\setminus\{1\}\) there is a simple piecewise \(C^1\)
oriented curve \(\gamma_\mu\), joining \(\mu^{-1}\) to \(\mu\), such that
\[
        \gamma_\mu\subset\C\setminus S_\tau,
\]
\begin{equation}\label{eq:cut-away-from-vertex}
        |\xi-1|\ge |\mu-1|,
        \qquad \xi\in\gamma_\mu,
\end{equation}
and
\begin{equation}\label{eq:cut-integral-bound}
        \int_{\gamma_\mu}\frac{|d\xi|}{|\xi-1|}
        \le
        2\pi+\log\frac1{m_\tau},
        \qquad
        m_\tau:=\frac{\tau-1}{\tau+1}.
\end{equation}
Moreover, there exists $C_\tau>0$ such that
\begin{equation}\label{eq:cut-length-bound}
        \operatorname{length}(\gamma_\mu)
        \le C_\tau |1-\mu|,
\end{equation}
and, for every \(1<\delta<\tau\), there is \(c_{\delta,\tau}>0\) such that
\begin{equation}\label{eq:cut-distance-bound}
        \dist(S_\delta,\gamma_\mu)
        \ge c_{\delta,\tau}|1-\mu|.
\end{equation}
\end{lemma}

\begin{proof}
Put
\[
        u:=1-\xi,
        \qquad
        \Omega_\tau=1-S_\tau .
\]
Thus the vertex \(1\) is moved to the origin, and
\[
        |\xi-1|=|u|,
        \qquad
        |d\xi|=|du|.
\]
Let
\[
        u_\mu:=1-\mu,
        \qquad
        u_{\mu^{-1}}:=1-\mu^{-1}
        =-\frac{u_\mu}{\mu}.
\]
Since \(|\mu|<1\), one has \(|u_\mu|<|u_{\mu^{-1}}|\).

We construct first a path \(\widetilde\gamma_\mu\) in the \(u\)-plane from
\(u_{\mu^{-1}}\) to \(u_\mu\), and then put
\(\gamma_\mu:=1-\widetilde\gamma_\mu\).  The elementary fact needed in the construction is the following consequence of
Lemma~\ref{lem:geometry}\textup{(ii)}.  For each \(r>0\), the points
\(u=re^{i\theta}\) which belong to \(\Omega_\tau\) are exactly those for which
\[
        \cos\theta>
        \frac1\tau+\frac r2\Bigl(1-\frac1{\tau^2}\Bigr).
\]
Thus, on each circle \(|u|=r\), the complement of \(\Omega_\tau\) is described
by the reverse inequality.  This complementary set contains the negative real
point \(-r\).

Now \(u_\mu\in\partial\Omega_\tau\), while
\(u_{\mu^{-1}}\notin\Omega_\tau\), because \(\mu^{-1}\notin S_\tau\).  We
define \(\widetilde\gamma_\mu\) as follows.

First, on the circle
\[
        |u|=|u_{\mu^{-1}}|=\frac{|u_\mu|}{|\mu|},
\]
join \(u_{\mu^{-1}}\) to \(-|u_{\mu^{-1}}|\) along the part of the circle on which
the reverse inequality above holds.
Second, take the radial segment on the negative real axis from
\(-|u_{\mu^{-1}}|\) to \(-|u_\mu|\).  This segment is outside
\(\Omega_\tau\), since \(\Omega_\tau\subset\Sigma_{\psi_\tau}\) and
\(\psi_\tau<\pi/2\).  Finally, on the circle \(|u|=|u_\mu|\), join \(-|u_\mu|\) to \(u_\mu\) along
the part of the circle on which the reverse inequality holds.  This last arc
ends at \(u_\mu\in\partial\Omega_\tau\), which is harmless. 
 The only possible ambiguity happens when $u_\mu$ is real. In this situation, we can take either of two possible half-circles for the last arc. 
Thus
\[
        \widetilde\gamma_\mu\subset\C\setminus\Omega_\tau,
\]
and consequently
\(
   \gamma_\mu\subset\C\setminus S_\tau.
\)
The path is simple and piecewise \(C^1\).  Since every point of
\(\widetilde\gamma_\mu\) has modulus at least \(|u_\mu|\), we also have
\eqref{eq:cut-away-from-vertex}.

Let us estimate the integral in \eqref{eq:cut-integral-bound}.  Each circular
arc is traversed through an angle at most \(\pi\), and therefore the two
circular arcs contribute at most \(2\pi\) to \(\int |du|/|u|\).  The radial
segment contributes
\[
        \int_{|u_\mu|}^{|u_{\mu^{-1}}|}\frac{dt}{t}
        =
        \int_{|u_\mu|}^{|u_\mu|/|\mu|}\frac{dt}{t}
        =
        \log\frac1{|\mu|}.
\]
Thus
\[
        \int_{\gamma_\mu}\frac{|d\xi|}{|\xi-1|}
        =
        \int_{\widetilde\gamma_\mu}\frac{|du|}{|u|}
        \le
        2\pi+\log\frac1{|\mu|}.
\]
Since \(|\mu|\ge m_\tau\) on \(\Gamma_\tau\), this gives
\eqref{eq:cut-integral-bound}.

The estimate \eqref{eq:cut-length-bound} follows from a similar argument. 
Indeed, the two circular arcs have total length at most
\[
        \pi |u_\mu|+\pi\frac{|u_\mu|}{|\mu|},
\]
and the radial segment has length
\[
        \frac{|u_\mu|}{|\mu|}-|u_\mu|.
\]
Since \(|\mu|\ge m_\tau\), the total length is bounded by
\(C_\tau |u_\mu|=C_\tau|1-\mu|\).

It remains to prove \eqref{eq:cut-distance-bound}.  Fix \(1<\delta<\tau\).
We first consider \(\mu\) close to \(1\).  Choose numbers \(b,a\) such that
\[
        \psi_\delta<b<a<\psi_\tau .
\]
By Lemma~\ref{lem:geometry}\textup{(ii)}, there is \(t_0>0\) such that
\[
        \Omega_\delta\cap B(0,t_0)\subset\{u:\ |\arg u|<b\}.
\]
After decreasing \(t_0\), the construction of \(\widetilde\gamma_\mu\) gives,
for all sufficiently small \(|u_\mu|\),
\[
        \widetilde\gamma_\mu\subset B(0,t_0/2)
        \quad\text{and}\quad
        \widetilde\gamma_\mu
        \subset \{v:\ |\arg v|\ge a\}.
\]
Hence, for \(u\in\Omega_\delta\cap B(0,t_0)\) and
\(v\in\widetilde\gamma_\mu\),
\[
        |u-v|\ge c(|u|+|v|)\ge c|u_\mu|.
\]
If \(u\in\Omega_\delta\setminus B(0,t_0)\), then
\[
        |u-v|\ge t_0/2\ge c|u_\mu|
\]
after the smallness condition on \(|u_\mu|\) is strengthened if necessary.
Thus
\[
        \dist(\Omega_\delta,\widetilde\gamma_\mu)\ge c|u_\mu|
\]
for \(\mu\) sufficiently close to \(1\).

For the remaining values of \(\mu\), say \(|u_\mu|\ge\varepsilon\), the curves
\(\widetilde\gamma_\mu\) lie in \(\C\setminus\Omega_\tau\) and depend on
\(\mu\) over a closed subarc of \(\Gamma_\tau\setminus\{1\}\).  Since
\(\overline{\Omega_\delta}\setminus B(0,\varepsilon/2)\) is a compact subset
of \(\Omega_\tau\), and the part \(\Omega_\delta\cap B(0,\varepsilon/2)\) is
separated from all such curves by $\varepsilon/2$,  
the distance between
\(\Omega_\delta\) and \(\widetilde\gamma_\mu\) is bounded below by a positive
constant depending only on \(\delta,\tau,\varepsilon\).  Since \(|u_\mu|\) is
bounded above on \(\Gamma_\tau\), this lower bound is also at least
\(c_{\delta,\tau}|u_\mu|\).  Returning to \(\xi=1-u\), we obtain
\eqref{eq:cut-distance-bound}.
\end{proof}

This completes the geometric part of the scalar construction.  In the next  subsection, we shall use the
chosen cuts to write the normalized logarithmic branches in a form suitable for
both scalar and operator estimates.

\subsection{The logarithmic branches and the scalar reproducing  formula}

Fix \(1<\delta<\tau\).  For each
\(\mu\in\Gamma_\tau\setminus\{1\}\), let \(\gamma_\mu\) be the cut given by
Lemma~\ref{lem:cutexists}. 
By \eqref{eq:cut-distance-bound}, 
the cut does not meet \(S_\delta\).

We now define the branch associated with this particular cut.    Put
\begin{equation}\label{eq:c-mu-definition}
        c_\mu
        :=
        i\pi-\int_{\gamma_\mu}\frac{d\xi}{1-\xi},
\end{equation}
and set
\begin{equation}\label{eq:branch-normalisation}
        L_\mu(z)
        :=
        c_\mu+\int_{\gamma_\mu}\frac{d\xi}{z-\xi},
        \qquad z\in\C\setminus\gamma_\mu.
\end{equation}
Then
\[
        L_\mu(1)=i\pi.
\]
The notation \(L_\mu\) depends on the chosen cut \(\gamma_\mu\), which is
fixed throughout the paper. We suppress this dependence from the notation.
 The constant \(c_\mu\) is chosen only to impose
this normalization.  The next lemma identifies \(L_\mu\) with the
corresponding branch of \(\log\varphi_\mu\).

\begin{lemma}\label{lem:cutrepr}
Let \(\mu\in\Gamma_\tau\setminus\{1\}\).  The function \(L_\mu\) defined by
\eqref{eq:branch-normalisation} satisfies
\[
        \exp L_\mu(z)=\varphi_\mu(z),
        \qquad z\in\mathbb C\setminus\gamma_\mu.
\]
Hence \(L_\mu\) is the branch of
\(\log\varphi_\mu\) on \(\mathbb C\setminus\gamma_\mu\) normalized by
\(L_\mu(1)=i\pi\).
\end{lemma}

\begin{proof}
By definition, \(L_\mu\) is holomorphic on
\(\mathbb C\setminus\gamma_\mu\), and \(L_\mu(1)=i\pi\).
Differentiating the cut integral gives
\[
        L_\mu'(z)
        =
        \frac1{z-\mu^{-1}}-\frac1{z-\mu},
        \qquad z\in\mathbb C\setminus\gamma_\mu .
\]
On the other hand,
\[
        \varphi_\mu(z)
        =
        \mu\,\frac{z-\mu^{-1}}{z-\mu},
\]
and hence
\[
        \frac{\varphi_\mu'(z)}{\varphi_\mu(z)}
        =
        \frac1{z-\mu^{-1}}-\frac1{z-\mu}.
\]
Thus
\[
        L_\mu'(z)=\frac{\varphi_\mu'(z)}{\varphi_\mu(z)},
        \qquad z\in\mathbb C\setminus\gamma_\mu .
\]
Therefore
\[
        z\mapsto \frac{\exp L_\mu(z)}{\varphi_\mu(z)}
\]
is locally constant on \(\mathbb C\setminus\gamma_\mu\).  Since the
complement of a Jordan arc is connected, this function is constant.  At
\(z=1\),
\[
        \frac{\exp L_\mu(1)}{\varphi_\mu(1)}
        =
        \frac{e^{i\pi}}{-1}
        =
        1.
\]
Hence \(\exp L_\mu=\varphi_\mu\) on
\(\mathbb C\setminus\gamma_\mu\).  
In particular, since
$S_\delta\subset\C\setminus\gamma_\mu$, the restriction of
$L_\mu$ to $S_\delta$ is the normalized branch of
$\log\varphi_\mu$ described above.
\end{proof}



Since we shall need to integrate the map $\mu\mapsto L_\mu,$ we also record the next simple lemma.

\begin{lemma}\label{lem:Lmu-parameter-regularity}
Let \(1<\delta<\tau\), and let \(K\) be a compact subarc of
\(\Gamma_\tau\setminus\{1\}\). 
 The map
\[
        (\mu,z)\mapsto L_\mu(z)
\]
is continuous on \(K\times \overline{S_\delta}\).  Moreover, for each fixed
\(z\in\overline{S_\delta}\), the function
\[
        \mu\mapsto L_\mu(z),\qquad \mu\in K,
\]
is the restriction of a holomorphic function of \(\mu\) in a neighbourhood
of \(K\).
\end{lemma}

\begin{proof}
We use only the branch identification already proved above.  On
\(\overline{S_\delta}\), \(L_\mu\) is the normalized branch satisfying
\(L_\mu(1)=\pi i\), and
\[
        \partial_z L_\mu(z)
        =
        \frac1{z-\mu^{-1}}-\frac1{z-\mu}.
\]
Since \(K\subset\Gamma_\tau\setminus\{1\}\) is compact and \(\delta<\tau\),
the points \(\mu\) and \(\mu^{-1}\), \(\mu\in K\), stay a positive distance
from \(\overline{S_\delta}\).  Hence the right-hand side is continuous in
\((\mu,z)\) on \(K\times\overline{S_\delta}\), and is locally holomorphic in
\(\mu\).

For \(z\in\overline{S_\delta}\), integrating the last identity along any
rectifiable path in \(\overline{S_\delta}\) from \(1\) to \(z\), and using
\(L_\mu(1)=\pi i\), gives
\[
        L_\mu(z)
        =
        \pi i+
        \int_1^z
        \left(
        \frac1{\zeta-\mu^{-1}}-\frac1{\zeta-\mu}
        \right)\,d\zeta .
\]
The integral is independent of the path, by the already established
holomorphy of \(L_\mu\) in the \(z\)-variable.  This formula gives the
asserted continuity and the local holomorphic dependence on \(\mu\).
\end{proof}

The following kernel estimate is the key for deriving the reproducing formula 
for $\HS(S_\tau).$ It uses two additional consequences of the construction of
\(\gamma_\mu\): the curve has length bounded by a constant times \(|1-\mu|\), while its
distance from the inner Stolz domain \(S_\delta\) is bounded below by a
constant multiple of \(|1-\mu|\).

\begin{lemma}\label{lem:Fmu-uniform}
Let \(1<\delta<\tau\).  There exists 
\(C_{\delta,\tau}>0\) such that
\begin{equation}\label{Fmu}
        \sup_{\mu\in\Gamma_\tau\setminus\{1\}}
        \sup_{z\in S_\delta}
        |L_\mu(z)|
        \le
        C_{\delta,\tau}.
\end{equation}
\end{lemma}

\begin{proof}
Throughout the proof, $C_\tau$ denotes a positive constant depending
only on $\tau$, and $C_{\delta,\tau}$ a positive constant depending
only on $\delta$ and $\tau.$ Their values may increase from one
occurrence to the next.
By the definition of \(c_\mu\) and by \eqref{eq:cut-integral-bound},
\[
        |c_\mu|
        \le
        \pi+
        \int_{\gamma_\mu}\frac{|d\xi|}{|\xi-1|}
        \le C_\tau .
\]
Hence, for \(z\in S_\delta\),
\[
        |L_\mu(z)|
        \le
        |c_\mu|+
        \int_{\gamma_\mu}\frac{|d\xi|}{|z-\xi|}
        \le
        C_\tau+
        \frac{\operatorname{length}(\gamma_\mu)}
             {\dist(S_\delta,\gamma_\mu)} .
\]
Using \eqref{eq:cut-length-bound} and \eqref{eq:cut-distance-bound}, we get
\[
        |L_\mu(z)|
        \le
        C_\tau+
        \frac{C_\tau |1-\mu|}
             {c_{\delta,\tau}|1-\mu|}
        \le C_{\delta,\tau}.
\]
The bound is uniform in \(z\in S_\delta\) and
\(\mu\in\Gamma_\tau\setminus\{1\}\).
\end{proof}

Now we are ready to prove the reproducing formula for $\HS(S_\tau),$
where $L_\mu$ plays the role of reproducing kernel.
\begin{theorem}\label{thm:scalar-reproducing}
Let \(1<\delta<\tau\), and let \(f\in\HS(S_\tau)\).  Then, for every
\(z\in S_\delta\),
\begin{equation}\label{eq:scalar-reproducing}
        f(z)-f(1)
        =
        \frac{1}{2\pi i}
        \int_{\Gamma_\tau}f'(\mu)L_\mu(z)\,d\mu.
\end{equation}
\end{theorem}
\begin{proof}
Note first that by Lemma~\ref{lem:Fmu-uniform},
\[
    |L_\mu(z)|
    \le C_{\delta,\tau},
    \qquad
    \mu\in\Gamma_\tau\setminus\{1\},
    \quad z\in S_\delta.
\]
Since $f'\in L^1(\Gamma_\tau)$, the function
\[
    g(z)
    :=
    \frac{1}{2\pi i}
    \int_{\Gamma_\tau}f'(\mu)L_\mu(z)\,d\mu,
    \qquad z\in S_\delta,
\]
is therefore well-defined.


Next, we compute the derivative of \(g\).  
Fix \(z_0\in S_\delta\), and choose a disc
\(D(z_0,\rho)\Subset S_\delta\).
Since \(\delta<\tau\),
there is a constant \(c>0\) such that
\[
        |\mu-z|\ge c,\qquad |1-\mu z|\ge c,
\]
for all \(\mu\in\Gamma_\tau\) and all \(z\in D(z_0,\rho)\). Hence
\[
        \partial_z L_\mu(z)
        =
        \frac{1}{\mu-z}-\frac{\mu}{1-\mu z}
\]
is uniformly bounded for \(\mu\in\Gamma_\tau\) and \(z\in D(z_0,\rho)\).
Since \(f'\in L^1(\Gamma_\tau)\), differentiation under the integral is
justified near \(z_0\). As \(z_0\in S_\delta\) was arbitrary, this gives
\[
        g'(z)
        =
        \frac{1}{2\pi i}
        \int_{\Gamma_\tau}\frac{f'(\mu)}{\mu-z}\,d\mu
        -
        \frac{1}{2\pi i}
        \int_{\Gamma_\tau}f'(\mu)\frac{\mu}{1-\mu z}\,d\mu ,
        \qquad z\in S_\delta.
\]

By the Cauchy formula for
\(E^1(S_\tau)\) in \eqref{cauchy}, the first integral equals \(f'(z)\).  The second integral is zero, since for fixed
\(z\in S_\delta\) the function
\[
        \mu\mapsto f'(\mu)\frac{\mu}{1-\mu z}
\]
belongs to \(E^1(S_\tau)\), and thus by \eqref{ezero} its integral over the boundary
 \(\Gamma_\tau\)
vanishes. Hence
\[
        g'(z)=f'(z),\qquad z\in S_\delta .
\]
Therefore $g-f$
is constant on \(S_\delta\).

It remains to identify this constant. Let \(z\to1\) inside \(S_\delta\). Then
\(f(z)\to f(1)\), and moreover, for every fixed
\(\mu\in\Gamma_\tau\setminus\{1\}\),
\[
        L_\mu(z)\to L_\mu(1)=i\pi .
\]
By the uniform kernel bound and \(f'\in L^1(\Gamma_\tau)\), dominated
convergence gives
\[
        g(z)\to
        \frac{1}{2\pi i}
        \int_{\Gamma_\tau}f'(\mu)i\pi\,d\mu
        =
        \frac12\int_{\Gamma_\tau}f'(\mu)\,d\mu
        =
        0.
\]
Consequently the constant is \(-f(1)\), and hence
\[
        f(z)-f(1)
        =
        \frac{1}{2\pi i}
        \int_{\Gamma_\tau}f'(\mu)L_\mu(z)\,d\mu,
        \qquad z\in S_\delta,
\]
as required.
\end{proof}


\section{Hardy--Sobolev functional calculus}\label{hcalculus}

\subsection{The operator logarithmic kernel}

We now pass from the scalar logarithmic kernels to their operator analogues.  Let
\(
        1<\sigma<\tau,
\)
and let \(T\in\L(X)\) be a Ritt operator of Stolz type \(\sigma\).
The parameter \(\tau\) will define the corresponding Hardy--Sobolev algebra, while the auxiliary
separation parameter \(\delta\), with
\(
        \sigma<\delta<\tau,
\)
will be chosen only when estimates are needed.  For such a \(\delta\), the
Ritt resolvent estimate gives
\begin{equation}\label{eq:Ritt-delta-bound}
        M_{\delta}(T):=
\sup_{\xi\in\mathbb C\setminus(S_\delta\cup\{1\})}
|\xi-1|\,\|(\xi-T)^{-1}\|<\infty.
\end{equation}
The constant \(M_{\delta}(T)\) may blow up as \(\delta\downarrow\sigma\),
and we will call it \emph{the Stolz constant} of $T.$


For each \(\mu\in\Gamma_\tau\setminus\{1\}\), we keep the cut
\(\gamma_\mu\) constructed in Lemma~\ref{lem:cutexists}.  By
\eqref{eq:cut-away-from-vertex},
\[
        \gamma_\mu
        \subset
        \mathbb C\setminus(S_\tau\cup\{1\})
        \subset
        \mathbb C\setminus(S_\delta\cup\{1\}),
\]
and hence \((\xi-T)^{-1}\) is defined for every
\(\xi\in\gamma_\mu\).

 Recall also the
branch normalization
\[
        L_\mu(z)=c_\mu+\int_{\gamma_\mu}\frac{d\xi}{z-\xi},
        \qquad
        L_\mu(1)=i\pi,
\]
where
\[
        c_\mu=i\pi-\int_{\gamma_\mu}\frac{d\xi}{1-\xi}.
\]
We define the corresponding operator logarithmic kernel by
\begin{equation}\label{eq:LmuT-cut-definition}
        L_\mu[T]:=
        c_\mu I-\int_{\gamma_\mu}(\xi-T)^{-1}\,d\xi .
\end{equation}


Since $L_\mu$ is holomorphic in a neighbourhood of $\sigma(T)$,  the Riesz--Dunford calculus gives
\begin{equation}\label{rd}
        L_\mu(T)
        =
        \frac{1}{2\pi i}
        \int_{\Gamma^{\mu}}
             L_\mu(\lambda)(\lambda-T)^{-1}\,d\lambda,
\end{equation}
for every positively oriented rectifiable Jordan curve surrounding \( \sigma(T)\)
such that  $\gamma_\mu$ is
contained in the unbounded component of $\mathbb C\setminus\Gamma^{\mu}$.
(Since \(\gamma_\mu\) is a Jordan arc disjoint from the compact set \(\sigma(T)\) such curves exist.)
\begin{lemma}\label{lem:operator-cut-faithfulness}
Let $T\in \mathcal L(X)$ be a Ritt operator of Stolz type $\sigma,$ and let $\tau> \sigma.$
For fixed \(\mu\in\Gamma_\tau\setminus\{1\}\), if  \(L_\mu [T]\) is
given by \eqref{eq:LmuT-cut-definition}, and \(L_\mu(T)\) is defined by \eqref{rd}, then
\[
L_\mu[T]=L_\mu(T).
\]
\end{lemma}

\begin{proof}


Substituting the cut representation
\[
    L_\mu(\lambda)
    =
    c_\mu+
    \int_{\gamma_\mu}\frac{d\xi}{\lambda-\xi}
\]
into the Riesz--Dunford integral and using Fubini's theorem gives
\[
\begin{aligned}
    \frac{1}{2\pi i}
    \int_{\Gamma^\mu}
         L_\mu(\lambda)(\lambda-T)^{-1}\,d\lambda
    &=
    c_\mu I
    \\
    &\quad+
    \int_{\gamma_\mu}
    \left(
    \frac{1}{2\pi i}
    \int_{\Gamma^\mu}
         \frac{(\lambda-T)^{-1}}{\lambda-\xi}\,d\lambda
    \right)d\xi.
\end{aligned}
\]
For every $\xi\in\gamma_\mu$, the inner integral is the
Riesz--Dunford value of the scalar function
$\lambda\mapsto(\lambda-\xi)^{-1}$, and hence equals
\[
    (T-\xi)^{-1}=-(\xi-T)^{-1}.
\]
Consequently,
\[
    \frac{1}{2\pi i}
    \int_{\Gamma^\mu}
         L_\mu(\lambda)(\lambda-T)^{-1}\,d\lambda
    =
    c_\mu I-
    \int_{\gamma_\mu}(\xi-T)^{-1}\,d\xi
    =
    L_\mu[T].
\]
This proves that the cut expression defining $L_\mu[T]$ is precisely
the Riesz--Dunford value $L_\mu(T)$.
\end{proof}

We finally record the logarithmic interpretation of $L_\mu[T]$.
Choose an open neighbourhood
\(
    U_\mu\subset\C\setminus\gamma_\mu
\)
of $\sigma(T)$, and put
\[
    V_\mu:=\varphi_\mu(U_\mu).
\]
Define
\[
    \log_\mu
    :=
    L_\mu\circ
    \bigl(\varphi_\mu|_{U_\mu}\bigr)^{-1}
    \qquad\text{on }V_\mu.
\]
Then
\[
    \exp\log_\mu(w)=w,
    \qquad w\in V_\mu,
\]
so that $\log_\mu$ is a holomorphic branch of the logarithm on
$V_\mu$.  Moreover,
\[
    \log_\mu\circ\varphi_\mu=L_\mu
    \qquad\text{on }U_\mu.
\]
Since
\[
    \sigma(\varphi_\mu(T))
    =
    \varphi_\mu(\sigma(T))
    \subset V_\mu,
\]
the composition rule for the holomorphic functional calculus and
Lemma~\ref{lem:operator-cut-faithfulness} yield
\[
\begin{aligned}
    \log_\mu\bigl(\varphi_\mu(T)\bigr)
    &=
    (\log_\mu\circ\varphi_\mu)(T)
    \\
    &=
    L_\mu(T)
    =
    L_\mu[T].
\end{aligned}
\]
Thus $L_\mu[T]$ is the logarithm of $\varphi_\mu(T)$ associated with
the scalar branch determined by the cut $\gamma_\mu$ and the
normalization $L_\mu(1)=i\pi$.

\subsection{Uniform operator logarithm bounds and operator reproducing formula}

The first operator estimate is the analogue of the scalar logarithmic bound,
but the proof uses a different source of control.  In the scalar estimate the
kernel \((z-\xi)^{-1}\) was controlled by the separation of \(S_\delta\) from
the cut.  In the operator estimate this kernel is replaced by the resolvent
\((\xi-T)^{-1}\), and the required control comes from the Ritt estimate \eqref{eq:off-Stolz-resolvent-prelim}
 outside a slightly larger Stolz domain.
The common ingredient is the same distinguished cut $\gamma_\mu$ and 
the uniform bound \eqref{eq:cut-integral-bound}.

To be able to write down the operator reproducing formula we need
a norm bound for the operator logarithmic kernel resembling the scalar estimate \eqref{Fmu}.
Such a bound is provided by the next result.

\begin{theorem}\label{thm:operator-log-bound}
Let \(1<\sigma<\tau\), and let \(T\in\L(X)\) be a Ritt operator
of Stolz type \(\sigma\).  For \(\mu\in\Gamma_\tau\setminus\{1\}\), let
\(L_\mu[T]\) be defined by \eqref{eq:LmuT-cut-definition}.  
Then there exists
\(C=C(T,\sigma,\tau)>0\) such that 
\begin{equation}\label{eq:operator-unif-log}
        \sup_{\mu\in\Gamma_\tau\setminus\{1\}}
        \|L_\mu[T]\|
        \le C.
\end{equation}
\end{theorem}
\begin{proof}
Fix
\[
        \delta=\frac{\sigma+\tau}{2},
\]
and put
\[
        K_\tau
        :=
        2\pi+\log\frac{1}{m_\tau}
        =
        2\pi+\log\frac{\tau+1}{\tau-1}.
\]
By \eqref{eq:LmuT-cut-definition},
\begin{equation}\label{bound1}
        \|L_\mu[T]\|
        \le
        |c_\mu|
        +
        \int_{\gamma_\mu}\|(\xi-T)^{-1}\|\,|d\xi|.
\end{equation}
The scalar normalization and Lemma~\ref{lem:cutexists} give
\[
        |c_\mu|
        \le
        \pi+
        \int_{\gamma_\mu}\frac{|d\xi|}{|\xi-1|}
        \le
        \pi+K_\tau.
\]
Since
\[
        \gamma_\mu
        \subset
        \mathbb C\setminus(S_\tau\cup\{1\})
        \subset
        \mathbb C\setminus(S_\delta\cup\{1\}),
\]
the Ritt estimate \eqref{eq:Ritt-delta-bound} yields
\[
        \|(\xi-T)^{-1}\|
        \le
        \frac{M_\delta(T)}{|\xi-1|},
        \qquad
        \xi\in\gamma_\mu.
\]
Hence, again by Lemma~\ref{lem:cutexists},
\begin{equation}\label{bound2}
\begin{aligned}
        \int_{\gamma_\mu}\|(\xi-T)^{-1}\|\,|d\xi|
        &\le
        M_\delta(T)
        \int_{\gamma_\mu}\frac{|d\xi|}{|\xi-1|} \\
        &\le
        M_\delta(T)K_\tau.
\end{aligned}
\end{equation}
Combining \eqref{bound1} and \eqref{bound2}, we obtain
\[
        \|L_\mu[T]\|
        \le
        \pi+\bigl(1+M_\delta(T)\bigr)K_\tau,
        \qquad
        \mu\in\Gamma_\tau\setminus\{1\}.
\]
Thus, letting
\[
        C
        :=
        \pi+\bigl(1+M_\delta(T)\bigr)K_\tau,
\]
finishes the proof of \eqref{eq:operator-unif-log}.
\end{proof}


We now deduce the operator counterpart of the reproducing formula 
\eqref{eq:scalar-reproducing} for $\HS(S_\tau).$ 
  Let \(1<\sigma<\tau\), and let \(T\) be a Ritt operator of
Stolz type \(\sigma\).  For \(f\in\HS(S_\tau)\) put
\begin{equation}\label{eq:HS-op-formula}
        \Psi_T(f)
        :=
        f(1)+
        \frac1{2\pi i}
        \int_{\Gamma_\tau}f'(\mu)L_\mu [T]\,d\mu .
\end{equation}
Here \(L_\mu[T]\) is the operator logarithmic kernel defined from the
cut representation in \eqref{eq:LmuT-cut-definition}.  
  

Theorem~\ref{thm:scalar-reproducing} gives the scalar reproducing
identity on a Stolz region.  To transfer this identity to the
Riesz--Dunford calculus when possibly $1\in\sigma(T)$, we truncate
the boundary integral away from $1$ and estimate the resulting error.

\begin{lemma}\label{lem:rational-truncation-compatibility}
Let $1<\sigma<\tau$, and let $T$ be a Ritt operator of Stolz type
$\sigma$.  Let $f$ be holomorphic in a neighbourhood of $\overline{S_\tau}$.
For $\varepsilon>0$ put
\[
        K_\varepsilon
        :=
        \{\mu\in\Gamma_\tau:\ |\mu-1|\ge\varepsilon\},
\]
with the inherited orientation, and define
\[
        h_\varepsilon(z)
        :=
        \frac1{2\pi i}
        \int_{K_\varepsilon}f'(\mu)L_\mu(z)\,d\mu .
\]
Then
\[
        h_\varepsilon(T)\longrightarrow f(T)-f(1),
        \qquad \varepsilon\downarrow0,
\]
in the uniform operator topology, where $f(T)$ is defined by the Riesz--Dunford functional calculus.
\end{lemma}
\begin{proof}
Choose \(\delta_0\) and \(\delta\) such that
\[
        \sigma<\delta_0<\delta<\tau .
\]
All constants below may depend on \(T,\delta_0,\delta,\tau\) and \(f\), but
not on \(\varepsilon\), and they may vary from line to line.  Put
\begin{equation}\label{eta}
        \eta=\eta(\varepsilon):=\varepsilon^2 .
\end{equation}
For \(\varepsilon>0\) sufficiently small, let
\begin{equation}\label{gammacur}
        \mathcal C_\varepsilon:=\Gamma_{\delta_0,\eta(\varepsilon)}
\end{equation}
be the positively oriented truncated curve defined as in \eqref{curve}.
Its bounded complementary component contains
\(\overline{S_\sigma}\). Moreover, 
\(\mathcal C_\varepsilon \subset \rho(T)\) and satisfies
\[
        \|(\lambda-T)^{-1}\|
        \le
        \frac{M_{\delta_0}(T)}{|\lambda-1|},
        \qquad \lambda\in\mathcal C_\varepsilon.
\]
For \(\varepsilon\) small, the compact set bounded by
\(\mathcal C_\varepsilon\) is contained in the domain of holomorphy of \(f\).


We first note that \(h_\varepsilon\) is holomorphic in a neighbourhood
of the compact set enclosed by \(\mathcal C_\varepsilon\).
Indeed, this compact set is contained in
\(
        \overline{S_{\delta_0}}\cup\overline{D(1,\eta)}.
\)
For \(\mu\in K_\varepsilon\), the cut \(\gamma_\mu\) lies outside
\(S_\tau\) and satisfies
\[
        |\xi-1|\ge|\mu-1|\ge\varepsilon>\eta,
        \qquad \xi\in\gamma_\mu.
\]
Since
\[
        \overline{S_{\delta_0}}\setminus\{1\}\subset S_\tau,
\]
the cuts do not meet the compact set enclosed by
\(\mathcal C_\varepsilon\).
Hence the Riesz--Dunford functional calculus gives
\[
        h_\varepsilon(T)-f(T)+f(1)
        =
        \frac{1}{2\pi i}
        \int_{\mathcal C_\varepsilon}
        E_\varepsilon(\lambda)(\lambda -T)^{-1}\,d\lambda ,
\]
where
\[
        E_\varepsilon(\lambda)
        :=
        h_\varepsilon(\lambda)-\bigl(f(\lambda)-f(1)\bigr).
\]


By Theorem~\ref{thm:scalar-reproducing}, for every
$\lambda\in S_\delta$,
\[
\begin{aligned}
    E_\varepsilon(\lambda)
    &=
    h_\varepsilon(\lambda)
    -
    \bigl(f(\lambda)-f(1)\bigr)
    \\
    &=
    -\frac{1}{2\pi i}
    \int_{\Gamma_\tau\setminus K_\varepsilon}
         f'(\mu)L_\mu(\lambda)\,d\mu.
\end{aligned}
\]


Put
\[
    M_f:=\sup_{\mu\in\Gamma_\tau}|f'(\mu)|.
\]
Applying Lemma~\ref{lem:Fmu-uniform} to the preceding integral, and then
using Lemma~\ref{lem:geometry} \textup{(iii)}, we obtain
\[
\begin{aligned}
    \sup_{\lambda\in S_\delta}|E_\varepsilon(\lambda)|
    &\le
    \frac{M_fC_{\delta,\tau}}{2\pi}
    \operatorname{length}
    \bigl(
        \Gamma_\tau\cap\{|\mu-1|<\varepsilon\}
    \bigr)
    \\
    &\le
    2C_{\rm B}C_{\delta,\tau}M_f\,\varepsilon.
\end{aligned}
\]

For \(\eta=\eta(\varepsilon)\) as in \eqref{eta}, write
\[
   \mathcal C_\varepsilon
   =
   \bigl(\Gamma_{\delta_0}\cap\{|\lambda-1|\ge\eta\}\bigr)
   \cup A_\eta .
\]
The first part of the union is contained in \(S_\delta\).  Therefore, by Lemma~\ref{lem:geometry}, 
\begin{gather}\label{delta0}
   \int_{\Gamma_{\delta_0}\cap\{|\lambda-1|\ge\eta\}}
   |E_\varepsilon(\lambda)|
   \|(\lambda-T)^{-1}\|\,|d\lambda| 
    \le
   C\varepsilon
   \int_{\Gamma_{\delta_0}\cap\{|\lambda-1|\ge\eta\}}
   \frac{|d\lambda|}{|\lambda-1|} \\
   \le
   C\varepsilon
   \Bigl(
      2\sqrt2\log\frac1\eta
      +
      \operatorname{length}(\Gamma_{\delta_0})
   \Bigr) 
   =
   C\varepsilon
   \Bigl(
      4\sqrt2\log\frac1\varepsilon
      +
      \operatorname{length}(\Gamma_{\delta_0})
   \Bigr).\notag
\end{gather}
Thus the contribution of this part tends to zero as $\varepsilon \downarrow 0.$


It remains to estimate the contribution of the closing arc \(A_\eta\).  Since \(L_\mu(1)=i\pi\),
\[
        h_\varepsilon(1)
        =
        \frac12 
        \int_{K_\varepsilon}f'(\mu)\,d\mu .
\]
As \(f'\) is holomorphic in a neighbourhood of \(\overline{S_\tau}\),
\[
        \int_{\Gamma_\tau}f'(\mu)\,d\mu=0,
\]
hence
\[
\begin{aligned}
        |h_\varepsilon(1)|
        \le
        \frac{M_f}{2}
        \operatorname{length}
        \bigl(\Gamma_\tau\cap\{|\mu-1|<\varepsilon\}\bigr) 
        \le
        2\pi C_{\rm B}M_f\,\varepsilon.
\end{aligned}
\]

We next estimate \(h_\varepsilon'\) near \(1\).  If
\(|\zeta-1|\le\eta\) and \(\mu\in K_\varepsilon\), then, for
\(\varepsilon\) small,
\[
        |\zeta-1|\le\eta=\varepsilon^2
        \le \frac12|\mu-1|.
\]
Consequently,
\begin{equation}\label{l1}
        |\mu-\zeta|
        \ge
        |\mu-1|-|\zeta-1|
        \ge
        \frac12|\mu-1|,
\end{equation}
and, since \(|\mu|\le1\),
\begin{gather}\label{l2}
        |1-\mu\zeta|
        =
        |(1-\mu)+\mu(1-\zeta)|  \\
        \ge
        |\mu-1|-|\mu|\,|\zeta-1|
        \ge
        \frac12|\mu-1|.\notag
\end{gather}

Using that
\[
        \partial_\zeta L_\mu(\zeta)
        =
        \frac{1-\mu^2}{(1-\mu\zeta)(\mu-\zeta)}
\]
and
\(
        |1-\mu^2|
        \le
        2|\mu-1|, 
\)
along with the two lower bounds \eqref{l1} and \eqref{l2} proved above, we obtain
\[
        |\partial_\zeta L_\mu(\zeta)|
        \le
        \frac{8}{|\mu-1|},
        \qquad
        |\zeta-1|\le\eta,\quad \mu\in K_\varepsilon.
\]
Therefore, arguing as in \eqref{delta0},
\[
\begin{aligned}
   \sup_{|\zeta-1|\le\eta}|h_\varepsilon'(\zeta)|
   &\le
   \frac4\pi
   \int_{K_\varepsilon}
        \frac{|f'(\mu)|}{|\mu-1|}\,|d\mu| \\
   &\le
   \frac{4M_f}{\pi}
   \left(
      2\sqrt2\log\frac1\varepsilon
      +
      \operatorname{length}(\Gamma_\tau)
   \right).
\end{aligned}
\]

For \(\lambda\in A_\eta\), this gives
\[
\begin{aligned}
        |h_\varepsilon(\lambda)|
        &\le
        |h_\varepsilon(1)|
        +
        |\lambda-1|
        \sup_{|\zeta-1|\le\eta}|h_\varepsilon'(\zeta)|        \\
        &\le
        C\varepsilon
        +
        C\eta\log \left(\frac{1}{\varepsilon}\right)
        \le
        C\varepsilon .
\end{aligned}
\]
Also,
\[
        |f(\lambda)-f(1)|
        \le C|\lambda-1|
        =
        C\eta
        =
        C\varepsilon^2,
        \qquad \lambda\in A_\eta .
\]
Thus
\[
        |E_\varepsilon(\lambda)|\le C\varepsilon,
        \qquad \lambda\in A_\eta .
\]
Since \(\operatorname{length}(A_\eta)\le 2\pi \eta\), we obtain
\[
\begin{aligned}
        \int_{A_\eta}
        |E_\varepsilon(\lambda)|
        \|(\lambda -T)^{-1}\|\,|d\lambda|
        &\le
        C\varepsilon
        \int_{A_\eta} \frac{|d\lambda|}{|\lambda-1|}        \\
        &\le
        C\varepsilon\, \frac{1}{\eta}\operatorname{length}(A_\eta)
        \le
        C\varepsilon .
\end{aligned}
\]
Hence the closing-arc contribution also tends to zero as \(\epsilon \downarrow 0\).

Combining the estimates on the two parts of
\(\mathcal C_\varepsilon\), we get
\[
        h_\varepsilon(T)\longrightarrow f(T)-f(1), \qquad  \varepsilon\downarrow0. 
\]
\end{proof}

\begin{theorem}\label{thm:operator-reproducing}
Let $1<\sigma<\tau$, and let $T\in\L(X)$ be a Ritt operator of Stolz type
$\sigma$.  For every $f\in\HS(S_\tau)$, the integral in
\eqref{eq:HS-op-formula} is absolutely convergent in the operator norm. 
Furthermore, 
there exists $C=C(T,\sigma,\tau)>0$ such that
\begin{equation}\label{eq:HS-op-bound}
        \|\Psi_T(f)\|
        \le
        |f(1)|
        +
        \frac{C}{2\pi}
        \int_{\Gamma_\tau}|f'(\mu)|\,|d\mu|
        \le
        \max\left(1,\frac{C}{2\pi}\right)\|f\|_{\HS(S_\tau)} .
\end{equation}
If $f$ is holomorphic in a neighbourhood of $\overline{S_\tau}$, then
\begin{equation}\label{permr}
        \Psi_T(f)=f(T),
\end{equation}
where the right-hand side is understood in the sense of the Riesz--Dunford functional calculus.
\end{theorem}

\begin{proof}
The absolute convergence and the estimate follow immediately from
Theorem~\ref{thm:operator-log-bound}.  Indeed, fixing $\delta=(\sigma+\tau)/2$  
and using \eqref{eq:operator-unif-log}, we obtain
\[
\begin{aligned}
        \|\Psi_T(f)\|
        &\le
        |f(1)|
        +
        \frac1{2\pi}
        \int_{\Gamma_\tau}|f'(\mu)|\,\|L_\mu[T]\|\,|d\mu| \\
        &\le
        |f(1)|
        +
        \frac{C}{2\pi}
        \int_{\Gamma_\tau}|f'(\mu)|\,|d\mu| .
\end{aligned}
\]
It remains to prove the compatibility assertion. Let \(f\) be holomorphic in
a neighbourhood of \(\overline{S_\tau}\). Then \(f\in\HS(S_\tau)\).
For sufficiently small \(\varepsilon>0\), let \(K_\varepsilon\) and
\(h_\varepsilon\) be as in
Lemma~\ref{lem:rational-truncation-compatibility}.

Let \(\mathcal C_\varepsilon\) be the integration curve defined by
\eqref{gammacur} in the proof of
Lemma~\ref{lem:rational-truncation-compatibility}.
For every
$\mu\in K_\varepsilon$, Lemma~\ref{lem:operator-cut-faithfulness} implies that
\[
        L_\mu[T]=
        \frac1{2\pi i}
        \int_{\mathcal C_{\varepsilon}}
             L_\mu(\lambda)(\lambda-T)^{-1}\,d\lambda.
\]
Since \(K_\varepsilon\) is compact and
\(\mathcal C_\varepsilon\) stays a positive distance from the cuts
\(\gamma_\mu\), Fubini's theorem yields
\[
\begin{aligned}
        h_\varepsilon(T)
        =
        \frac1{2\pi i}
        \int_{\mathcal C_{\varepsilon}}
             h_\varepsilon(\lambda)(\lambda-T)^{-1}\,d\lambda       
        =
        \frac1{2\pi i}
        \int_{K_\varepsilon}f'(\mu)L_\mu[T]\,d\mu .
\end{aligned}
\]
By Lemma~\ref{lem:rational-truncation-compatibility},
\[
        h_\varepsilon(T)\longrightarrow f(T)-f(1), \qquad \epsilon \downarrow 0.
\]
On the other hand, Theorem~\ref{thm:operator-log-bound} gives
\[
        \left\|
        \frac1{2\pi i}
        \int_{\Gamma_\tau\setminus K_\varepsilon}
             f'(\mu)L_\mu[T]\,d\mu
        \right\|
        \le
        \frac{C}{2\pi}
        \int_{\Gamma_\tau\setminus K_\varepsilon}|f'(\mu)|\,|d\mu|
        \longrightarrow0 .
\]
Therefore
\[
        \frac1{2\pi i}
        \int_{\Gamma_\tau}f'(\mu)L_\mu[T]\,d\mu
        =f(T)-f(1),
\]
and adding $f(1)$ proves \eqref{permr}.
\end{proof}

\subsection{Definition of the calculus and first properties}

We now show that the formula \eqref{eq:HS-op-formula} induces an operator calculus on the 
Hardy--Sobolev algebra, and thus can be considered as an operator reproducing formula.

\begin{theorem}\label{thm:HS-calc-Ritt}
Let \(T\) be a Ritt operator of Stolz type \(\sigma\), and let
\(\tau>\sigma\).  Then the map
\[
        \Psi_T:\HS(S_\tau)\longrightarrow\L(X)
\]
defined by \eqref{eq:HS-op-formula} is a bounded unital Banach algebra
homomorphism.  
\end{theorem}

\begin{proof}
Linearity, unitality and boundedness follow from
Theorem~\ref{thm:operator-reproducing}. 
 It remains to verify the
algebra-homomorphism property.

Let \(f,g\in\HS(S_\tau)\).  By Theorem~\ref{thm:density}, there exist
sequences of polynomials \((p_n)_{n\ge 1}\) and \((q_n)_{n \ge 1}\) such that
\[
        p_n\to f,
        \qquad
        q_n\to g,
        \qquad n\to\infty,
\]
in \(\HS(S_\tau)\).  Since \(\HS(S_\tau)\) is a Banach algebra,
\[
        p_nq_n\to fg,
        \qquad n\to\infty,
\]
in \(\HS(S_\tau)\). 
Hence, by the boundedness of $\Psi_T$,  
\[
        p_n(T)=\Psi_T(p_n)\to \Psi_T(f),
        \qquad
        q_n(T)=\Psi_T(q_n)\to \Psi_T(g),
\]
and similarly
\[
        \Psi_T(p_nq_n)\to \Psi_T(fg)
\]
as $n \to \infty$, in the operator norm.  Since polynomial compatibility gives
\[
        \Psi_T(p_nq_n)=(p_nq_n)(T)=p_n(T)q_n(T),
\]
passing to the limit yields 
\[
        \Psi_T(fg)=\Psi_T(f)\Psi_T(g).
\]
Thus \(\Psi_T\) is a bounded algebra homomorphism.
\end{proof}

We write
\[
        f_{\HS}(T):=\Psi_T(f),\qquad f\in\HS(S_\tau),
\]
for the corresponding \(\HS(S_\tau)\)-calculus.  
Thus by Theorem~\ref{thm:operator-reproducing}, 
there exists
\(C=C(T,\sigma,\tau)>0\) such that
\begin{equation}\label{eq:HS-calc-bound}
        \|f_{\HS}(T)\|
        \le
        C\,\|f\|_{\HS(S_\tau)}.
\end{equation}
Moreover, for every \(f\) holomorphic in a neighbourhood of \(\overline{S_\tau}\),
\[
        f_{\HS}(T)=f(T),
\]
where the right-hand side is defined by the Riesz--Dunford functional calculus.

Thus, if \(T\) is a Ritt operator of Stolz type \(\sigma\), then the preceding construction gives a bounded \(\HS(S_\tau)\)-calculus for every \(\tau>\sigma\).  The next proposition records the compatibility of these calculi for different values of \(\tau\).

\begin{proposition}\label{prop:HS-outer-independence}
Let $T$ be a Ritt operator of Stolz type $\sigma$, and let
$\tau_1,\tau_2>\sigma$.  Suppose that $f$ belongs to both
$\HS(S_{\tau_1})$ and $\HS(S_{\tau_2})$.  Then 
\[
   f_{\HS,\tau_1}(T)=f_{\HS,\tau_2}(T).
\]
\end{proposition}

\begin{proof}
Assume, without loss of generality, that $\tau_1<\tau_2$.  Choose
polynomials $(p_n)_{n \ge 1}$ such that
\[
   p_n\to f \quad\text{in }\HS(S_{\tau_2}) \quad\text{as } n\to\infty.
\]
The restriction map
$\HS(S_{\tau_2})\to\HS(S_{\tau_1})$ is continuous.  Indeed, the
supremum part of the norm decreases under restriction, while
the estimate \eqref{eq:touching-E1-two-sided}
 gives, for every
$g\in\HS(S_{\tau_2})$,
\[
   \int_{\Gamma_{\tau_1}} |g'(\zeta)|\,|d\zeta|
   \le  C_{\rm B} 
   \int_{\Gamma_{\tau_2}} |g'(\zeta)|\,|d\zeta|.
\]
Hence $p_n\to f$ in $\HS(S_{\tau_1})$  when $n \to \infty$ as well.

By Theorem~\ref{thm:HS-calc-Ritt}, applied with the two Stolz parameters
$\tau_1$ and $\tau_2$, we have
\[
   p_n(T)\to f_{\HS,\tau_1}(T),
   \qquad
   p_n(T)\to f_{\HS,\tau_2}(T)
\]
in the operator norm.  Thus the two operators are limits of the same sequence
\((p_n(T))\), and therefore they coincide.
\end{proof}

 Therefore, when the choice of an admissible Stolz parameter is irrelevant, we shall simply say that \(T\) admits the Hardy--Sobolev calculus and suppress the parameter \(\tau\) from the notation.

The next crucial corollary will allow us to estimate the operator norms
in $A^1$-calculus using the finer Hardy--Sobolev calculus.

\begin{corollary}\label{cor:A1-compatibility}
Let \(T\) be a Ritt operator of Stolz type \(\sigma\), and let
\[
        f(z)=\sum_{n=0}^\infty c_nz^n\in A^1(\D).
\]
Then, for every \(\tau>\sigma\), one has \(f\in\HS(S_\tau)\), and the
Hardy--Sobolev calculus for \(T\) is compatible with the usual
\(A^1\)-functional calculus:
\[
        f_{\HS}(T)=\sum_{n=0}^\infty c_nT^n,
\]
where the series on the right converges absolutely in the operator norm.
\end{corollary}

\begin{remark}
By Proposition~\ref{prop:HS-outer-independence}, the value of \(f_{\HS}(T)\)
is independent of the choice of admissible Stolz parameter $\tau$.
\end{remark}

\begin{proof}
The inclusion \(f\in\HS(S_\tau)\) follows from
Proposition~\ref{prop:A1-HS}.  Let
\[
        p_N(z)=\sum_{n=0}^N c_nz^n .
\]
The proof of Proposition~\ref{prop:A1-HS}, applied to the tail
\(f-p_N\), gives
\[
        \|f-p_N\|_{\HS(S_\tau)}
        \le (1+C_\tau)\sum_{n>N}|c_n|
        \to 0 .
\]
Hence, by the boundedness of the Hardy--Sobolev calculus and the compatibility with the Riesz--Dunford functional calculus established above,
\[
        p_N(T)=(p_N)_{\HS}(T)\longrightarrow f_{\HS}(T) \qquad \text{as} \quad N \to \infty
\]
in the operator norm.  On the other hand, since \(T\) is power bounded,
\[
        \left\|\sum_{n>N}c_nT^n\right\|
        \le
        \left(\sup_{k\ge0}\|T^k\|\right)
        \sum_{n>N}|c_n|
        \longrightarrow0 .
\]
Thus \(p_N(T)\to\sum_{n=0}^\infty c_nT^n\) as $N \to \infty$ in the operator norm, and the two
limits coincide.
\end{proof}

\subsection{Necessity: the Cauchy kernels and the Ritt resolvent estimate}

We now prove the converse implication.  Assume that a bounded
Hardy--Sobolev calculus is already available.  To recover the Ritt resolvent
estimate, we test the calculus on normalized Cauchy kernels.  The idea is
simple: these kernels have uniformly bounded Hardy--Sobolev norms, and their
images under the calculus are the corresponding resolvents.

\begin{lemma}\label{lem:Kz-HS}
For $|z|>1$ define
\begin{equation}\label{ckernel}
  K_z(\lambda):=\frac{z-1}{z-\lambda},\qquad \lambda\in\C\setminus\{z\}.
\end{equation}
Then $K_z \in \HS(S_\tau)$ for every $\tau>1$
and there exists a constant $C_\tau>0$ such that
\begin{equation}\label{eq:Kz-HS-unif}
  \sup_{|z|>1}\,\|K_z\|_{\HS(S_\tau)}\le C_\tau.
\end{equation}
\end{lemma}

\begin{proof}
Since \(z\notin\overline{S_\tau}\) for \(|z|>1\), the function \(K_z\) is
holomorphic in a neighbourhood of \(\overline{S_\tau}\), and
\[
        K_z'(\lambda)=\frac{z-1}{(z-\lambda)^2}.
\]
Thus
\[
        \|K_z\|_{\HS(S_\tau)}
        =
        \sup_{\lambda\in S_\tau}\frac{|z-1|}{|z-\lambda|}
        +
        \int_{\Gamma_\tau}
        \frac{|z-1|}{|z-\lambda|^2}\,|d\lambda| .
\]
Since \(\lambda\mapsto (z-\lambda)^{-1}\) is holomorphic in a
neighbourhood of \(\overline{S_\tau}\), its modulus attains its maximum on
\(\Gamma_\tau\).  Hence it is enough to prove
\begin{equation}\label{eq:kernel-main-est}
        \sup_{|z|>1}
        |z-1|\left(
        \sup_{\lambda \in\Gamma_\tau}\frac1{|z-\lambda|}
        +
        \int_{\Gamma_\tau}\frac{|d\lambda|}{|z-\lambda|^2}
        \right)<\infty .
\end{equation}
Put \(r:=|1-z|\).  We first record the elementary estimate
\begin{equation}\label{eq:outside-Stolz-distance}
        \dist(z,\overline{S_\tau})\ge \frac{r}{1+\tau},
        \qquad |z|\ge1 .
\end{equation}
Indeed, if \(w\in S_\tau\), then
\[
        1-|w|\le |z|-|w|\le |z-w|,
        \qquad
        |1-w|\le \tau(1-|w|),
\]
and hence
\[
        |1-z|\le |1-w|+|w-z|\le (1+\tau)|z-w|.
\]
Taking the infimum over \(w\in S_\tau\) gives
\eqref{eq:outside-Stolz-distance}.  Consequently
\[
  r\sup_{\lambda\in\Gamma_\tau}\frac1{|z-\lambda|}
        \le 1+\tau .
\]

It remains to estimate the integral in \eqref{eq:kernel-main-est}.  
We write
\[
        \int_{\Gamma_\tau}\frac{|d\lambda|}{|z-\lambda|^2}=I_1+I_2,
\]
where the integral \(I_1\) is taken over \(\Gamma_\tau\cap D(1,2r)\), and \(I_2\) -- over
\(\Gamma_\tau\setminus D(1,2r)\).  By Lemma~\ref{lem:geometry}\textup{(iii)}
and \eqref{eq:outside-Stolz-distance},
\[
        I_1
        \le
        \frac{(1+\tau)^2}{r^2}
        \length(\Gamma_\tau\cap D(1,2r))
        \le
        \frac{C_\tau}{r}.
\]

For the second term, if \(\lambda \not \in D(1,2r)\), then
\[
        |z-\lambda|\ge |1-\lambda|-|1-z|\ge \frac12|1-\lambda|,
\]
and hence
\[
        I_2
        \le
        4\int_{\Gamma_\tau\setminus D(1,2r)}
             \frac{|d\lambda|}{|1-\lambda|^2}.
\]
On \(\Gamma_\tau\) one has \(|1-\lambda|=\tau(1-|\lambda|)\).  Thus
\[
        \Gamma_\tau\setminus D(1,2r)
        =
        \Gamma_\tau\cap \overline{D(0,R_0)},
        \qquad
        R_0:=1-\frac{2r}{\tau},
\]
with the convention that the set is empty if \(R_0\le0\).  Assume \(R_0>0\)
and put \(R_1:=(1+R_0)/2\).  Then \(R_0<R_1<1\).  The domain
\(S_\tau\cap D(0,R_0)\) is convex, and its boundary is compactly contained in
\(D(0,R_1)\).  Applying Lemma~\ref{lem:Granados-convex} in the disc
\(D(0,R_1)\) to the function \(\lambda\mapsto(1-\lambda)^{-2}\) and to the inner curve
\(\partial(S_\tau\cap D(0,R_0))\), we obtain
\[
        \int_{\partial(S_\tau\cap D(0,R_0))}
             \frac{|d\lambda|}{|1-\lambda|^2}
        \le
        C_{\rm B}
        \int_{|\lambda|=R_1}
             \frac{|d\lambda|}{|1-\lambda|^2}.
\]
Since
\[
        \Gamma_\tau\setminus D(1,2r)
        \subset \partial(S_\tau\cap D(0,R_0)),
\]
it follows that
\[
        \int_{\Gamma_\tau\setminus D(1,2r)}
             \frac{|d\lambda|}{|1-\lambda|^2}
        \le
        C_{\rm B}
        \int_{|\lambda|=R_1}
             \frac{|d\lambda|}{|1-\lambda|^2}.
\]
The last integral is explicit:
\[
        \int_{|\lambda|=R_1}\frac{|d\lambda|}{|1-\lambda|^2}
        =
        \frac{2\pi R_1}{1-R_1^2}
        \le
        \frac{2\pi}{1-R_1}.
\]
Since \(1-R_1=(1-R_0)/2=r/\tau\), we obtain \(I_2\le C_\tau/r\).  Combining
this with the estimate for \(I_1\), we get
\[
        \int_{\Gamma_\tau}\frac{|d\lambda|}{|z-\lambda|^2}
        \le \frac{C_\tau}{r}.
\]
Together with the supremum estimate this proves \eqref{eq:kernel-main-est}, and hence implies the claim.
\end{proof}

To characterize Ritt operators by means of Hardy--Sobolev functional calculi, we now formulate an abstract notion independent of the preceding construction.

\begin{definition}\label{def:bounded-HS}
Let \(\tau>1\) and let \(T\in\L(X)\).  We say that \(T\) admits a bounded
\(\HS(S_\tau)\)-calculus if there exists a bounded unital Banach algebra
homomorphism
\[
        \Phi:\HS(S_\tau)\longrightarrow\L(X)
\]
such that
\begin{equation}\label{perm}
        \Phi(p)=p(T)
\end{equation}
for every polynomial \(p\).  In this case we write
\[
        f_{\HS}(T):=\Phi(f),\qquad f\in\HS(S_\tau).
\]
\end{definition}

By Theorem~\ref{thm:density}, such a homomorphism is unique if it exists. Thus
there is no ambiguity in the notation \(f_{\HS}(T)\) for a fixed \(\tau\).

\begin{theorem}\label{thm:Ritt-iff-HS}
For \(T\in\L(X)\), the following assertions are equivalent.
\begin{enumerate}[label=\textnormal{(\alph*)}]
\item \(T\) is a Ritt operator.
\item \(T\) admits a bounded \(\HS(S_\tau)\)-calculus for some \(\tau>1\).
\end{enumerate}
Moreover, if \(T\) is of Stolz type \(\sigma\), then \(T\) admits the bounded \(\HS(S_\tau)\)-calculus for every \(\tau>\sigma\).
\end{theorem}

\begin{proof}
If \(T\) is Ritt, choose \(\sigma>1\) such that \(T\) is of Stolz type
\(\sigma\). Then Theorem~\ref{thm:HS-calc-Ritt} gives a bounded
\(\HS(S_\tau)\)-calculus for every \(\tau>\sigma\). This proves
\textup{(a)}\(\Rightarrow\)\textup{(b)} and the final assertion.

Conversely, assume that \(T\) admits a bounded \(\HS(S_\tau)\)-calculus for
some \(\tau>1\), and denote it by
\[
        \Phi:\HS(S_\tau)\to\L(X).
\]
Thus \(\Phi\) is a bounded unital homomorphism and \(\Phi(p)=p(T)\) for every
polynomial \(p\). Let \(e(\lambda)=\lambda\).
If \(a\notin\overline{S_\tau}\), then \((a-e)^{-1}\) is holomorphic in a
neighbourhood of \(\overline{S_\tau}\). Hence \((a-e)^{-1}\in \HS(S_\tau)\),
and \(aI-T=\Phi(a-e)\) is invertible.
Thus
\[
        \sigma(T)\subset\overline{S_\tau}.
\]

For \(|z|>1\), let
\(
 K_z
\) be defined by \eqref{ckernel}.
By Lemma~\ref{lem:Kz-HS},
\[
        \sup_{|z|>1}\|K_z\|_{\HS(S_\tau)}<\infty .
\]
Since \(K_z=(z-1)(z-e)^{-1}\) in \(\HS(S_\tau)\), we have
\[
        \Phi(K_z)=(z-1)(z-T)^{-1}.
\]
Therefore
\[
        \sup_{|z|>1}\|(z-1)(z-T)^{-1}\|
        \le
        \|\Phi\|\sup_{|z|>1}\|K_z\|_{\HS(S_\tau)}<\infty .
\]
Together with \(\sigma(T)\subset\overline{S_\tau}\subset\D\cup\{1\}\), this is
the Ritt resolvent condition.
\end{proof}

The following immediate corollary is one of the main results of the paper.
\begin{corollary}
An operator \(T\in\L(X)\) is Ritt if and only if it admits a bounded Hardy--Sobolev functional calculus on some Stolz domain.
\end{corollary}

\section{Beurling--Kato: sufficiency}\label{sec:BK-sufficiency}

In this section we obtain the main \emph{discrete} counterpart of Beurling--Kato
type theorems. The absence of a small-time parameter leads to arguments that are
specific to the power family $(T^n)_{n\ge1}$.

\subsection{The discrete power Beurling--Kato mechanism}\label{subsec:discrete-KB-mechanism}

The decisive 
point in discrete time is that one should work with the whole tail
of powers $\{T^n:n\ge1\}$ rather than with a single operator $w(T)$. The
boundary singularity is concentrated at $1$, and the scale $n=n(\xi)$ plays the
role of the continuous-time parameter when one tries to recover the Ritt factor
$(\xi-1)$.

The common core of the section is therefore a \emph{powerwise arc-resolvent}
estimate:
\[
  \sup_{n\ge1}\sup_{\eta\in I}\|(\eta -T^n)^{-1}\|<\infty
\]
on a closed arc $I\subset\T\setminus\{1\}$ of positive length. Once this is
available, a geometric hitting lemma places suitable powers of a boundary point
$\xi$ close to $1$ inside $I$, and the propagation theorem upgrades the
resulting local Ritt estimate to the full Ritt resolvent bound.

There are two natural ways to produce this powerwise arc-resolvent estimate.
The first one starts from a one-point resolvent bound
\[
  \sup_{n\ge1}\|(\zeta -T^n)^{-1}\|<\infty,
  \qquad \zeta\in\T\setminus\{1\},
\]
and leads directly to the discrete Kato theorem. 
The second one starts from a
Beurling--Kato defect
\[
  \rho:=\sup_{n\ge1}\|w(T^n)\|<m_I(w),
  \qquad w\in A^{1,1}(\D),
\]
for a closed arc $I\subset \mathbb T\setminus \{1\},$ where, as above, $m_I(w)=\min_{\eta\in I}|w(\eta)|$, 
and converts it into the same arc-resolvent estimate by means of uniform control
of divided differences.

The section is organised accordingly: first the transfer
mechanism itself, then the one-point resolvent criterion, and finally the complementary defect
criterion and its examples.

\subsection{Propagation of resolvent bounds}\label{subsec:Lyubich-propagation}

The following propagation statement 
turns a local function estimate on a boundary arc
into the same type of estimate in a one-sided neighbourhood of the arc.
While we are interested only in resolvent bounds,
we prove this step in greater generality,
in view of future potential applications
in the study of Ritt operators.

\begin{lemma}\label{lem:local-propagation}
Let \(X\) be a Banach space, let \(0<\varepsilon<1\), and put
\[
        \Omega
       :=
        \{re^{i\theta}:\ 1<r<1+\varepsilon,\, |\theta|<\varepsilon\}.
\]
Let \(F:\Omega\to X\) be holomorphic and assume that \(F\) has continuous
boundary values on the two arcs
\[
        \{e^{i\theta}:\ 0<|\theta|<\varepsilon\}.
\]
Assume that, for some \(A,B>0\),
\begin{equation}\label{eq:local-radial-bound}
        (|z|-1)\|F(z)\|\le A,
        \qquad z\in\Omega,
\end{equation}
and
\begin{equation}\label{eq:local-boundary-bound}
        \|(\xi-1)F(\xi)\|\le B,
        \qquad \xi=e^{i\theta},\quad 0<|\theta|<\varepsilon .
\end{equation} 
Then
\[
       \|(z-1)F(z)\|\le 32\max\{B,4A\},
        \qquad
        1<|z|<1+\varepsilon/2,\quad |\arg z|<\varepsilon/2.
\]
\end{lemma}

\begin{proof}
%
%
For a small $\delta\in(0,\varepsilon/4)$ set
\[
        a=e^{i\delta/2},
        \qquad
        b=e^{3i\varepsilon/4},
\]
and consider the box
\[
        \Omega_\delta
        =
        \{\rho e^{i\varphi}:\ 1<\rho<1+\varepsilon/2,\ 
          \delta/2<\varphi<3\varepsilon/4\}.
\]
Define
\[
        H(z)
        =
        (z-a)(z-b)F(z),
        \qquad z\in\Omega_\delta.
\]
By the assumed boundary continuity of $F$, we have
\[
        H\in C(\overline{\Omega_\delta};X)
        \cap {\rm Hol}(\Omega_\delta;X).
\]

Set
\(
        K:=\max\{2B,8A\}.
\)
We claim that
\begin{equation}\label{hfunct}
        \|H(z)\|\le K\varepsilon,
        \qquad z\in\partial\Omega_\delta.
\end{equation}
On $\partial\Omega_\delta\cap\mathbb T$ we have \(z=e^{i\varphi}\) with
\(\delta/2\le\varphi\le3\varepsilon/4\). Hence, 
\[
        |z-a|\le |z-1|,\qquad
        |z-b|\le2\varepsilon.
\]
Thus, by \eqref{eq:local-boundary-bound},
\[
        \|H(z)\|
        \le\varepsilon K,\qquad z\in \partial\Omega_\delta\cap\mathbb T.
\]
Next, on $\partial\Omega_\delta\cap(1+\varepsilon/2)\mathbb T$, we have 
\[
        |z-a|\le2\varepsilon,\qquad
        |z-b|\le2\varepsilon,\qquad |z|-1=\varepsilon/2.
\]
Therefore, by \eqref{eq:local-radial-bound},
\[
        \|H(z)\|
        \le\varepsilon K,\qquad z\in \partial\Omega_\delta\cap(1+\varepsilon/2)\mathbb T.
\]

On the radial side \(z=\rho a\),
\(1\le \rho\le1+\varepsilon/2\), 
we have
\[
        |z-a|=\rho-1,\qquad |z-b|\le2\varepsilon.
\]
Thus
\[
        \|H(z)\|
        \le \varepsilon K.
\]
The side \(z=\rho b\), \(1< \rho\le 1+\varepsilon/2\),  is identical.  This proves the boundary
estimate.

Applying the scalar maximum principle to the scalar-valued function
\(x^*\circ H\), where \(x^*\in X^*\) and \(\|x^*\|\le1\), and then
taking the supremum over such \(x^*\), extends \eqref{hfunct} to
\[
        \|H(z)\|\le K\varepsilon,
        \qquad z\in\Omega_\delta.
\]



Now, let 
$$        
\Omega'_\delta=\{\rho e^{i\varphi}:\ 1<\rho<1+\varepsilon/2,\ 
          2\delta<\varphi<\varepsilon/2\}.
$$

For $z\in\Omega'_\delta,$ elementary geometry gives  
\[
        \frac{|z-1|}{|z-a|}\le 1+ \frac{|a-1|}{|z-a|} \le 2,
        \qquad
        |z-b|\ge |e^{3i\varepsilon/4}-e^{i\varepsilon/2}|\ge{\varepsilon/8}.
\]
Since $\Omega'_\delta\subset\Omega_\delta$, it follows that
\[
\begin{aligned}
   \|(z-1)F(z)\|
   &=
   \frac{|z-1|}{|z-a|\,|z-b|}\,\|H(z)\|  \\
   &\le
   2\,\frac{8}{\varepsilon}\,\varepsilon K
   =16K,
   \qquad z\in\Omega'_\delta.
\end{aligned}
\]

Letting $\delta\to0$ we conclude that 
$$
\|(\rho e^{i\varphi}-1)F(\rho e^{i\varphi})\|
        \le 16K,\qquad  1<\rho<1+\varepsilon/2,\ 
          0\le \varphi<\varepsilon/2,
$$
and by symmetry, 
$$
\|(\rho e^{i\varphi}-1)F(\rho e^{i\varphi})\|
        \le 16K,\qquad  1<\rho<1+\varepsilon/2,\ 
          |\varphi|<\varepsilon/2.
$$
\end{proof}

If $F$ is a resolvent, then the local bound obtained in Lemma \ref{lem:local-propagation}
can be further transformed into a global resolvent bound
outside $\D$ as the next statement shows. 
This fact was already noted in the literature,
see e.g. \cite[Theorem 2.1]{Borovykh}, and thus the proof is given for completeness and clarity of exposition.

\begin{corollary}
\label{cor:Lyubich-propagation}
Let \(T\in\L(X)\) be power bounded and assume that
\[
        \sigma(T)\cap\T\subset\{1\}.
\]
Suppose that, for some \(\varepsilon>0\),
\[
        \sup_{0<|\theta|<\varepsilon}
        \|(e^{i\theta}-1)(e^{i\theta}-T)^{-1}\|<\infty .
\]
Then \(T\) is a Ritt operator.
\end{corollary}

\begin{proof}
Put
\[
        F(z)=(z-T)^{-1},\qquad |z|>1.
\]
If
\(
        M:=\sup_{n\ge0}\|T^n\|,
\)
then the Neumann series gives
\[
        (|z|-1)\|F(z)\|\le M,\qquad |z|>1.
\]
Moreover, by the spectral assumption, \(F\) has continuous boundary values on
\(\T\setminus\{1\}\) near \(1\), and the hypothesis gives
\[
B:=
\sup_{0<|\theta|<\varepsilon}
\|(e^{i\theta}-1)F(e^{i\theta})\|<\infty.
\]

Applying Lemma~\ref{lem:local-propagation} to \(F\) with $A=M$ and $B$ as above, we obtain that for some $\varepsilon_0>0$, 
\((z-1)(z-T)^{-1}\) is bounded in
\[
        \{z:\ 1<|z|<1+\varepsilon_0,\ |\arg z|<\varepsilon_0\}.
\]
On the remaining part of the annulus \(1<|z|<1+\varepsilon_0\), the function
is bounded by compactness, since \(\sigma(T)\cap\T\subset\{1\}\). Hence
\[
        \sup_{1<|z|<1+\varepsilon_0}
        \|(z-1)(z-T)^{-1}\|<\infty .
\]
Finally, for \(|z|\ge 1+\varepsilon_0\), the Neumann series gives
\[
        \|(z-1)(z-T)^{-1}\|
        \le
        M\,\frac{|z-1|}{|z|-1}
        \le
        M\,\frac{|z|+1}{|z|-1}
        \le
        M\,\frac{2+\varepsilon_0}{\varepsilon_0}.
\]
Therefore
\[
        \sup_{|z|>1}\|(z-1)(z-T)^{-1}\|<\infty .
\]
\end{proof}

\subsection{A Wiener--Sobolev control of divided differences}\label{subsec:divdiff}

To pass from invertibility of $w(\eta)I-w(T^n)$ to
resolvent bounds for $(\eta -T^n)^{-1}$ one needs uniform control of the divided
differences
\[
  h_\eta(z)=\frac{w(\eta)-w(z)}{\eta-z}.
\]

The next simple lemma serves this purpose.

We use the notation $A^1(\D)$, $A^{1,1}(\D)$ and the corresponding norms introduced in Subsection~\ref{subsec:notation}.

\begin{lemma}\label{lem:divdiff-A11}
Let $w(z)=\sum_{n\ge0} c_n z^n\in A^{1,1}(\D)$ and fix $\eta\in\T$.
Define
\[
  h_\eta(z):=
  \begin{cases}
    \dfrac{w(\eta)-w(z)}{\eta-z}, & z\neq \eta,\\[1ex]
    w'(\eta), & z=\eta.
  \end{cases}
\]
Then $h_\eta\in A^1(\D)$ and
\begin{equation}\label{eq:h-eta-A1}
  \|h_\eta\|_{A^1}\le \|w\|_{A^{1,1}}, 
  \qquad \eta\in\T.
\end{equation}
Moreover, for every power-bounded operator $T$,
\begin{equation}\label{eq:divdiff-operator-id}
  w(\eta)-w(T)=(\eta-T)\,h_\eta(T),
\end{equation}
where $h_\eta(S)$ is defined by the absolutely convergent $A^1$-series.
\end{lemma}

\begin{proof}
For each $n\ge1$ we have
\[
  \frac{\eta^n-z^n}{\eta-z}=\sum_{k=0}^{n-1}\eta^{n-1-k}z^k,
\]
hence
\[
  h_\eta(z)=\sum_{n\ge1} c_n\sum_{k=0}^{n-1}\eta^{n-1-k}z^k.
\]
Taking $A^1$--norms and using $|\eta|=1$ gives
\[
  \|h_\eta\|_{A^1}
  \le \sum_{n\ge1}|c_n|\sum_{k=0}^{n-1}1
  =\sum_{n\ge1}n|c_n|
  \le \sum_{n\ge0}(n+1)|c_n|=\|w\|_{A^{1,1}}.
\]
Identity \eqref{eq:divdiff-operator-id} follows by applying the same algebraic identity
termwise in the $A^1$-series (which converges since $T$ is power-bounded).
\end{proof}

\subsection{A geometric hitting lemma}\label{subsec:arc-selection}

Once one has resolvent control for $T^n$ on a fixed arc
$I\subset\T\setminus\{1\}$, the next task is to transfer this information back to
$(\xi -T)^{-1}$ for $\xi$ near $1$. The elementary lemma below does exactly this by
choosing $n=n(\xi)$ so that $\xi^n\in I$ while $n|\xi-1|$ stays uniformly
bounded.

\begin{lemma}
\label{lem:hitting-arc}
Let $I\subset\T\setminus \{1\}$ be a closed arc of positive length.
Then there exists $\theta_*\in(0,1)$ with the following property:

for every $\xi=e^{i\theta}\in\T$ with $0<|\theta|<\theta_*$ there exists an integer
$n=n(\xi)\ge 1$ such that
\[
\xi^n\in I
\qquad\text{and}\qquad
n\,|\xi-1|\le 2\pi .
\]
\end{lemma}

\begin{proof}
It suffices to treat $0<\theta<\theta_*$, since the case $\theta<0$ follows by applying the positive case to the conjugate arc $\overline I$.
Choose $\phi_0\in(0,2\pi)$ and $n_0\in\N$ so large that, with $\delta:=1/n_0$,
\[
        J:=\{e^{i\phi}:\ |\phi-\phi_0|\le \delta\}\subset I,
        \qquad
        0<\phi_0-\delta<\phi_0+\delta<2\pi .
\]
Set $\theta_*:=\delta,$
let $0<\theta<\theta_*$, and write $\xi=e^{i\theta}$.  
Let $n\in \mathbb N$ be the smallest number such that
\[
        n\theta\ge \phi_0-\delta .
\]
Then
\[
        (n-1)\theta<\phi_0-\delta,
\]
and hence
\[
        n\theta<\phi_0-\delta+\theta<\phi_0 .
\]
Thus
\[
        n\theta\in[\phi_0-\delta,\phi_0]\subset[\phi_0-\delta,\phi_0+\delta],
\]
so that $\xi^n=e^{in\theta}\in J\subset I$. Finally,
\[
        n|\xi-1|=n|e^{i\theta}-1|\le n\theta<\phi_0<2\pi .
\]
This proves the assertion.
\end{proof}

\subsection{A unified power Beurling--Kato principle}\label{subsec:power-KB}

Now we are able to develop the sufficiency direction in the discrete Beurling--Kato theory,
leading to Ritt resolvent bounds.
We elaborate two routes 
based on the discrete analogue of Kato's resolvent criterion
and on norm-gap assumptions in Beurling's spirit.
Their common core is a uniform resolvent estimate for \(T^n\) on a fixed
arc \(I\subset\mathbb T\setminus\{1\}\).
Once this estimate is available, Lemma~\ref{lem:hitting-arc} transfers it  back to a local
Ritt estimate for $T$ near $1$, and the propagation Corollary \ref{cor:Lyubich-propagation} finishes the argument.
We therefore isolate the transfer step first and then record two ways of producing
its hypothesis.

\begin{theorem}\label{thm:arc-resolvent-powers-ritt}
Let \(T\in\L(X)\) be power-bounded and assume that
\[
  \sigma(T)\cap\T\subset\{1\}.
\]
Assume that there exist a closed arc $I\subset\T\setminus\{1\}$ of positive length
and a constant $K\ge1$ such that
\begin{equation}\label{eq:powerwise-arc-resolvent}
  \sup_{n\ge1}\sup_{\eta\in I}\|(\eta -T^n)^{-1}\|\le K.
\end{equation}
Then $T$ is a Ritt operator.
\end{theorem}

\begin{proof}
Let $\sup_{n \ge 0} \|T^n\|:=M$  and
let $\theta_*>0$ be given by Lemma~\ref{lem:hitting-arc}. Choose $\varepsilon>0$ so small that
whenever $\xi=e^{i\theta}\in\T$ satisfies $0<|\xi-1|<\varepsilon$, one has $0<|\theta|<\theta_*$. Fix such a point $\xi$.
Since $\sigma(T)\cap\T\subset\{1\}$, we have $\xi\in\rho(T)$. By Lemma~\ref{lem:hitting-arc} there exists
$n=n(\xi)\ge1$ such that $\eta:=\xi^n\in I\subset\T\setminus\{1\}$ and
\[
  n|\xi-1|\le 2\pi.
\]

Using the factorisation
\[
  \xi^n -T^n=(\xi -T)\,Q_n(\xi,T),
  \qquad
  Q_n(\xi,T)=\sum_{k=0}^{n-1}\xi^{n-1-k}T^k,
\]
we obtain
\[
  (\xi -T)^{-1}=Q_n(\xi,T)\,(\xi^n -T^n)^{-1}.
\]
Hence $\|Q_n(\xi,T)\|\le \sum_{k=0}^{n-1}\|T^k\|\le Mn$, and therefore
\[
  \|(\xi -T)^{-1}\|\le MnK.
\]
Multiplying by $|\xi-1|$ and using $n|\xi-1|\le 2\pi$ yields
\begin{equation}\label{eq:local-Ritt-circle-arc}
  \sup\Bigl\{\|(\xi-1)(\xi -T)^{-1}\|:\ \xi\in\T,\ 0<|\xi-1|<\varepsilon\Bigr\}<\infty.
\end{equation}
Applying Corollary~\ref{cor:Lyubich-propagation},
we conclude that \(T\) is a Ritt operator.
\end{proof}


The one-point resolvent criterion is now immediate.

\begin{theorem}\label{thm:one-point-power-Kato}
Let $T\in\L(X)$ be power-bounded and assume that
\[
  \sigma(T)\cap\T\subset\{1\}.
\]
Assume that for some $\zeta\in\T\setminus\{1\}$,
\[
  \sup_{n\ge1}\|(\zeta -T^n)^{-1}\|<\infty.
\]
Then $T$ is a Ritt operator.
\end{theorem}

\begin{proof}
Put
\[
        K_\zeta:=\sup_{n\ge1}\|(\zeta -T^n)^{-1}\|<\infty.
\]
By the Neumann series, there is a closed arc
$I\subset\T\setminus\{1\}$ of positive length, containing $\zeta$, such that
\[
        \sup_{n\ge1}\sup_{\eta\in I}\|(\eta -T^n)^{-1}\|<\infty.
\]
Thus \eqref{eq:powerwise-arc-resolvent} holds, and
Theorem~\ref{thm:arc-resolvent-powers-ritt} implies that $T$ is Ritt.
\end{proof}

The norm-gap, or defect, route is a complementary way of
producing the same
powerwise arc-resolvent estimate.
It is conceptually parallel, but analytically different: here one starts from the strict inequality
\[
  \sup_{n\ge1}\|w(T^n)\|<m_I(w)
\]
on a nonempty closed arc $I\subset\T\setminus\{1\}$ and produces the same uniform
arc-resolvent estimate through the estimates for divided differences in $A^{1,1}(\D)$.
The argument relies on the following variation upon Kato's discrete resolvent bound.

\begin{proposition}\label{prop:defect-to-arc-resolvent}
Let $T\in\L(X)$ be power-bounded: $\|T^n\|\le M$ for $n\ge0$.
Let $w\in A^{1,1}(\D)$ and let 
$I\subset\T\setminus\{1\}$ be a closed arc
of positive length. Assume that
\begin{equation}\label{eq:defect-wTn}
  \rho:=\sup_{n\ge1}\|w(T^n)\| \,<\, m_I(w),
\end{equation}
Then
\begin{equation}\label{eq:resolvent-Tn-on-I}
  \sup_{n\ge1}\sup_{\eta\in I}\|(\eta -T^n)^{-1}\|
  \le \frac{M\|w\|_{A^{1,1}}}{m_I(w)-\rho}.
\end{equation}
\end{proposition}

\begin{proof}
Fix $n\ge1$ and $\eta\in I$. Since $|w(\eta)|\ge m_I(w)>\rho\ge \|w(T^n)\|$,
the operator $w(\eta)I-w(T^n)$ is invertible and
\[
  \|(w(\eta)-w(T^n))^{-1}\|
  \le \frac{1}{|w(\eta)|-\|w(T^n)\|}
  \le \frac{1}{m_I(w)-\rho}.
\]
Let $h_\eta$ be as in Lemma~\ref{lem:divdiff-A11}. Then
\[
  w(\eta)I-w(T^n)=(\eta-T^n)\,h_\eta(T^n).
\]
Both factors on the right-hand side are functions of $T^n$, hence they commute.
Since their product $w(\eta)I-w(T^n)$ is invertible, each factor is invertible as well.
In particular $\eta\in\rho(T^n)$. For commuting invertible factors, the inverse of the
product is the product of the inverses, 
so we obtain
\[
  (\eta -T^n)^{-1}
  =(w(\eta)-w(T^n))^{-1}h_\eta(T^n).
\]
Moreover,
\[
  \|h_\eta(T^n)\|
  \le \|h_\eta\|_{A^1}\,\sup_{k\ge0}\|T^{nk}\|
  \le M\,\|w\|_{A^{1,1}}.
\]
This yields \eqref{eq:resolvent-Tn-on-I}.
\end{proof}

Combining the preceding proposition with the transfer step, we obtain the
following defect condition ensuring the Ritt resolvent bound.

\begin{theorem}\label{thm:power-KB-arc}
Let $T\in\L(X)$ be power-bounded and assume that
\[
  \sigma(T)\cap\T\subset\{1\}.
\]
Let $w\in A^{1,1}(\D)$ and
let 
$I\subset\T\setminus\{1\}$ be a closed arc
of positive length. If
\[
  \sup_{n\ge1}\|w(T^n)\|< \min_{\eta\in I} |w(\eta)|,
\]
then $T$ is a Ritt operator.
\end{theorem}

\begin{proof}
By Proposition~\ref{prop:defect-to-arc-resolvent}, condition \eqref{eq:defect-wTn}
produces the arc-resolvent estimate \eqref{eq:powerwise-arc-resolvent}.
Applying Theorem~\ref{thm:arc-resolvent-powers-ritt}, we conclude that $T$ is Ritt.
\end{proof}

The preceding argument shows how the Ritt factor $(\xi-1)$ is produced \emph{dynamically} via the scale selection $n=n(\xi)$, rather than by a static factorisation of $w$. The role of $A^{1,1}(\D)$ in the defect route is precisely to control divided differences uniformly in $\eta\in\T$, so that invertibility of $w(\eta)I-w(T^n)$ can be converted into resolvent control of $(\eta -T^n)^{-1}$, uniformly in $n$.


The defect also forces spectral exclusion on $\T$. We shall use the following
standard spectral inclusion and boundary estimate, see, for instance,
\cite[Theorem 3.35]{LeMerdy}.

\begin{lemma}\label{lem:approx-point}
Let $T\in\L(X)$ be power bounded.
If $w\in A^1(\D)$, then
\begin{equation}\label{eq:A1-spectral-mapping}
        \sigma(w(T))\subset w(\sigma(T)).
\end{equation}
Moreover, if $\lambda\in\sigma(T)\cap\T$, then
\begin{equation}\label{eq:w-lambda-bound}
        |w(\lambda)|\le \|w(T)\|.
\end{equation}
\end{lemma}

The next proposition records the spectral exclusion that already follows from the same defect hypothesis.
\begin{proposition}\label{prop:power-KB-spectrum-exclusion}
Let \(T\in\mathcal L(X)\) be power bounded, let
\(w\in A^{1,1}(\mathbb D)\), and let
\(I\subset\mathbb T\setminus\{1\}\) be a closed arc of positive length.
Assume that
\[
        \sup_{n\ge1}\|w(T^n)\|<m_I(w).
\]
Then every point of \(\sigma(T)\cap\T\) is a root of unity.

If, moreover, \(I\) meets the orbit
\(\{\lambda^n:n\ge1\}\) of every root of unity
\(\lambda\neq1\), then
\[
        \sigma(T)\cap\T\subset\{1\}.
\]
\end{proposition}

\begin{proof}
Put
\[
        \rho:=\sup_{n\ge1}\|w(T^n)\|.
\]
Let \(\lambda\in\sigma(T)\cap\T\). By
Lemma~\ref{lem:approx-point},
\[
        |w(\lambda^n)|
        \le
        \|w(T^n)\|
        \le\rho
        <
        m_I(w),
        \qquad n\ge1.
\]
Hence
\[
        \{\lambda^n:n\ge1\}\cap I=\emptyset.
\]
If \(\lambda\) is not a root of unity, its orbit is dense in \(\T\)
and therefore meets 
 \(I\), a
contradiction. The additional assumption excludes all roots of unity
different from \(1\).
\end{proof}

This completes the sufficiency core of the defect approach to resolvent bounds. Starting from hypotheses on the
powers $(T^n)_{n \ge 0}$, we have obtained the common arc-resolvent mechanism, derived
the one-point Kato criterion and the complementary defect criterion, and isolated
the spectral information that already follows at this stage. 
The subsequent examples and sharpness results remain within this power-based picture.
Later, in Section~\ref{sec:HS-back-powers}, we shall return to the same
quantities from the converse direction, using the Hardy--Sobolev calculus
of Section~\ref{hcalculus}.

\section{Examples and sharpness}\label{examples-sec}

We record three elementary examples of norm-gap conditions implying
the Ritt property: model polynomial symbols, rational variants of these
symbols, and an entire symbol.
 The sharpness results following these
examples have a different role and describe limitations of the criterion.

\par\medskip
\noindent\textbf{The model symbols $w(z)=1-z^m$.}
For \(w(z)=1-z^m\) one has \(\|w\|_{\infty}=2\), and the corresponding defect
condition is exactly a zero--two condition for the powers \(T^{mn}\).

\begin{corollary}\label{bk-example}
\label{cor:1-zm}
Let \(T\in\L(X)\) be power bounded, and let \(m\in\N\).  Suppose that
\(\sigma(T)\cap\T\subset\{1\}\) and
\begin{equation}\label{eq:defect-1-zm}
        \sup_{n\ge1}\|I-T^{mn}\|<2 .
\end{equation}
Then \(T\) is a Ritt operator.
\end{corollary}

\begin{proof}
Let \(w(z)=1-z^m\).  Then \(w\in A^{1,1}(\D)\), \(w(1)=0\), and
\(\|w\|_{\infty}=2\).  Choose \(\eta_0\in\T\) with
\(\eta_0^m=-1\).  
By continuity, every sufficiently small closed arc
$I\subset\T\setminus\{1\}$ containing $\eta_0$ satisfies
\[
        m_I(w)>\sup_{n\ge1}\|I-T^{mn}\|.
\]

Since
\[
        w(T^n)=I-T^{mn},\qquad n\ge1,
\]
Theorem~\ref{thm:power-KB-arc} applies and gives that \(T\) is Ritt.
\end{proof}

\begin{remark}
For \(m=1\) this is a power analogue of Beurling--Kato's zero--two form.  In the
contraction case the constant \(2\) is also the natural boundary constant.
\end{remark}

\par\medskip
\noindent\textbf{A rational variant of the model symbols.}
We next replace the model polynomial symbol \(1-z^m\) by the rational
symbol
\[
        w_r(z):=\frac{1-z^m}{1-rz^m},
        \qquad 0<r<1.
\]
This gives the following rational version of the preceding criterion.

\begin{corollary}\label{cor:rational-regularised}
Let \(T\in\L(X)\) be power bounded, let \(m\in\N\), and let \(0<r<1\).
Suppose that \(\sigma(T)\cap\T\subset\{1\}\) and
\begin{equation}\label{eq:rational-regularised-defect}
        \sup_{k\ge1}
        \left\|
        (I-T^{mk})(I-rT^{mk})^{-1}
        \right\|
        <
        {2\over 1+r}.
\end{equation}
Then \(T\) is a Ritt operator.
\end{corollary}

\begin{proof}
Let
\[
        w(z):={1-z^m\over 1-rz^m},\qquad z\in\D .
\]
Then \(w\in A^{1,1}(\D)\), since
\[
        {1\over 1-rz^m}=\sum_{j\ge0}r^jz^{mj},
        \qquad
        w(z)=\sum_{j\ge0}r^j\bigl(z^{mj}-z^{m(j+1)}\bigr).
\]
Moreover \(w(1)=0\), and, putting \(u=z^m\), we get
\[
        \|w\|_\infty
        =
        \max_{|u|=1}\left|{1-u\over 1-ru}\right|
        =
        {2\over 1+r}.
\]
Indeed, for \(u=e^{i\theta}\),
\[
        \left|{1-u\over 1-ru}\right|^2
        =
        {2(1-\cos\theta)\over 1+r^2-2r\cos\theta},
\]
and the right-hand side is maximal at \(\cos\theta=-1\).

Choose \(\eta_0\in\T\) with \(\eta_0^m=-1\).  By continuity, every
sufficiently small closed arc \(I\subset\T\setminus\{1\}\) containing
\(\eta_0\) satisfies
\[
        m_I(w)>
        \sup_{k\ge1}
        \left\|
        (I-T^{mk})(I-rT^{mk})^{-1}
        \right\|.
\]
For every \(k\ge1\),
\[
        w(T^k)=(I-T^{mk})(I-rT^{mk})^{-1},
\]
where
\[
        (I-rT^{mk})^{-1}=\sum_{j\ge0}r^jT^{mkj}
\]
converges in norm.  Thus Theorem~\ref{thm:power-KB-arc} applies and gives
that \(T\) is Ritt.
\end{proof}

\begin{remark}
For \(r=0\) one formally recovers the model symbol \(1-z^m\) and the
constant \(2\).  For \(0<r<1\) the threshold becomes \(2/(1+r)\), while 
the corresponding rational operator expression is
\[
        (I-T^{mk})(I-rT^{mk})^{-1}.
\]
The case \(m=1\) gives the M\"obius family \((1-z)/(1-rz)\).
\end{remark}

\par\medskip
\noindent\textbf{The entire symbol \(e^{-z}\).}
The same method applies to the entire symbol \(w(z)=e^{-z}\) and yields the following defect condition for the Ritt property.
\begin{corollary}\label{ex:exp-symbol}
Let \(T\in\L(X)\) be power bounded and suppose that
\(\sigma(T)\cap\T\subset\{1\}\).  If
\begin{equation}\label{eq:exp-symbol-defect}
        \sup_{n\ge1}\|e^{-T^n}\|<e,
\end{equation}
then \(T\) is a Ritt operator.
\end{corollary}

\begin{proof}
Let
\[
        w(z)=e^{-z},\qquad z\in\D .
\]
Then \(w\in A^{1,1}(\D)\), since
\[
        e^{-z}=\sum_{k\ge0}{(-1)^k\over k!}z^k
\]
and the corresponding weighted coefficient sum converges.  For every
power-bounded operator \(T\),
\[
        w(T^n)=e^{-T^n},\qquad n\ge1,
\]
where the exponential is defined by the absolutely convergent power series.
On the unit circle,
\(
        \|w\|_{\infty}=e,
\)
and this maximum is attained at \(-1\), while \(|w(1)|=e^{-1}<e\).
By continuity, one can choose a closed arc
\(I\subset\T\setminus\{1\}\) containing \(-1\) such that
\[
        m_I(w)>\sup_{n\ge1}\|e^{-T^n}\|.
\]
Theorem~\ref{thm:power-KB-arc} therefore applies and gives that \(T\) is
Ritt.
\end{proof}

\noindent\textbf{Sharpness and limitations.}
\label{subsec:examples_thm35}
The preceding examples illustrate sufficient conditions.  We now record two
limitations.  A single small defect \(\|w(T)\|\) for a normal operator $T$ on a Hilbert space need not imply the Ritt property:
the powers may still attain the boundary maximum \(\|w\|_\infty\) along a spectrum
approaching \(1\) tangentially.  
The last two results illustrate, in complementary ways, the role of the
full family of powers in the one-point resolvent characterization of the
Ritt property. In the first, the property fails along every subsequence;
in the second, it holds along a sufficiently sparse subsequence although
the operator is not Ritt.

We use the Stolz domains introduced in \eqref{eq:Stolz-Ssigma}.

\begin{proposition}
\label{prop:sharpness_abstract_w_exact}
Let a non-zero $w\in A^{1,1}(\D)$ satisfy $w(1)=0.$
Then for every $\varepsilon>0$ there exists a normal contraction $T$ on $\ell^2:=\ell^2(\mathbb N\cup \{0\})$ with
$\sigma(T)\subset \D\cup\{1\}$ such that:
\begin{enumerate}
\item[(i)] \(\sigma(T)\not\subset\overline{S_\sigma}\) for every
\(\sigma>1\) (hence \(T\) is not Ritt).
\item[(ii)] $\|w(T)\|<\varepsilon$.
\item[(iii)] $\displaystyle \sup_{n\ge1}\|w(T^n)\|=\|w\|_\infty$.
\end{enumerate}
\end{proposition}

\begin{proof}
Let
\(\eta_0=e^{i\varphi_0}\) with $\varphi_0\in(0,2\pi)$ be such that $|w(\eta_0)|=\|w\|_\infty.$
By continuity of $w$ on $\overline\D$ and $w(1)=0$, choose $\delta_0>0$ such that
$|w(z)|<\varepsilon$ whenever $|z-1|<\delta_0$.

Pick any sequence \((\theta_k)_{k\ge1}\subset\mathbb R\) such that
\(\theta_k\downarrow0\),
\[
   2\theta_k<\delta_0,\qquad
   \theta_k\le1,\qquad
   \theta_k<\varphi_0.
\]

For every $k \in \mathbb N,$ set
\[
\varepsilon_k:=\theta_k^2,\qquad \lambda_k:=(1-\varepsilon_k)e^{i\theta_k}\in\D.
\]
Then for every $k \in \mathbb N$ we have
 \[ |1-\lambda_k|\le \varepsilon_k+|1-e^{i\theta_k}|\le \theta_k^2+\theta_k\le 2\theta_k<\delta_0,\]
  hence
$|w(\lambda_k)|<\varepsilon$ for all $k$.
Since \(\lambda_k\to1\) as \(k\to\infty\) and \(w(1)=0\), it follows that
\[
        \sup_{k\ge0}|w(\lambda_k)|<\varepsilon,
        \qquad \lambda_0:=1.
\]

Let $T$ be diagonal on $\ell^2$ with eigenvalues $\{\lambda_0,\lambda_1,\lambda_2,\dots\}$.
Then $T$ is normal, $\|T\|\le1$, and
\[
        \|w(T)\|   =
        \sup_{\mu\in\sigma(T)}|w(\mu)|
        =   \sup_{k\ge0}|w(\lambda_k)| <\varepsilon,
\]
proving \textup{(ii)}.

Moreover, \(r_k:=|\lambda_k|=1-\theta_k^2\uparrow1\) and
\(|1-\lambda_k|\asymp\theta_k\) as \(k\to\infty\), so
\[
\frac{|1-\lambda_k|}{1-r_k}
\asymp
\frac{\theta_k}{\theta_k^2}
=
\frac1{\theta_k}
\longrightarrow\infty.
\]
This proves \textup{(i)}, and in particular \(T\) is not Ritt.

Finally, define
\[
        n_k:=\left\lfloor\frac{\varphi_0}{\theta_k}\right\rfloor,
        \qquad k\in\mathbb N.
\]
Then, as \(k\to\infty\),
\[
        n_k\theta_k\longrightarrow\varphi_0,
        \qquad
        n_k\varepsilon_k
        =
        n_k\theta_k^2
        \sim
        \varphi_0\theta_k
        \longrightarrow0.
\]
Hence
\[
        \lambda_k^{n_k}
        =
        (1-\varepsilon_k)^{n_k}e^{in_k\theta_k}
        \longrightarrow
        e^{i\varphi_0}
        =
        \eta_0.
\]
By continuity of \(w\) on \(\overline\D\),
\[
        |w(\lambda_k^{n_k})|
        \longrightarrow
        |w(\eta_0)|
        =
        \|w\|_\infty.
\]
Since \(\lambda_k^{n_k}\in\sigma(T^{n_k})\), it follows that
\[
        \|w(T^{n_k})\|
        \ge
        |w(\lambda_k^{n_k})|
        \longrightarrow
        \|w\|_\infty,
\]
and hence
\(
        \sup_{n\ge1}\|w(T^n)\|
        \ge
        \|w\|_\infty.
\)
On the other hand, since \(T^n\) is normal,
\[
   \|w(T^n)\|
   =
   \max_{\lambda\in\sigma(T^n)}|w(\lambda)|
   \le\|w\|_\infty.
\]
Thus $\sup_{n\ge1}\|w(T^n)\|=\|w\|_\infty$, proving (iii).
\end{proof}

Since each \(w(T^n)\) is normal, for every \(N\in \mathbb N\),
\[
\begin{aligned}
   a_N(w,T)
   =
   \sup_{n\ge1}\|w(T^n)^N\|
   =
   \left(\sup_{n\ge1}\|w(T^n)\|\right)^N
   =
   \|w\|_\infty^N,
\end{aligned}
\]
hence \(\gamma(w,T)=\|w\|_\infty\).
So, this example shows that a strict norm gap for \(w(T)\) does not force
a strict exponential defect along the power orbit: one may have
\(\|w(T)\|<\|w\|_\infty\) while
\(\gamma(w,T)=\|w\|_\infty\).
Moreover, tangential approach of the spectrum to \(1\) is already
enough to destroy the Ritt property
as the next simple proposition shows.

\begin{proposition}
\label{lem:curve_kills_subsequence}
Let \(T\) be a normal contraction on $\ell^2$ satisfying
\(
        \sigma(T)\cap\T\subset\{1\},
\)
and
 assume that for some $\theta_0\in(0,\pi)$ there exists
$\psi:(0,\theta_0]\to[0,\infty)$ with $\psi(\theta)/\theta\to0$ as $\theta\downarrow0$ such that
\[
\Gamma:=\{(1-\psi(\theta))e^{i\theta}:\ 0<\theta\le\theta_0\}\subset\sigma(T).
\]
Let $(n_k)_{k \ge 1}\subset \mathbb N$ be any strictly increasing sequence.
Then, for every \(\lambda\in\T\setminus\{1\}\), as \(k\to\infty\),
\[
\begin{aligned}
\dist(\lambda,\sigma(T^{n_k}))&\longrightarrow0,\\
\|(\lambda-T^{n_k})^{-1}\|&\longrightarrow\infty.
\end{aligned}
\]
\end{proposition}

\begin{proof}
Fix $\lambda=e^{i\varphi},$ with \(0<\varphi<2\pi\).
For large $k$, $\theta_k:=\varphi/n_k\in(0,\theta_0]$, so $z_k:=(1-\psi(\theta_k))e^{i\theta_k}\in\sigma(T)$.
Then $z_k^{n_k}\in\sigma(T^{n_k})$ and
\[
z_k^{n_k}=(1-\psi(\theta_k))^{n_k}e^{i\varphi}.
\]
As \(k\to\infty\),
\[
        n_k\psi(\theta_k)
        =
        \frac{\varphi}{\theta_k}\psi(\theta_k)
        \longrightarrow0.
\]
Hence
\[
        (1-\psi(\theta_k))^{n_k}\longrightarrow1,
        \qquad
        z_k^{n_k}\longrightarrow e^{i\varphi}=\lambda.
\]
Therefore
\[
        \dist(\lambda,\sigma(T^{n_k}))
        \le
        |\lambda-z_k^{n_k}|
        \longrightarrow0.
\]
Moreover, \(\lambda\in\rho(T^{n_k})\) for every \(k\), since otherwise
the spectral mapping theorem would give some
\(\mu\in\sigma(T)\cap\mathbb T\) such that
\(\mu^{n_k}=\lambda\), which is impossible as
\(\sigma(T)\cap\mathbb T\subset\{1\}\).
Normality yields
\[
        \|(\lambda-T^{n_k})^{-1}\|
        =
        \frac1{\dist(\lambda,\sigma(T^{n_k}))}
        \longrightarrow\infty.
\]
\end{proof}


Next we show
that the estimate in 
Theorem~\ref{thm:one-point-power-Kato}
cannot, in general, be tested
only along a subsequence.

\begin{proposition}
\label{prop:unbounded_quotients_subseq_resolvent}
Let $(n_k)_{k \ge 1}\subset \mathbb N$ be strictly increasing with $\sup_{k\ge 2}\frac{n_k}{n_{k-1}}=\infty$.
Fix $\lambda\in\T\setminus\{1\}$.
Then there exists a normal contraction $T$ on $\ell^2$ with $\sigma(T)\subset\D\cup\{1\}$ such that \(\sigma(T)\not\subset\overline{S_\sigma}\) for every \(\sigma>1\) (hence $T$ is not Ritt), but
\[
\sup_{k\ge1}\|(\lambda -T^{n_k})^{-1}\|<\infty.
\]
\end{proposition}

\begin{proof}
Fix $\lambda=e^{i\varphi}\neq1$ and set $\delta_\lambda:=|\lambda-1|\in(0,2]$.
Choose a strictly increasing sequence $(K(j))_{j \ge 1}\subset \mathbb N$ with $K(j)\ge2$ for every $j$, such that
\[
\frac{n_{K(j)}}{n_{K(j)-1}}\ \ge\ \frac{16j\log2}{\delta_\lambda},\qquad j\ge1.
\]
Define
\[
\varepsilon_j:=\frac{\log2}{n_{K(j)}},\qquad
\theta_j:=j\,\varepsilon_j,\qquad
\mu_j:=(1-\varepsilon_j)e^{i\theta_j}, \qquad j \ge 1,
\]
and let $T$ be diagonal with eigenvalues $\{1,\mu_1,\mu_2,\dots\}$.
By the choice of \((K(j))_{j\ge1}\), as \(j\to\infty\),
\[
        \varepsilon_j\longrightarrow0,
        \qquad
        0<\theta_j
        =
        \frac{j\log2}{n_{K(j)}}
        \le
        \frac{\delta_\lambda}{16\,n_{K(j)-1}}
        \longrightarrow0,
        \qquad
        \mu_j\longrightarrow1.
\]

First, as in the  proof of Proposition \ref{prop:sharpness_abstract_w_exact},
$1-|\mu_j|=\varepsilon_j$ and $|1-\mu_j|\asymp \theta_j$, hence
\[
\frac{|1-\mu_j|}{1-|\mu_j|}
\asymp
\frac{\theta_j}{\varepsilon_j}
=
j
\longrightarrow\infty.
\]
So
\(
   \sigma(T)\not\subset\overline{S_\sigma}\)
for every $\sigma>1,$
and \(T\) is not Ritt.

Next,
fix $k$ and $j$ in $\mathbb N,$ and put $m:=n_k$.  We are going to estimate from below $|\lambda-\mu^{n_k}_j|$.

\emph{Case 1: $k\ge K(j)$.}
Then $m\ge n_{K(j)}$ and
\[
|\mu_j|^m\le |\mu_j|^{n_{K(j)}}\le e^{-\varepsilon_j n_{K(j)}}=e^{-\log2}=\frac12.
\]
Since $|\lambda|=1$, this gives $|\lambda-\mu_j^m|\ge 1-\frac12=\frac12$.

\emph{Case 2: $k\le K(j)-1$.}
Then $m\le n_{K(j)-1}$ and thus
\[
m\varepsilon_j\le n_{K(j)-1}\varepsilon_j=(\log2)\frac{n_{K(j)-1}}{n_{K(j)}}
\le \frac{\delta_\lambda}{16j},
\]
and similarly
\[
|m\theta_j|\le n_{K(j)-1}\,j\,\varepsilon_j
=j(\log2)\frac{n_{K(j)-1}}{n_{K(j)}}
\le \frac{\delta_\lambda}{16}.
\]
Therefore, using that $|(1-\varepsilon)^m-1|\le m\varepsilon$ and $|e^{it}-1|\le |t|$, we obtain
\begin{align*}
|\mu_j^m-1|
=|(1-\varepsilon_j)^m e^{im\theta_j}-1|
\le& |(1-\varepsilon_j)^m-1|+|e^{im\theta_j}-1|
\\
\le& m\varepsilon_j+|m\theta_j|
\le \frac{\delta_\lambda}{8}.
\end{align*}
Therefore
\[
|\lambda-\mu_j^m|
\ge |\lambda-1|-|\mu_j^m-1|
\ge \delta_\lambda-\frac{\delta_\lambda}{8}
=\frac{7\delta_\lambda}{8}.
\]

Combining both cases yields, for all $j,k$,
\[
|\lambda-\mu_j^{n_k}|\ge \min\Bigl\{\frac12,\frac{7|\lambda-1|}{8}\Bigr\}=:d_\lambda>0.
\]
Since also $|\lambda-1|>0$ for the eigenvalue $1$, we get
$\dist(\lambda,\sigma(T^{n_k}))\ge d_\lambda$ for all $k$. Hence, by normality,
\[
\|(\lambda -T^{n_k})^{-1}\|=\frac{1}{\dist(\lambda,\sigma(T^{n_k}))}\le \frac1{d_\lambda},
\qquad k\ge1,
\]
as required.
\end{proof}


Together with Propositions~\ref{prop:sharpness_abstract_w_exact} and
\ref{lem:curve_kills_subsequence}, this illustrates the sharpness
picture: exact boundary saturation may persist along the full power
family, tangentially rich spectra rule out subsequence resolvent control
altogether, yet very sparse subsequences need not detect non-Ritt
behaviour.

\section{Beurling--Kato: necessity. From Hardy--Sobolev calculus back to powers}\label{sec:HS-back-powers}


The preceding Beurling--Kato sufficiency results show that suitable
power
defects force the Ritt property.  We now turn to the opposite
question: if \(T\) is already Ritt, which defects must necessarily occur?
In this section  we explain why
the norm gaps and exponential defects used above are 
pertinent objects in the discrete situation.  The answer uses the Hardy--Sobolev calculus constructed in
Section~\ref{hcalculus}: once \(T\) is Ritt, the bounded \(\HS\)-calculus allows us to pass
from estimates of scalar functions on Stolz domains to uniform estimates
for the families \(w(T^n)\).  
Thus the calculus constructed earlier yields the necessary estimates in the Beurling--Kato theory.

We first record two elementary facts used in the necessity arguments:
stability of Stolz domains under powers and submultiplicativity of the
exponential defect sequence.

Throughout the geometric preliminaries at the beginning of this section we fix $\sigma>1$ and the Stolz domain
$S_\sigma$ given by \eqref{eq:Stolz-Ssigma}.

\begin{lemma}\label{lem:Stolz-union-compact}
For $\sigma>1$ we have
\[
   S_{\sigma,n} := \{z^n : z\in S_\sigma\}\subset S_\sigma, \qquad n\in \N.
\]
\end{lemma}

\begin{proof}
The assertion is immediate for $n=1$.  Let
$z\in S_\sigma$ and $n\ge2$.  Then $z^n\in\D$ and
\[
\frac{|1-z^n|}{1-|z^n|}
=\frac{\left|\sum_{k=0}^{n-1}z^k\right|}{\sum_{k=0}^{n-1}|z|^k}\,
\frac{|1-z|}{1-|z|}
\le \frac{|1-z|}{1-|z|}<\sigma .
\]
Thus $z^n\in S_\sigma$.
\end{proof}

As a consequence, if $w$ is holomorphic in $\D$, continuous on $\overline\D$
and satisfies
\[
   |w(1)| < \|w\|_\infty
      := \sup_{|z|=1}|w(z)|,
\]
then for every $\sigma>1$ we have
\begin{equation}\label{eq:geom-C1}
   \sup_{n\in\N}\sup_{z\in S_\sigma}|w(z^n)|
      \le \sup_{z\in S_\sigma}|w(z)|
      < \|w\|_\infty.
\end{equation}



Throughout the necessity arguments, constants $C(\sigma)$ depend only on $\sigma$,
and constants $C(\sigma,w)$ depend only on $\sigma$ and on $w$.

The condition $|w(1)|<\|w\|_\infty$ in the next theorem is not a normalization.
It means that the boundary maximum of $|w|$ is separated from the distinguished
point $1$, so that comparison arcs can be chosen away from $1$.  Moreover, when
$1\in\sigma(T)$ this separation is forced by any strict ordinary or exponential
defect. See Remark~\ref{rem:w1-from-gamma} below.

For $T$ and $w$ as they occur below, put, for $N\in\N$,
\begin{equation}\label{eq:aN-def}
  a_N(w,T):=\sup_{n\ge1}\,\|w(T^n)^N\|.
\end{equation}

\begin{lemma}
\label{lem:submult-radius}
The sequence $N\mapsto a_N(w,T)$ is submultiplicative:
\[
  a_{N+K}(w,T)\le a_N(w,T)\,a_K(w,T), \qquad N,K\in\N.
\]
Consequently, the limit
\begin{equation}\label{eq:gamma-def}
  \gamma(w,T):=\lim_{N\to\infty} a_N(w,T)^{1/N}
  =\inf_{N\ge1} a_N(w,T)^{1/N}
\end{equation}
exists. 
\end{lemma}

\begin{proof}
For each fixed $n$,
\[
  \|w(T^n)^{N+K}\|\le \|w(T^n)^N\|\,\|w(T^n)^K\|, \qquad N, K \in \mathbb N.
\]
Taking the supremum over $n$ on both sides of the above inequality yields submultiplicativity.
The existence of the limit in \eqref{eq:gamma-def} follows from Fekete's lemma
applied to $\log a_N$. 
\end{proof}

\begin{theorem}\label{thm:HS-disc-necessity-A11}
Let $T\in\L(X)$ be a Ritt operator of Stolz type $\sigma>1$.
Let $w
\in A^{1,1}(\D)$ satisfy 
\begin{equation}\label{eq:w1-strict}
   |w(1)|<\|w\|_\infty.
\end{equation}
Then there exist constants $C>0$ and $\eta\in(0,\|w\|_\infty)$ such that
\begin{equation}\label{eq:KB-exp-decay-A11}
   \sup_{n\ge1}\|w(T^n)^N\|\le C\,\eta^N,
   \qquad N\ge1.
\end{equation}
In particular, $\gamma(w,T)<\|w\|_\infty$.
\end{theorem}

\begin{proof}
Choose $\widetilde\sigma>\sigma$ and set
\[
   \alpha:=\sup_{z\in S_{\widetilde\sigma}}|w(z)|.
\]
By Lemma~\ref{lem:Stolz-union-compact},
\[
   z\in S_{\widetilde\sigma} \ \Longrightarrow\ z^n\in S_{\widetilde\sigma}, \quad n\ge1.
\]
Moreover $\alpha<\|w\|_\infty$. Indeed, 
if
\(\alpha=\|w\|_\infty\), choose a sequence
\(
   (z_j)_{j\ge1}\subset S_{\widetilde\sigma}
\)
such that
\(
   |w(z_j)|\to\|w\|_\infty
\)
as $j \to \infty.$
 By compactness of
$\overline{S_{\widetilde\sigma}}$ we may assume, after passing to a subsequence, that
$z_j\to z_0\in\overline{S_{\widetilde\sigma}}$ as $j\to\infty$. Hence
$|w(z_0)|=\|w\|_\infty$ by continuity. If $z_0\neq1$, then
$z_0\in\D$ and the maximum modulus principle forces $w$ to be constant,
a contradiction. If $z_0=1$, then this contradicts
$|w(1)|<\|w\|_\infty$.

Now fix $N,n\ge1$ and set $F_{N,n}(z):=w(z^n)^N$. Then for
$z\in S_{\widetilde\sigma}$ we have $z^n\in S_{\widetilde\sigma}$, hence
\[
   \sup_{z\in S_{\widetilde\sigma}}|F_{N,n}(z)|\le \alpha^N.
\]

Since $w\in A^{1,1}(\D)$, its derivative $w'$ belongs to the Wiener algebra.
In particular,
\begin{equation}\label{eq:L-def}
   L:=\sup_{|z|\le1}|w'(z)|
   \le 
   \|w'\|_{A^1}<\infty.
\end{equation}

For $z\in\Gamma_{\widetilde\sigma}$ we compute
\[
   F'_{N,n}(z)=N\,w(z^n)^{N-1}\,w'(z^n)\,n z^{n-1}.
\]
Using that $|w(z^n)|\le\alpha$ and by \eqref{eq:L-def}, we obtain 
\[
   |F'_{N,n}(z)|
   \le N\,\alpha^{N-1}\,L\,n|z|^{n-1}.
\]
Therefore Lemma~\ref{lem:kernel-bound-Stolz}, applied with
$\widetilde\sigma$ in place of $\sigma$, yields
\[
   \int_{\Gamma_{\widetilde\sigma}}|F'_{N,n}(z)|\,|dz|
   \le N\,\alpha^{N-1}\,L
      \int_{\Gamma_{\widetilde\sigma}}n|z|^{n-1}\,|dz|
   \le C_{\widetilde\sigma}\,L\,N\,\alpha^{N-1},
\]
and thus
\begin{equation}\label{eq:HS-FNn-A11}
   \|F_{N,n}\|_{\HS(S_{\widetilde\sigma})}
   \le \alpha^N + C_{\widetilde\sigma}L\,N\,\alpha^{N-1}
   \le C_1(\widetilde\sigma,w)\,(N+1)\,\alpha^{N-1}.
\end{equation}

Since $T$ is of Stolz type $\sigma$ and $\widetilde\sigma>\sigma$,
Theorem~\ref{thm:HS-calc-Ritt} yields a bounded \(\HS\)-calculus on
$S_{\widetilde\sigma}$. Hence there exists
$C_{\mathrm{HS}}=C_{\mathrm{HS}}(T,\widetilde\sigma)\ge1$ such that
\[
   \|F_{\HS}(T)\|\le C_{\mathrm{HS}}\,\|F\|_{\HS(S_{\widetilde\sigma})}.
\]
Moreover, for $F_{N,n}(z)=w(z^n)^N$ we have
$F_{N,n,\HS}(T)=w(T^n)^N$. Indeed, $F_{N,n}\in A^1(\D)$, and its Taylor
polynomials converge to $F_{N,n}$ in $A^1(\D)$, hence in
$\HS(S_{\widetilde\sigma})$ by Proposition~\ref{prop:A1-HS}. Since the
$\HS$-calculus agrees with the polynomial calculus, the identity follows.
Hence
\[
   \|w(T^n)^N\|
   =\|F_{N,n,\HS}(T)\|
   \le C_{\mathrm{HS}}\,C_1(\widetilde\sigma,w)\,(N+1)\,\alpha^{N-1}, 
   \qquad n,N\ge1.
\]

Choose any $\eta$ with $\alpha<\eta<\|w\|_\infty$.
Since $\alpha/\eta<1$, there exists $C_2\ge1$ such that
\[
   (N+1)\alpha^{N-1}\le C_2\,\eta^N, \qquad N\ge1.
\]
Taking the supremum over $n\ge1$ gives \eqref{eq:KB-exp-decay-A11} with
$C:=C_{\mathrm{HS}}\,C_1\,C_2$.  The existence of the limit defining
$\gamma(w,T)$ follows from the submultiplicativity argument in
Lemma~\ref{lem:submult-radius}, and the estimate
$\gamma(w,T)\le \eta<\|w\|_\infty$ follows by taking $N$th roots.
\end{proof}



\begin{remark}\label{rem:w1-from-gamma}
Assume that $T$ is power-bounded and $1\in\sigma(T)$,
and let $w\in A^{1}(\D)$. Then $1\in\sigma(T^n)$ for all $n\ge1$, hence by Lemma~\ref{lem:approx-point}
one has
$w(1)\in\sigma\!\bigl(w(T^n)\bigr)$ for all $n\ge1.$
Therefore $r\!\bigl(w(T^n)\bigr)\ge |w(1)|$ and so
\[
\|w(T^n)^N\|^{1/N}\ge r\!\bigl(w(T^n)\bigr)\ge |w(1)|, \qquad n,N\ge1.
\]
Taking supremum over $n\ge1$ and then letting $N\to\infty$ yields
\[
        |w(1)|\le \gamma(w,T).
\]
In particular, the exponential defect
\(
        \gamma(w,T)<\|w\|_\infty
\)
implies
\(
        |w(1)|<\|w\|_\infty.
\)
Similarly, if
\[
        \sup_{n\ge1}\|w(T^n)\|<\|w\|_\infty,
\]
then $|w(1)|<\|w\|_\infty$.  Thus the separation condition at $1$ is not
merely an artefact of the proof.
\end{remark}

\subsection{The discrete Kato criterion}

Using a simple trick, we now show the necessity direction for the analogue of
Kato's theorem in Theorem~\ref{thm:one-point-power-Kato}.  That theorem gives
the forward implication: a uniform one-point resolvent estimate for the powers
$(T^n)_{n \ge 1}$ implies the Ritt property.  We prove the converse statement, and obtain
a somewhat surprising statement about Ritt constants of powers as a byproduct.
The
first step is the next lemma, providing uniformity for resolvent estimates for
$(T^n)_{n \ge 1}$.

We shall use the following standard characterisation of the Ritt condition, discussed in Subsection \ref{poweritt}.  An
operator \(T\in \mathcal L(X)\) is Ritt if and only if
\[
  M_T:=\sup_{n\ge 0}\|T^n\|<\infty
\qquad \text{and} \qquad
  N_T:=\sup_{n\ge 1} n\|T^n(I-T)\|<\infty.
\]
We shall also need the Banach space of bounded vector-valued sequences defined by $\ell^\infty(X):=\{(x_n)_{n \ge 1}\subset X: \sup_{n \ge 1}\|x_n\|<\infty\}.$
\begin{lemma}\label{lem:direct-sum-powers-ritt}
Let $T$ be a Ritt operator on $X$, and define
\[
  S_T=\bigoplus_{m\ge1}T^m
\]
on $\ell^\infty(X)$, that is,
\[
  S_T(x_1,x_2,\ldots)=(Tx_1,T^2x_2,T^3x_3,\ldots), \qquad (x_1,x_2,\ldots) \in \ell^\infty(X).
\]
Then $S_T$ is a Ritt operator. More precisely,
\[
  \sup_{n\ge0}\|S_T^n\|\le M_T
\qquad \text{and} \qquad
  \sup_{n\ge1} n\|S_T^n(I-S_T)\|\le N_T.
\]
\end{lemma}

\begin{proof}
If $n\ge0$, then
\[
  S_T^n=\bigoplus_{m\ge1}T^{mn}.
\]
Hence
\[
  \sup_{n\ge0}\|S_T^n\|
  =\sup_{n\ge0}\sup_{m\ge1}\|T^{mn}\|
  \le M_T .
\]
Moreover,
\[
  S_T^n(I-S_T)=\bigoplus_{m\ge1}T^{mn}(I-T^m).
\]
For $m \ge1$,
\[
  I-T^m=(I-T)(I+T+\cdots+T^{m-1}),
\]
and therefore for all $n \ge 1,$
\[
  T^{mn}(I-T^m)
  =\sum_{j=0}^{m-1}T^{mn+j}(I-T).
\]
Thus
\[
\begin{aligned}
  n\|T^{mn}(I-T^m)\|
  &\le n\sum_{j=0}^{m-1}\|T^{mn+j}(I-T)\|  \\
  &\le N_T\, n\sum_{j=0}^{m-1}\frac1{mn+j}
  \le N_T .
\end{aligned}
\]
Taking the supremum over $m$ gives the required estimate for
$S_T^n(I-S_T)$.
\end{proof}

If $T$ is Ritt, then a simple argument as above shows
that each power of $T$ is Ritt.

 The preceding lemma also gives uniform control of the resolvent Ritt constants of the powers of a Ritt operator.
 For a Ritt operator $T$, define its Ritt constant by 
\[
        R(T):=\sup_{|z|>1}|z-1|\|(z-T)^{-1}\|.
\]
There are other related definitions of Ritt constants in the literature.
We omit their discussion here. 
\begin{theorem}\label{ritt_constants}
Let $T$ be a Ritt operator on $X$. Then
\begin{equation}\label{rm}
  \sup_{m\ge1}R(T^m)
  <\infty.
\end{equation}
\end{theorem}

\begin{proof}
By the preceding lemma, $S_T$ is Ritt on $\ell^\infty(X)$. Applying the
resolvent characterization of Ritt operators to this single operator
$S_T$ gives
\[
  \sup_{|z|>1}|z-1|\,\|(z-S_T)^{-1}\|<\infty .
\]
For $|z|>1$, each $z-T^m$ is invertible, and
\[
  z-S_T=\bigoplus_{m\ge1}(z-T^m).
\]
Thus the inverse is the diagonal operator on $\ell^\infty(X)$ given by
\[
  (z-S_T)^{-1}=\bigoplus_{m\ge1}(z-T^m)^{-1},
\]
and hence
\[
  \|(z-S_T)^{-1}\|
  =\sup_{m\ge1}\|(z-T^m)^{-1}\|.
\]
This proves \eqref{rm}.

\end{proof}

 The next corollary provides a converse to Kato's condition for the Ritt property.
\begin{corollary}\label{ritt_converse}
Let $T$ be a Ritt operator. 
Then for every $\xi\in\mathbb T \setminus\{1\}$,
\[
  \xi\in\rho (T^m),\qquad m\ge1,
\]
and
\begin{equation}\label{rm1}
  \sup_{m\ge1}\|(\xi -T^m)^{-1}\|
  \le \frac{C}{|\xi-1|},
\end{equation}
where the constant $C$ depends only on $T$. In particular, for
every closed arc $I\subset\mathbb T\setminus\{1\}$,
\[
  \sup_{\xi\in I}\sup_{m\ge1}\|(\xi -T^m)^{-1}\|<\infty .
\]
\end{corollary}
\begin{proof}
Fix $\xi\in\mathbb T \setminus\{1\}$. We first check that
$\xi\in\rho (T^m)$ for every $m\ge1$. If
$\xi\in\sigma(T^m)$, then by spectral mapping there exists
$\lambda\in\sigma(T)$ such that
\(
  \lambda^m=\xi .
\)
Since $T$ is power bounded, $\sigma(T)\subset\overline{\mathbb D}$.
Thus $|\lambda|=1$. Since $T$ is Ritt,
\[
  \sigma(T)\cap\mathbb T \subset\{1\},
\]
and so $\lambda=1$. This would give $\xi=1$, a contradiction.
Therefore $\xi\in\rho (T^m)$ for all $m\ge1$.

For $r>1$, the estimate \eqref{rm} gives
\[
  \|(r\xi -T^m)^{-1}\|
  \le \frac{C}{|r\xi-1|},
  \qquad m\ge1.
\]
Letting $r\downarrow1$ and using norm continuity of the resolvent for
each fixed $m$, we obtain
\[
  \|(\xi -T^m)^{-1}\|
  \le \frac{C}{|\xi-1|},
  \qquad m\ge1.
\]
Taking the supremum over $m,$  the estimate \eqref{rm1} follows. The statement
for closed arcs follows immediately.
\end{proof}

The next corollary is now immediate. It is a complete discrete analogue of Kato's
characterisation of holomorphic semigroups mentioned in the introduction.

\begin{corollary}\label{DisKato}
Let $T\in\L(X)$ be power bounded and assume that
\[
        \sigma(T)\cap\T\subset\{1\}.
\]
Then $T$ is a Ritt operator if and only if there exists
$\zeta\in\T\setminus\{1\}$ such that
\begin{equation}\label{kato_ness}
        \sup_{n\ge1}\|(\zeta -T^n)^{-1}\|<\infty.
\end{equation}
If $T$ is Ritt, then this estimate holds for every
$\zeta\in\T\setminus\{1\}$.
\end{corollary}

We finish this subsection with a discrete counterpart of a result due to Pazy for
strongly continuous semigroups on Banach spaces, see \cite[Chapter~2.5, Corollary~5.8]{Pazy}.
It shows that, for contractions
on uniformly convex spaces, the discrete analogue of the classical
Beurling--Kato zero--two condition (see Corollary \ref{bk-example}) is forced by the Ritt property. The proof is based on Kato's criterion proved above.

\begin{theorem}\label{Pazy}
Let $X$ be a uniformly convex Banach space, and let $T$ be a contractive Ritt operator. Then
\begin{equation}\label{Rap}
\sup_{n\ge 1}\|I-T^n\|<2.
\end{equation}
\end{theorem}

\begin{proof}
Suppose, to the contrary, that
\[
        \sup_{n\ge1}\|I-T^n\|=2.
\]
Then there exist  $(n_k)_{k \ge 1}\subset \mathbb N$ and $(x_k)_{k \ge 1}\subset X$ with $\|x_k\|=1$ for all $k$ such that
\begin{equation}\label{1UC}
\lim_{k\to\infty}\|x_k-T^{n_k}x_k\|=2.
\end{equation}
By \eqref{1UC} and uniform convexity of $X$, it follows that
\[
\lim_{k\to\infty}\|x_k+T^{n_k}x_k\|=0.
\]
On the other hand, Corollary~\ref{DisKato} with $\zeta=-1$ yields a constant $K>0$ such that
\[
1=\|(I+T^{n_k})^{-1}(I+T^{n_k})x_k\|\le K\|(I+T^{n_k})x_k\|\to 0,\qquad k\to\infty,
\]
which is impossible. Hence \eqref{Rap} holds.
\end{proof}
Thus the strict ordinary zero--two gap, which is not forced by the Ritt
property in general, is forced for contractive Ritt operators on uniformly
convex spaces.  The next example shows that the uniform convexity assumption
cannot simply be omitted.
\begin{example}\label{ex:projection-non-UC}
The uniform convexity assumption in Theorem~\ref{Pazy} cannot simply be removed. Let
$X=\C^2$ with
\[
        \|(x,y)\|=\max\{|x|,|x+y|\},
\]
and let $T(x,y)=(x,0)$. Then $T^2=T$, hence $T$ is Ritt, and $T$ is contractive since
\[
        \|T(x,y)\|=|x|\le \max\{|x|,|x+y|\}=\|(x,y)\|.
\]
However
\[
        \|I-T\|=2.
\]
Indeed, the upper bound follows from $|y|\le |x|+|x+y|$, and equality is obtained by taking $x=-y/2$. Thus
\[
        \sup_{n\ge1}\|I-T^n\|=2,
\]
although $T$ is a contractive Ritt operator.
\end{example}

\subsection{Defect vs exponential defect: a unified KB picture}

We now assemble the two directions developed so far. Section~\ref{sec:BK-sufficiency}
gave sufficient conditions for the Ritt property in terms of the powers $(T^n)_{n \ge 1}$,
whereas Theorem~\ref{thm:HS-disc-necessity-A11} and Corollary~\ref{DisKato} give the corresponding necessary conditions in terms of the same power quantities. The purpose of the present subsection is
to collect the defect, exponential-defect, and discrete Kato statements, and to separate the
sufficiency statements from the necessity and equivalence statements.

Let $T\in\L(X)$ be power-bounded, with $\|T^n\|\le M$ for $n\ge0$, and
$\sigma(T)\cap \mathbb T \subset \{1\}$.
Fix a function $w\in A^{1,1}(\D),$ and recall the notation
\[
  m_I(w)=\min_{\eta\in I}|w(\eta)|,
\]
for a nonempty closed arc $I\subset\T.$
The number \(m_I(w)\) is the comparison level used in the defect
conditions below.

\subsubsection{Defect assumptions}

The next definition fixes the terminology used throughout the rest of the section.
Both assumptions depend on the same arc $I$ and the same comparison level
$m_I(w)$.

\begin{definition}\label{def:defects}
Let \(I\subset\T\) be a nonempty closed arc.
\begin{enumerate}
\item[(D)] We say that \(w\) has a \emph{defect on \(I\)} if
\[
   \sup_{n\ge1}\|w(T^n)\|<m_I(w).
\]

\item[(ED)] We say that \(w\) has an \emph{exponential defect on \(I\)} if
\[
   \gamma(w,T)<m_I(w).
\]
Equivalently, there exist constants \(C\ge1\) and
\(\eta\in[0,m_I(w))\) such that
\[
   a_N(w,T)\le C\eta^N,\qquad N\in\N.
\]
\end{enumerate}
\end{definition}

An elementary relation between \textup{(D)} and \textup{(ED)} is as follows.
\begin{lemma}\label{lem:D-ED-rel}
Let $I\subset\T$ be a nonempty closed arc. Then the following assertions hold.
\begin{enumerate}
\item[(i)] If $w$ satisfies \textup{(D)} on $I$, then it satisfies \textup{(ED)} on
$I$, and more precisely
\[
   \gamma(w,T)\le \sup_{n\ge1}\|w(T^n)\|<m_I(w).
\]

\item[(ii)] If $w$ satisfies \textup{(ED)} on $I$, then there exists $N\ge1$ such that
\[
   \sup_{n\ge1}\|w(T^n)^N\|<m_I(w)^N.
\]
Equivalently, $W:=w^N$ satisfies \textup{(D)} on the same arc $I$.
\end{enumerate}
\end{lemma}

\begin{proof}
For (i), put
\[
   A:=\sup_{n\ge1}\|w(T^n)\|.
\]
Then
\[
   a_N(w,T)=\sup_{n\ge1}\|w(T^n)^N\|\le A^N,
\]
and hence $\gamma(w,T)\le A<m_I(w)$.

For (ii), by the definition of $\gamma(w,T)$ as
\[
   \gamma(w,T)=\lim_{N\to\infty}a_N(w,T)^{1/N},
\]
we may choose $N\ge1$ such that
\[
   a_N(w,T)^{1/N}<m_I(w).
\]
This is exactly
\[
   \sup_{n\ge1}\|w(T^n)^N\|<m_I(w)^N.
\]
Since $W=w^N$, the last inequality says exactly that $W$ satisfies
\textup{(D)} on $I$.
\end{proof}
Thus, on the sufficiency side, \textup{(ED)} reduces to \textup{(D)}
after replacing \(w\) by a suitable power \(w^N\).  The point of keeping
\textup{(ED)} as a separate condition is the converse direction: the Ritt
property does not force, in general, a strict ordinary defect for the original
symbol \(w\), see Example  \ref{ex:projection-exp-not-ordinary} below.

\subsubsection{Sufficiency: defect and exponential defect both imply Ritt}

The sufficient direction is first stated for \textup{(D)}, since this is the
hypothesis used directly in that argument. The exponential version then follows
from Lemma~\ref{lem:D-ED-rel} by passing to a power of the symbol.

The ordinary defect case has already been established in
Theorem~\ref{thm:power-KB-arc}. In the terminology of
Definition~\ref{def:defects}, it takes the following form.

\begin{corollary}\label{cor:KB-defect-suff}
Let $T\in\L(X)$ be power-bounded and assume that
\[
  \sigma(T)\cap\T\subset\{1\}.
\]
Let $w\in A^{1,1}(\D)$, and let
$I\subset\T\setminus\{1\}$ be a closed arc of positive length.
If $w$ satisfies \textup{(D)} on $I$, then $T$ is a Ritt operator.
\end{corollary}

\begin{proof}
By the definition of \textup{(D)}, the defect hypothesis of
Theorem~\ref{thm:power-KB-arc} is satisfied. Hence $T$ is a Ritt operator.
\end{proof}




\begin{theorem}\label{thm:KB-expdef-suff}
Let $T\in\L(X)$ be power-bounded and assume that $\sigma(T)\cap\T\subset\{1\}$.
Let $w\in A^{1,1}(\D)$ and let $I\subset\T\setminus\{1\}$ be a closed arc of positive length.
If $w$ satisfies \textup{(ED)} on $I$, then $T$ is a Ritt operator.
\end{theorem}

\begin{proof}
By Lemma~\ref{lem:D-ED-rel}(ii), choose $N$ so that
\[
  \sup_{n\ge1}\|w(T^n)^N\|<m_I(w)^N.
\]
Put $W:=w^N$.  Since $A^{1,1}(\D)$ is a Banach algebra,
$W\in A^{1,1}(\D)$.  Moreover, $W(T^n)=w(T^n)^N$ and
$m_I(W)=m_I(w)^N$.  Hence $W$ satisfies \textup{(D)} on $I$.  Corollary~\ref{cor:KB-defect-suff}
applied to $W$ therefore implies that $T$ is Ritt.
\end{proof}

\subsubsection{Necessity: Ritt implies exponential defect}

The converse direction has a different form. 
 The estimate
obtained from the Hardy--Sobolev calculus is instead
\[
        a_N(w,T)\le C\eta^N,\qquad N\in\mathbb N,
\]
with \(\eta<\|w\|_\infty\).  This gives the strict exponential defect.

\begin{theorem}\label{thm:Ritt-necessity-exp}
Let $T$ be a Ritt operator and let $w\in A^{1,1}(\D)$ be such that 
$|w(1)|<\|w\|_\infty$.
Then there exist constants $C\ge1$ and $\eta\in(0,\|w\|_\infty)$ such that
\[
   a_N(w,T)=\sup_{n\ge1}\|w(T^n)^N\| \le C\,\eta^N,\qquad N\in\N.
\]
In particular, for every nonempty closed arc $I\subset\T$ with $m_I(w)>\eta$, the
function $w$ satisfies \textup{(ED)} on $I$.
\end{theorem}

\begin{proof}
Choose $\sigma>1$ such that $T$ is of Stolz type $\sigma$. Applying
Theorem~\ref{thm:HS-disc-necessity-A11} to $T$ and $w$, we obtain constants
$C\ge1$ and $\eta\in(0,\|w\|_\infty)$ such that
\[
   a_N(w,T)=\sup_{n\ge1}\|w(T^n)^N\|\le C\,\eta^N,
   \qquad N\in\N.
\]
Since $(a_N(w,T))_{N\ge1}$ is submultiplicative by
Lemma~\ref{lem:submult-radius}, we obtain
\[
\gamma(w,T)=\lim_{N\to\infty} a_N(w,T)^{1/N}\le \eta<\|w\|_\infty.
\]
Thus, if $m_I(w)>\eta$, then $\gamma(w,T)<m_I(w)$, which is exactly
\textup{(ED)} on $I$.
\end{proof}
\begin{example}\label{ex:projection-exp-not-ordinary}
A strict ordinary defect for the symbol \(w(z)=1-z\) is not necessary
for the Ritt property without
additional norm assumptions.
  Let $H=\C^2$ with the Euclidean norm and, for
$a>0$, put
\[
        P_a=\begin{pmatrix}1&a\\0&0\end{pmatrix}.
\]
Then $P_a^2=P_a$, so $P_a$ is power bounded and has the Ritt property. On the other hand,
\[
        I-P_a=\begin{pmatrix}0&-a\\0&1\end{pmatrix},
        \qquad
        \|I-P_a\|=\sqrt{1+a^2}.
\]
Thus for the symbol \(w(z)=1-z\) one has
\[
        \sup_{n \ge 1}\|w(P_a^n)\|= \sup_{n\ge1}\|I-P_a^n\|=\sqrt{1+a^2},
\]
which can be arbitrarily large.  In particular, $\|w(P_a^n)\|>2$ for $a >\sqrt 3.$

For the same symbol, however,
\(w(P_a)=I-P_a\) is again a non-zero projection. Hence
\[
        a_N(w,P_a)=\|I-P_a\|,\qquad N\ge1,
\]
and therefore
\[
        \gamma(w,P_a)=1<2=\|w\|_\infty .
\]
Thus the converse cannot, in general, be formulated with the ordinary defect
for the symbol \(w\).
\end{example}

\subsection{Alternative formulations using $\|w\|_\infty$}

We record two reformulations in which the arc is not part of the assumptions.
It is selected afterwards, near a point where \(|w|\) attains its boundary
maximum \(\|w\|_\infty\).



\begin{theorem}
\label{thm:KB-defect-supnorm}
Let $T\in\L(X)$ be power-bounded and assume
\[
  \sigma(T)\cap\T\subset\{1\}.
\]
Let $w\in A^{1,1}(\D)$ be such that
\[
  \sup_{n\ge1}\|w(T^n)\|<\|w\|_\infty.
\]
Then $T$ is a Ritt operator.
\end{theorem}

\begin{proof}
Put
\[
        \beta:=\sup_{n\ge1}\|w(T^n)\|.
\]
There is a point $\eta_0\in\T$ with $|w(\eta_0)|=\|w\|_\infty$. 
By continuity, there exists a closed arc
$I\subset\T\setminus\{1\}$ of positive length such that $m_I(w)>\beta$.
Hence $w$ satisfies \textup{(D)} on $I$, and Corollary~\ref{cor:KB-defect-suff}
yields that $T$ is Ritt.
\end{proof}

\begin{theorem}
\label{thm:KB-expdef-supnorm}
Let $T\in\L(X)$ be power-bounded and assume
\[
  \sigma(T)\cap\T\subset\{1\}.
\]
Let $w\in A^{1,1}(\D)$ satisfy 
\[
  \gamma(w,T)<\|w\|_\infty.
\]
Then $T$ is a Ritt operator.
\end{theorem}

\begin{proof}
Choose $\beta$ with
\[
        \gamma(w,T)<\beta<\|w\|_\infty .
\]
As in the previous proof, 
there exists a closed arc
$I\subset\T\setminus\{1\}$ of positive length such that $m_I(w)>\beta$.
Thus $w$ satisfies \textup{(ED)} on $I$. Theorem~\ref{thm:KB-expdef-suff}
therefore implies that $T$ is Ritt.
\end{proof}

\begin{corollary}
\label{cor:Ritt-iff-expdef-supnorm}
Let $T\in\L(X)$ be power-bounded and assume
\[
  \sigma(T)\cap\T\subset\{1\}.
\]
Let $w\in A^{1,1}(\D)$ satisfy $|w(1)|<\|w\|_\infty$.
Then the following are equivalent:
\begin{enumerate}
\item[(i)] $T$ is a Ritt operator.
\item[(ii)] $\gamma(w,T)<\|w\|_\infty$.
\item[(iii)] There exist constants $C\ge1$ and $\eta\in(0,\|w\|_\infty)$ such that
\[
   \sup_{n\ge1}\|w(T^n)^N\| \le C\,\eta^N, \qquad N\in\N.
\]
\end{enumerate}
\end{corollary}

\begin{proof}
The equivalence of (ii) and (iii) follows from Lemma~\ref{lem:submult-radius} applied
to the sequence $a_N(w,T)$. The implication (ii)$\Rightarrow$(i) is exactly
Theorem~\ref{thm:KB-expdef-supnorm}.
Finally, if \(T\) is Ritt, choose \(\sigma>1\) such that \(T\) is of Stolz type
\(\sigma\).
By Theorem~\ref{thm:HS-disc-necessity-A11}, applied with this
Stolz-type parameter, there exist constants $C\ge1$ and
$\eta\in(0,\|w\|_\infty)$ such that
\[
  a_N(w,T)=\sup_{n\ge1}\|w(T^n)^N\|\le C\eta^N, \qquad N\in\N,
\]
which is (iii).
\end{proof}

The preceding results complete the passage from the Hardy--Sobolev calculus
back to the power quantities used in the Beurling--Kato argument. Under the
stated spectral assumptions, the Ritt property is equivalent to
the strict inequality
\[
        \gamma(w,T)<\|w\|_\infty
\]
for each non-constant \(w\in A^{1,1}(\mathbb D)\) satisfying
\(|w(1)|<\|w\|_\infty\).  Moreover, the strict inequality for one such function $w$ implies the Ritt property. 

The condition is invariant under equivalent
renormings of \(X\), since equivalent norms change \(a_N(w,T)\) by at most
a fixed multiplicative factor, which disappears after taking the \(N\)-th
root.

\section{Convex combinations preserve the Ritt property}\label{S8x}

The equivalence between the Ritt property and exponential defects also yields
a permanence property for convex combinations of powers of a Ritt operator.
The following statement was one of the main results of \cite{GomTom-Indiana}.
The proof there proceeds by direct resolvent estimates for the convex
combination and is considerably more involved. The proof below is simpler,
but it does not preserve quantitative control of the Stolz type of the
convex combination.

\begin{theorem}\label{thm:convex-combinations-Ritt}
Let $T$ be a Ritt operator, and let
\[
F(z)=\sum_{k=0}^\infty \alpha_k z^k,\qquad \alpha_k\ge 0,\qquad \sum_{k=0}^\infty \alpha_k=1.
\]
Then $F(T)$ is a Ritt operator.
\end{theorem}

\begin{proof}

If $F\equiv1$, then $F(T)=I$, and the conclusion is immediate. Hence we may
assume that $F$ is non-constant.

Since $T$ is Ritt, we have
\begin{equation}\label{eq:convex-ritt-powerbd}
\|T^k\|\le M,\qquad k\ge 0,
\end{equation}
and
\[
\sigma(T)\subset \overline{\D},\qquad \sigma(T)\cap\T\subset\{1\}.
\]

We first show that $F(T)$ is power bounded. For each $n\in\mathbb N$ write
\[
F(z)^n=\sum_{k=0}^\infty \beta_{n,k}z^k,
\]
where $(\beta_{n,k})_{k\ge0}$ is the $n$-fold convolution of $(\alpha_k)_{k\ge0}$. Then
\[
\beta_{n,k}\ge0,\qquad \sum_{k=0}^\infty \beta_{n,k}=1.
\]
Since $F(T)=\sum_{k=0}^\infty \alpha_kT^k$ with absolute convergence in the operator norm, the same holds for every power, and
\[
F(T)^n=\sum_{k=0}^\infty \beta_{n,k}T^k.
\]
Therefore
\begin{equation}\label{eq:convex-ritt-powerbd-F}
\|F(T)^n\|\le \sum_{k=0}^\infty \beta_{n,k}\|T^k\|
\le M\sum_{k=0}^\infty \beta_{n,k}=M,
\qquad n\in\mathbb N.
\end{equation}
By Lemma~\ref{lem:approx-point},
\[
\sigma(F(T))\subset F(\sigma(T))\subset \overline{\D}.
\]
Moreover, if $\lambda\in\sigma(F(T))\cap\T$, then $\lambda=F(\mu)$ for some
$\mu\in\sigma(T)\subset\overline{\D}$. Since $F$ is a convex combination of the monomials $z^k$,
we have $|F(z)|<1$ whenever $|z|<1$. Hence $|\lambda|=1$ forces $|\mu|=1$, and therefore
$\mu\in\sigma(T)\cap\T\subset\{1\}$. Consequently $\mu=1$ and
$\lambda=F(1)=1$. Thus
\begin{equation}\label{eq:convex-ritt-spectrum-F}
\sigma(F(T))\subset \overline{\D},\qquad \sigma(F(T))\cap\T\subset\{1\}.
\end{equation}

To prove that $F(T)$ is Ritt, it remains to verify an exponential defect estimate for
$1-z$. Since $T$ is Ritt, Theorem~\ref{thm:Ritt-iff-HS} yields a parameter $\tau_1>1$ and a bounded \(\HS\)-calculus on $S_{\tau_1}$ for~$T$. 
Let $C_{\tau_1}$ be the corresponding constant from Lemma~\ref{lem:kernel-bound-Stolz} and set 
\[
\alpha:=\sup_{z\in S_{\tau_1}}|1-z|.
\]
Then Lemma~\ref{lem:geometry}\textup{(ii)} gives
\[
        \alpha
        =
        \frac{2\tau_1}{\tau_1+1}
        <2.
\]
Fix $\eta\in(\alpha,2)$. Since $\alpha<\eta$, there exists a constant $C_\eta\ge1$ such that
\begin{equation}\label{eq:convex-ritt-alpha-bound}
\alpha^N+C_{\tau_1}N\alpha^{N-1}\le C_\eta\eta^N,
\qquad N\in\mathbb N.
\end{equation}

For integers $k_1,\dots,k_N\ge0$, define
\[
W_{k_1,\dots,k_N}(z):=\prod_{j=1}^N(1-z^{k_j}),\qquad z\in\D.
\]
If \(k_j=0\) for some \(j\), then
\(W_{k_1,\dots,k_N}\equiv0\), and all the required estimates are
immediate.  We may therefore assume that \(k_j\ge1\) for every \(j\).

By Lemma~\ref{lem:Stolz-union-compact}, each map $z\mapsto z^{k_j}$ leaves $S_{\tau_1}$ invariant, hence
\[
\sup_{z\in S_{\tau_1}}|W_{k_1,\dots,k_N}(z)|\le \alpha^N.
\]
Moreover, the product rule gives
\[
W'_{k_1,\dots,k_N}(z)
=- \sum_{
1\le j\le N
} k_j z^{k_j-1}\prod_{\ell\ne j}(1-z^{k_\ell}),
\]
so that on $\Gamma_{\tau_1}$,
\[
|W'_{k_1,\dots,k_N}(z)|\le \alpha^{N-1}\sum_{
1\le j\le N
} k_j|z|^{k_j-1}.
\]
Using Lemma~\ref{lem:kernel-bound-Stolz}, we obtain
\[
\int_{\Gamma_{\tau_1}}|W'_{k_1,\dots,k_N}(z)|\,|dz|
\le \alpha^{N-1}\sum_{
1\le j\le N
}\int_{\Gamma_{\tau_1}}k_j|z|^{k_j-1}|dz|
\le C_{\tau_1}N\alpha^{N-1}.
\]
Hence
\[
\|W_{k_1,\dots,k_N}\|_{\HS(S_{\tau_1})}
\le \alpha^N+C_{\tau_1}N\alpha^{N-1}
\le C_\eta\eta^N.
\]
By the bounded \(\HS\)-calculus for $T$, there is a constant $K_{\tau_1}$ such that
\begin{equation}\label{eq:convex-ritt-product-bound}
\|W_{k_1,\dots,k_N}(T)\|\le K_{\tau_1}C_\eta\eta^N,
\qquad k_1,\dots,k_N\ge0.
\end{equation}
Since the \(\HS\)-calculus is multiplicative and compatible with polynomials,
\[
W_{k_1,\dots,k_N}(T)=\prod_{j=1}^N(I-T^{k_j}).
\]

Now fix $n\in\mathbb N$. As above,
\[
I-F(T)^n=\sum_{k=0}^\infty \beta_{n,k}(I-T^k).
\]
Therefore
\[
(I-F(T)^n)^N
=\sum_{k_1,\dots,k_N\ge0}\beta_{n,k_1}\cdots\beta_{n,k_N}
\prod_{j=1}^N(I-T^{k_j}),
\]
where the series converges absolutely in the operator norm by \eqref{eq:convex-ritt-product-bound}. Using \eqref{eq:convex-ritt-product-bound} and
\[
\sum_{k_1,\dots,k_N\ge0}\beta_{n,k_1}\cdots\beta_{n,k_N}=
\Bigl(\sum_{k=0}^\infty \beta_{n,k}\Bigr)^N=1,
\]
we obtain
\[
\|(I-F(T)^n)^N\|
\le K_{\tau_1}C_\eta\eta^N,
\qquad n,N\in\mathbb N.
\]
Taking the supremum over $n\ge1$ and then $N$th roots gives
\[
\gamma(1-z,F(T))\le \eta<2.
\]
By Theorem~\ref{thm:KB-expdef-supnorm}, applied to the power-bounded operator $F(T)$ and the symbol $w(z)=1-z$, it follows from \eqref{eq:convex-ritt-spectrum-F} that $F(T)$ is a Ritt operator.
\end{proof}

\section{A positive domination result for Ritt operators}\label{sec:positive-domination-ritt}

The preceding sections established a discrete Kato criterion for Ritt operators in terms of uniform one-point resolvent bounds for the powers. We finish with a positive-domination application of this criterion. The order-theoretic input is a resolvent comparison for positive operators, in the form recorded below. Applied to the pairs $S^n\le T^n$, it transfers the resolvent bounds required by the discrete Kato criterion from $T$ to $S$.

Throughout this section \(X\) is a Banach lattice.  For the background on
Banach lattices and positive operators we refer, for instance, to
\cite{Meyer}.
For $a>0$, $b>0$ and $\zeta\in\T$, define the classes of bounded positive operators
\[
\mathcal K_\zeta(X):=\{T\in\L(X): r(T)=1,\ \zeta\in\rho(T),\ T\ge0\},
\]
and
\[
\mathcal K_\zeta(X,a,b):=
\{T\in\mathcal K_\zeta(X):\sup_{k\ge1}\|T^k\|\le a,
\ \|(\zeta -T)^{-1}\|\le b\},
\]
where $r(T)$ stands for the spectral radius of $T.$
Following \cite[Lemma~4]{Gluck}, we record the following consequence of \cite[Theorem~1.4]{Rabiger}.

\begin{lemma}\label{lem:positive-dominated-resolvent}
There exists a constant $C_\zeta(a,b)>0$ such that, whenever
\[
   0\le S\le T,
   \qquad
   T\in\mathcal K_\zeta(X,a,b),
\]
one has $\zeta\in\rho(S)$ and
\begin{equation}\label{eq:positive-dominated-resolvent-bound}
   \|(\zeta -S)^{-1}\|\le C_\zeta(a,b).
\end{equation}
\end{lemma}
\begin{proof}
The assertion \(\zeta\in\rho(S)\) follows from \cite[Theorem~1.4]{Rabiger}.
It remains to obtain a uniform resolvent bound. It is enough to prove that
there is a constant \(C_\zeta(a,b)>0\) such that
\[
        \|x\|\le C_\zeta(a,b)\|(\zeta -S)x\|,
        \qquad x\in X,
\]
for all \(S,T\) satisfying \(0\le S\le T\) and \(T\in\mathcal K_\zeta(X,a,b)\).

Suppose that no such constant exists. 
Then there exist sequences
\[
   (T_n)_{n\ge1},\ (S_n)_{n\ge1}\subset\mathcal L(X),
   \qquad
   (x_n)_{n\ge1}\subset X,
\]
such that
\[
   T_n\in\mathcal K_\zeta(X,a,b),\qquad
   0\le S_n\le T_n,\qquad
   \|x_n\|=1,
\]
and
\[
       \lim_{n \to \infty}\|(\zeta -S_n)x_n\|=0 .
\]
Consider the Banach lattice
\[
        \mathcal X:=\ell^\infty(X)
\]
and the coordinatewise operators
\[
        \mathcal T(y_n)_{n\ge1}:=(T_ny_n)_{n\ge1},
        \qquad
        \mathcal S(y_n)_{n\ge1}:=(S_ny_n)_{n\ge1}.
\]
Then
\[
        \mathcal T\in\mathcal K_\zeta(\mathcal X,a,b),
        \qquad
        0\le \mathcal S\le\mathcal T .
\]
By \cite[Theorem~1.4]{Rabiger}, \(\zeta\in\rho(\mathcal S)\). Hence there is
\(c>0\) such that
\[
        c\|y\|_{\mathcal X}
        \le
        \|(\zeta -\mathcal S)y\|_{\mathcal X},
        \qquad y\in\mathcal X .
\]
For each \(n\), apply this estimate to the vector
\[
        y^{(n)}=(0,\ldots,0,x_n,0,\ldots)\in\mathcal X,
\]
where \(x_n\) is placed in the \(n\)-th coordinate. Then
\[
        c
        \le
        \|(\zeta -\mathcal S)y^{(n)}\|_{\mathcal X}
        =
        \|(\zeta -S_n)x_n\|,
\]
which contradicts \(\|(\zeta -S_n)x_n\|\to0\). Thus the uniform lower bound
holds, and therefore
\[
        \|(\zeta -S)^{-1}\|\le C_\zeta(a,b).
\]
\end{proof}

The next result is a discrete version of \cite[Theorem 1]{Gluck}.
\begin{corollary}\label{cor:positive-domination-ritt}
Let $X$ be a Banach lattice. If $T$ is a positive Ritt operator and $S$ satisfies
\[
   0\le S\le T,
\]
then $S$ is a Ritt operator.
\end{corollary}

\begin{proof}
Since $T$ is 
Ritt, it is power bounded. From $0\le S\le T$ 
we get, by induction,
\[
   0\le S^n\le T^n,
   \qquad n\ge1.
\]
Since $X$ is a Banach lattice, \(0\le A\le B\) implies
\(\|A\|\le\|B\|\). Hence, from \(0\le S^n\le T^n\), we obtain
\[
        \|S^n\|\le \|T^n\|,\qquad n\ge1.
\]
Therefore
\[
        r(S)=\lim_{n\to\infty}\|S^n\|^{1/n}
        \le
        \lim_{n\to\infty}\|T^n\|^{1/n}
        =r(T).
\]

If $r(T)<1$, then $r(S)\le r(T)<1$, and $S$ is Ritt. We may therefore assume that
$r(T)=1$. Fix $\zeta\in\T\setminus\{1\}$. Since $T$ is Ritt, Corollary~\ref{DisKato} gives a constant
$b_\zeta>0$ such that
\[
   \sup_{n\ge1}\|(\zeta -T^n)^{-1}\|\le b_\zeta .
\]
Let
\[
   a:=\sup_{n\ge1}\|T^n\|<\infty .
\]
For every $n\ge1$, the operator $T^n$ is positive, satisfies
$r(T^n)=1$, and belongs to $\mathcal K_\zeta(X,a,b_\zeta)$. 
Since
\(0\le S^n\le T^n\), Lemma~\ref{lem:positive-dominated-resolvent} gives
\[
   \zeta\in\rho(S^n),
   \qquad
   \|(\zeta-S^n)^{-1}\|
   \le C_\zeta(a,b_\zeta),
   \qquad n\ge1.
\]
Thus
\[
   \sup_{n\ge1}\|(\zeta -S^n)^{-1}\|<\infty .
\]
Since \(\zeta\in\mathbb T\setminus\{1\}\) was arbitrary, the case
\(n=1\) also shows that
\[
   \sigma(S)\cap\mathbb T\subset\{1\}.
\]
Corollary~\ref{DisKato} now implies that $S$ is Ritt.
\end{proof}

\begin{remark}
The assertion of Corollary~\ref{cor:positive-domination-ritt} in the case $T=I$ is classical.  It can also be deduced by combining \cite[Theorem~1.5]{Dungey} with \cite[Theorem~1]{Gluck}.
\end{remark}

\end{document}